\documentclass[11pt,reqno,amscd,verbatim]{amsart}
\usepackage[applemac]{inputenc}
\usepackage{amsmath}
\usepackage{amsthm}
\usepackage{amscd}
\usepackage{amsfonts}
\usepackage{amssymb}
\usepackage{mathrsfs}
\usepackage{geometry}
\usepackage{latexsym}
\usepackage{enumerate}
\usepackage{mathtools}
\usepackage{xcolor}
\usepackage{eucal}
\newgeometry{top=2.2cm,bottom=1.78cm,left=2cm,right=2cm}
\usepackage{tikz}
\usepackage{graphicx}
\usepackage{multicol}
\usepackage[enableskew,vcentermath]{youngtab}

\newtheorem{thm}{Theorem}[section]%[chapter] theorem number will %continue
\newtheorem{lem}[thm]{Lemma}
\newtheorem{cor}[thm]{Corollary}
\newtheorem{prop}[thm]{Proposition}
\theoremstyle{definition}
\newtheorem{example}[thm]{Example}
  
\newtheorem{defn}[thm]{Definition}

\newtheorem{rem}[thm]{Remark}
\newtheorem{rems}[thm]{Remarks}

\numberwithin{equation}{thm}
\def\sA{{\mathcal A}}
\def\sB{{\mathcal B}}

\def\sD{{\mathcal D}}
\def\sE{{\mathcal E}}
\def\sF{{\mathcal F}} 

\def\sH{{\mathcal H}}

\def\sO{{\mathcal O}}
\def\sP{{\mathcal P}}

\def\sR{{\mathcal R}}
\def\sS{{\mathcal S}}

\def\wW{{{W}}}
\def\sX{{\mathcal X}}

\def\sZ{{\mathcal Z}}
                 
\def\dW{{\check{W}}}\def\dS{{\check{S}}}

\def\dl{{\check{\ell}}}

\def\bfc{{\bold c}}

\def\bfr{{\bold r}}

\def\fkd{{\frak d}}
\def\fkf{{\frak f}}

\def\BC{{\Bbb C}}

\def\cBA{{\check{\mathbb A}}}
\def\cBB{{\check{\mathbb B}}}

\def\fS{{\frak S}}

\def\partial{\delta}

\def\al{\alpha} 
\def\be{\beta} 
\def\de{\delta}
\def\ep{{\epsilon}}

\def\ga{\gamma}

\def\la{\lambda}
\def\ka{\kappa} 
\def\th{\theta} 
\def\ro{\rho} 
\def\La{\Lambda}

\def\Om{\Omega}
\def\si{\sigma}

\def\oW{{\bar W}}

\def\wtmu{{\widetilde\mu}}

\def\End{\text{\rm End}}

\def\dim{\text{\rm dim\,}}

\def\det{{\text{\rm\,det\,}}}

\def\ggi#1{\left[\!\left[#1\right]\!\right]}
\def\lr#1{\langle#1\rangle}

\def\up{\upsilon}
\def\vsg{\varsigma}

\def\chW{\check W}

\def\wla{{\widehat\lambda}}
\def\wmu{{\widehat\mu}}
\def\wnu{{\widehat\nu}}
\def\wga{{\widehat\ga}}

\def\GL{{\text{\rm GL}}}

\def\O{{\text{\rm O}}}
\def\SO{{\text{\rm SO}}}
\def\WB{W^{\textsc{b}}}
\def\WD{W^{\textsc{d}}}

\def\scb{\textsc{b}}
\def\scd{\textsc{d}}
\def\bsq{{\boldsymbol q}}
\def\bbb{\overset{\circ\!\circ\!\circ}\Xi}
\def\hhh{\overset{\bullet}\Xi}
\def\hhhd{{\overset{\bullet}\Xi\centerdot}}
\def\bbh{\overset{\circ\!\circ\!\bullet}\Xi}
\def\hbb{\overset{\bullet\!\circ\!\circ}\Xi}
\def\hbh{\overset{\bullet\!\circ\!\bullet}\Xi}
\def\sbs{\overset{\star\circ*}\Xi}

\def\cW{{\check W}}
\def\cG{{\check G}}
\def\cS{{\check S}}
\def\cR{{\check R}}
\def\tmu{\widetilde\mu}
\def\tla{\widetilde\la}
\def\tillam{\widetilde\la}
\def\Stab{\text{\rm Stab}}
\def\fkm{{\mathfrak m}}
\def\fkp{{\mathfrak p}}
\def\sfm{{\mathsf m}}
\def\lamz{{\lambda^\bullet}}
\def\labt{{\lambda^\bullet}}
\def\cx{{\check x}}
\def\csH{{\check\sH}}
\def\op{{\text{\rm op}}}

\def\fkF{\mathfrak F}
\def\sfJ{{\mathsf J}}
\def\oo{{\bullet}}
\def\td{{\widehat d}}
\def\cH{\check H}
\newcommand{\Lambdad}{\Lambda^\scd(n,r)}
\def\csD{\check\sD}
\newcommand{\sgnura}{\text{\rm sgn}(\urcm)}
\newcommand{\ura}{A^{\llcorner}}
\newcommand{\ur}[1]{{#1}^{\text{\normalsize $\llcorner$}}}
\newcommand{\pap}[1]{^{+}\hspace{-0.8mm} {#1}\hspace{-0.5mm} ^{+}}
\newcommand{\map}[1]{^{-}\hspace{-1.5mm} {#1}\hspace{-0.5mm} ^{+}}
\newcommand{\pam}[1]{^{+}\hspace{-0.5mm} {#1}\hspace{-0.5mm} ^{-}}
\newcommand{\mam}[1]{^{-}\hspace{-1.3mm} {#1}\hspace{-0.5mm} ^{-}}

\newcommand{\paep}[1]{^{+}\hspace{-0.8mm} {#1} ^{\ep}}
\newcommand{\pamep}[1]{^{+}\hspace{-0.8mm} {#1} ^{-\ep}}
\newcommand{\maep}[1]{^{-}\hspace{-0.8mm} {#1}^{\ep}}
\newcommand{\mamep}[1]{^{-}\hspace{-1.2mm} {#1}^{-\ep}}

\newcommand{\papsgn}{^{+}\hspace{-0.5mm}{(\apl)}^{\epsilon}}

\newcommand{\mapmsgn}{^{-}\hspace{-0.5mm} {(\apl)}\hspace{-.5mm}^{-\epsilon}}

\newcommand{\apl}{{}_h\!A_{p}}
\newcommand{\apu}{{}_h\!A_{\overline{p}}}
\def\dnr{\csD(n,r)}
\newcommand{\xid}{\check\Xi(n,r)}

\newcommand{\sj}{S^{\jmath}(n,r)}

\newcommand{\lamp}{\lambda^+}
\newcommand{\lamm}{\lambda^-}
\newcommand{\mup}{\mu^+}
\newcommand{\mum}{\mu^-}
\newcommand{\muz}{\mu^\bullet}
\def\sfF{\mathsf F}
\def\urcm{A^{\text{\normalsize $\llcorner$}}}
\def\ulcm{A^{\text{\normalsize $\lrcorner$}}}
\def\llcm{A_{\text{\normalsize$\urcorner$}}}
\def\lrcm{A_{\text{\normalsize$\ulcorner$}}}
\def\ro{\text{\rm ro}}\def\co{\text{\rm co}}
\def\sfFs{\sfF^{s}}
\def\csO{\check\sO}
\def\Gr{\text{\rm iGr}}
\newcommand{\e}[1]{e_{#1}}
\newcommand{\gba}{g_{h,A,p}}
\newcommand{\gca}{g'_{_{h,  A,\overline{p}}}}

\newcommand{\epsiaepsi}[3]{^{#2}\hspace{-1mm} {#1}\hspace{-0.2mm} ^{#3}}

\newcommand{\czero}{\dot C}
\newcommand{\bzero}{\dot B}
\newcommand{\azero}{\dot A}
\newcommand{\apzero}{ _h\!{\dot A}_{p}}

\def\rw{\text{\rm rw}}
\def\cw{\text{\rm cw}}
\def\diag{\text{\rm diag}}

\begin{document}
\title[The $q$-Schur algebras in type $D$]%
{The $q$-Schur algebras in type $D$, I:\\
Fundamental multiplication formulas}% and their regular representations}
\author{Jie Du$^*$}

\address{School of Mathematics and Statistics, University of New South Wales,
Sydney 2052, Australia}
%\address{{\it Home page}: {\tt http://www.maths.unsw.edu.au/$\sim$jied}} 
\email{j.du@@unsw.edu.au}
\author{Yiqiang Li$^{\star}$}

\address{Department of Mathematics, University at Buffalo, SUNY, Buffalo, New York 14260}
\email{yiqiang@buffalo.edu}
\author{Zhaozhao Zhao}
\address{School of Mathematical Sciences, Tongji University, Shanghai 200092, China}
\email{1810878@tongji.edu.cn}

\date{\today}

\thanks{$^*$Supported partially by the UNSW Science FRG (J. Du) and ARC DP220102861 (Daniel Chan)}
\thanks{$^\star$Supported partially by NSF grant DMS1801915 and Simons Travel Support for Mathematicians.}
%\thanks{
%Y. Li is grateful to Chun-Ju Lai for pointing out a gap in  the parametrization of $\SO$-orbits in~\cite{FL} and for proposing its correction.
%The authors
%would like to thank the University of 
%New South Wales and the University at Buffalo for
%their hospitality during the writing of the paper. The third-named author thanks the China Scholarship Council for financial support.} 

\begin{abstract} By embedding the Hecke algebra $\cH_\bsq$ of type $D$ into the Hecke algebra $H_{\bsq,1}$ of type $B$ with unequal parameters $(\bsq,1)$, the $\bsq$-Schur algebras $S^\kappa_\bsq(n,r)$ of type $D$ is naturally defined as the endomorphism algebra of the tensor space with the $\cH_\bsq$-action restricted from the $H_{\bsq,1}$-action that defines the $(\bsq,1)$-Schur algebra $S^\jmath_{\bsq,1}(n,r)$ of type $B$. 
We investigate the algebras $S^\jmath_{\bsq,1}(n,r)$ and $S^\kappa_\bsq(n,r)$ both algebraically and geometrically and describe their natural bases, dimension formulas and weight idempotents. Most importantly, we use the geometrically derived two sets of the fundamental multiplication formulas in $S^\jmath_{\bsq,1}(n,r)$ to derive multi-sets (9 sets in total!) of the fundamental multiplication formulas in $S^\kappa_\bsq(n,r)$.
% and derive the multiplication formulas of the natural bases by some generators.
%In this way, we obtain a construction of the matrix representations of the regular modules of these algebras.
\end{abstract}

\maketitle

\tableofcontents

\section{Introduction}
The $q$-Schur algebras of type $A$ originated from two sources. One originates from  quantum group theory and the other from  representation theory of finite general linear groups. In \cite{Ji}, Jimbo used $q$-Schur algebras to describe the (quantum) Schur--Weyl duality, while in \cite{DJ} Dipper--James used them to investigate representations of finite general linear groups in non-defining characteristic. Thus, $q$-Schur algebras play a central role in connecting representations of quantum linear groups with those of finite general linear groups. 

Although one naturally expected thirty years ago the existence of such a connection for other finite classical groups, the Schur--Weyl--Brauer duality and its quantum version for classical types other than $A$ cannot be used to establish such a connection. This is because the Brauer algebras and their quantum analogs, the so-called BMW algebras (see \cite[Th.~10.2.5]{CP} and the references therein), have a very different representation theory from those of Hecke algebras of type $B/C/D$. Thus, 
we hypothesized the existence of a `Schur-Weyl-Hecke duality'  for other classical types. This puzzle has finally been solved  after thirty years!

In their study of quantum symmetric pairs and their canonical basis theory, Bao--Wang \cite{BW} discovered a new Schur--Weyl duality which is called here the Schur--Weyl--Hecke duality. The Schur--Weyl--Hecke duality of type $B/C$ involves the twin $i$-quantum groups ${\bf U}^\jmath(n)$ and ${\bf U}^\imath(n)$ and the Hecke algebras of types $B$ and $C$. These $i$-quantum groups appear in the quantum symmetric pairs of type AIII with no black nodes in their Satake diagrams. The $q$-Schur algebras involved have numerous applications; see, e.g., \cite{BKLW, DW1, LNX, Li, LW, MW}. Note that, for the type $D$ case, 
a commuting action of the $i$-quantum groups ${\bf U}^\jmath (n)$, ${\bf U}^\imath (n)$ and Hecke algebras of type $D$ on tensor spaces was first observed by Ehrig and Stroppel~\cite{ES} in their study of 
  maximal parabolic category $\sO$/isotropic Grassmannians of $\mathfrak{so}_{2n}$.   
  See also Bao's work~\cite{Bao}.
  It is natural to ask if there is a Schur--Weyl--Hecke duality in type $D$?

In \cite{FL}, Z. Fan and the second-named author investigated a natural geometric setting of the $q$-Schur algebras of type $D$ (see \cite[Appendix]{LL} for the associated algebraic setting) and their limiting algebra $\mathcal K$, following the work \cite{BKLW}. Moreover, certain $i$-quantum groups are also introduced so that the type $D$ $q$-Schur algebras are their homomorphic images. However, it has been noticed that multiplication formulas in the type $D$ $q$-Schur algebras discussed in \cite[\S4]{FL} are incomplete (see \cite[Prop.~4.3.2]{FL}). In other words, more cases need to be considered in the determination of the natural basis when the algebra is defined by $q$-permutation modules of the type $D$ Hecke algebra (see \cite[Appendix]{LL}). Hence, the dimension formulas need to be revised and multiplication formulas about these new basis elements should be computed as well. 

The aim of the paper is to fill these gaps in \cite{FL} and \cite{LL} and to set up a foundation for a possible solution to the aforementioned question. It is interesting to note that, on the way to achieve the goal, we discovered more. 
The strategy here is new as we develop two parallel theories: a type $B$ theory in unequal parameters and a type $D$ theory. We then use a splitting process via either weight idempotents or geometric splitting to link the two. 

We now describe it in more details. We start with the algebraic approach developed in \cite{DS2}, where the Hecke algebra $\cH_\bsq$ of type $D_r$ is embedded in the Hecke algebra $H_{\bsq,1}$ of type $B_r$ in unequal parameters. Interestingly, there is a geometric construction for $H_{q,1}$, where $q$ is a prime power,  using complete isotropic flag variety over the even-order orthogonal group $\O_{2r}(q)$. Then, we may present the $(\bsq,1)$-Schur algebra $S_{\bsq,1}^\jmath(n,r)$ of type $B$ both algebraically, using $q$-permutation modules, and geometrically, using pairs of isotropic partial flag variety over $\O_{2r}(q)$ and their associated convolution algebra (see Theorem \ref{geosettingB}).  In this way, we may define the $\bsq$-Schur algebra $S_{\bsq}^\kappa(n,r)$ of type $D$ by restriction from $H_{\bsq,1}$ to $\cH_\bsq$ and identify it with the algebra $S^\scd_\bsq(n,r)$ defined in \cite{LL} (see Proposition \ref{kappa-D}).  Next, using a criterion for splitting an $\O_{2r}(q)$-orbit into two $\SO_{2r}(q)$-orbits,
a new parametrization of $\SO_{2r}(q)$-orbits is obtained (see \eqref{SOorbit}). We then introduce the natural basis for $S_\bsq^\scd(n,r)$, indexed by (labelled) double cosets in the type $D$ Weyl group, and the defining orbital basis for the convolution algebra over $\SO_{2r}(q)$, indexed by $\SO_{2r}(q)$-orbits. Theorem \ref{geosettingD} gives a geometric setting for $S_q^\scd(n,r)$ via  a basis correspondence.
%Since double cosets about subgoups $H$ and $K$ of a finite group $G$ can be interpreted as the orbits of an $H\times K$-action on $G$ and these $H\times K$-orbits are in bijection with the orbits of the diagonal $G$-action on $G/H\times G/K$, we may easily keep a track of the labelling of the bases both algebraically and geometrically. 
This precise basis correspondence (or identification) in loc. cit. is the key to finding  the fundamental multiplication formulas in these generic $q$-Schur algebras of type $D$ (Theorems \ref{thmhhh}--\ref{thbbh} and \ref{thmchhh}--\ref{thchbb}) through a splitting process of the corresponding formulas of type $B$ given in Theorem \ref{sjcon}. 

More precisely, we summarize the fundamental multiplication formulas in the paper as follows. (See notational list below.)

\begin{thm} 
Let $S_{\bsq,1}^\jmath(n,r)$ and $S_{\bsq}^\kappa(n,r)$ be the $\bsq$-Schur algebras of type $B$ and $D$, respectively, over the polynomial ring $\mathbb Z[\bsq]$
with their respective bases $\{e_A\mid  A \in \Xi\}$ and $\{\phi_\cBA\mid \cBA\in\check \Xi\}$. Suppose $B,C\in\Xi$ such that $B-E_{h,h+1}^\th$ and $C-E_{h+1,h}^\th$ $(h\in[1,n])$ are diagonal and, in $S_{\bsq,1}^\jmath(n,r)$,
$$e_Be_A=\de_{\co(B),\ro(A)}\sum_{1\leq p\leq N}g_{h,A,p}\e{\apl},\quad e_Ce_A=\delta_{\co(C),\ro(A)}\sum \limits_{1 \leq p \leq N} \gca e_{\apu},$$
for some $g_{h,A,p},\gca\in\mathbb Z[\bsq]$.
(See \eqref{hp} for the notation $\apl,\apu$.) If $\cBA$, $\cBB$, and $\check\BC$ in $\check\Xi$ are obtained by a certain splitting process from $A,B,$ and $C$, respectively, then the structure constants for $S_{\bsq}^\kappa(n,r)$ appearing in $\phi_{\cBB}\phi_\cBA$ (resp., $\phi_{\check\BC}\phi_\cBA$) are some of the $\gba$ (resp. $\gca$ or $\frac12\gca$).
\end{thm}

We organize the paper as follows. In Section 2, we first introduce,  following \cite[(3.3.2)(2)]{DS1} and \cite[Rem.~2.6(1)]{DS2}, the algebraic definition for the type $D$ $q$-Schur algebras $S^\kappa_{\bsq}(n,r)$ associated with (or covering) the type $B$ algebra $S_{\bsq,1}^\jmath(n,r)$ in unequal parameters. We then reinterpret $S^\kappa_{\bsq}(n,r)$ as the algebra $S^\scd_\bsq(n,r)$ defined by Heck algebras of type $D$ and their $q$-permutation modules (see \cite[(A.3.3)]{LL}). In Section 3, we present the geometric definition of $S^\jmath_{\bsq,1}(n,r)$, largely following \cite{BKLW, FL}. In Section 4, we develop a precise labelling for those $\O_{2r}(q)$-orbits which split into two $\SO_{2r}(q)$-orbits. This allows us to label the natural basis for $S^\scd_\bsq(n,r)$ both algebraically in terms of double cosets and geometrically in terms of $\SO_{2r}(q)$-orbits. Finally, in Section 5, the multiplication formulas in the $\bsq$-Schur algebra $S_{\bsq,1}^\jmath(n,r)$ of type $B$ (with unequal parameters) will be derived geometrically via convolution products as in \cite{BKLW}. In the last Sections 6 and 7, fundamental multiplication formulas for the $\bsq$-Schur algebra of type $D$ are derived using either weight idempotents or orbit splitting techniques. This 
includes especially those missing ones in \cite{FL} (see Remarks \ref{error} and \ref{error2}).

In a forthcoming paper, we will use the multiplication formulas discovered here to seek the existence of the Schur--Weyl--Hecke duality in type $D$.
% with an expectation to a construction of the subalgebra ${\bf U}^\kappa(n)$ between ${\bf U}^\jmath_{\bsq,1}(n)$ and ${\bf U}(\mathfrak{gl}_{2n+1})$. 

\smallskip
\noindent
{\bf Some notation hints:} 
There are two parallel notations for types $B$ and $D$. The following table lists some of them for reference. Let $\ep\in\{+,-\}$.
\begin{center}
\begin{tabular}{|c|c|c|c|c|c|c|c|c|c|}\hline
&(1)&(2)&(3)&(4)&(5)&(6)&(7)&(8)&(9)\\\hline
$B$&$W$&$\La(n+1,r)$&$H_{\bsq,1}$&$S^\scb_{\bsq,1}(n,r)$&$e_A$&$\Xi=\Xi_{N,2r}=\{A\}=\hhh\sqcup\overset\circ\Xi$&$G(q)$&$\sO_A$&$\ro,\co$\\\hline
$D$&$\cW$&$\La^\scd(n,r)$&$\cH_\bsq$&$S^\scd_\bsq(n,r)$&$\phi_\cBA$&$\check\Xi=\check \Xi(n,r)=\{\cBA\}=\hhhd\sqcup\overset\circ\Xi_+\sqcup\overset\circ\Xi_-
$&$\cG(q)$&$\csO(\cBA)$&$\rw,\cw$\\\hline
\end{tabular}
\end{center}
%\begin{center}
\begin{multicols}{2}
where 
\begin{itemize}
\item[(1)] Weyl groups $W=W^\scb$ and $\cW=W^\scd$; \item[(2)] Index sets of orbits of parabolic subgroups\\ or isotropic flags in $\sX$; \item[(3)] Hecke algebras  in unequal/equal parameter; \item[(4)] $\bsq$-Schur algebras; \item[(5)] Natural bases for $\bsq$-Schur algebras;
\item[(6)] Index sets for natural bases; \item[(7)] Finite orthogonal/special orthogonal groups $G(q)=\O_{2r}(q)/\cG(q)=\SO_{2r}(q)$; \item[(8)] $G(q)$-/$\cG(q)$-orbits in $\sX\times\sX$, set of pairs of $n$-step isotropic flags; \item[(9)] Weight functions defined in \eqref{roco} and \eqref{rwcw}.
\end{itemize}
\end{multicols}

Throughout the paper, let $n$ be a positive integer and let $N=2n+1$.

Let $\sA:=\mathbb Z[\bsq]$ be the integral polynomial ring in indeterminate $\bsq$. For any integer $m\geq0$, define the Gaussian integer $\ggi{m}=\frac{\bsq^m-1}{\bsq-1}$. 

For integers $a<b$, let
$[a,b]:=\{x\in\mathbb Z\mid a\leq x\leq b\}$.

\subsection{Acknowledgements}

The second-named author is grateful to Chun-Ju Lai for pointing out a gap in  the parametrization of $\SO_{2r}(q)$-orbits in~\cite{FL}; see \cite[(A.4.2)]{LL}, Remark~\ref{error}, and (\ref{SOorbit}).

The authors express their gratitude to the University of New South Wales and the University at Buffalo for their hospitality during the preparation of this paper. Additionally, the third author acknowledges the financial support provided by the China Scholarship Council.

We also thank Professor Weiqiang Wang for his comments on an early version of the paper. 

We sincerely thank the referee for  valuable comments and  suggestions.
%\end{center}

\section{The $\bsq$-Schur algebras of type $D$: an algebraic setting}

We start with the definitions of the Weyl groups, Hecke algebras, and $\bsq$-Schur algebres of type $B/D$. By
identifying $W$ and its parabolic subroups with certain fixed-point subgroups of the symmetric group $\fS_{2r}$, we then index the natural (or double coset) basis for the $\bsq$-Schur algebra $S_{\bsq,1}^\scb(n,r)$ of type $B$ by a certain matrix set $\Xi_{N,2r}$. This matrix set will soon be used in Section 3 to label the $\O_{2r}(q)$-orbits in $\sX\times \sX$ for an $n$-step isotropic flag variety $\sX$.      

%$\bbb,\bbh,\hbb,\hbh,\hhh$

\subsection{The Weyl groups $W/\cW$, Hecke algebras $H_{\bsq,1}/\cH_\bsq$, and $\bsq$-Schur algebras $S^\jmath_{\bsq,1}(n,r)/S^\kappa_\bsq(n,r)$} Let $W=\WB=W(B_r)$ ($r\geq2$) and $ \cW=\WD=W(D_r)$ ($r\geq 4$) denote the Weyl groups of type $B_r$ and $D_r$ associated with the Dynkin diagrams:

\medskip
 
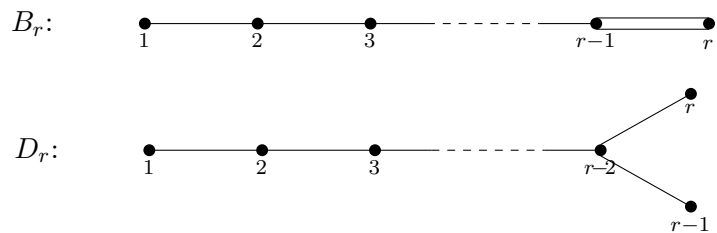
\begin{center}
\begin{tikzpicture}[scale=1.5]
\fill(-1,0) node {$B_r$:};
\fill (0,0) circle (1.5pt);
\fill (1,0) circle (1.5pt);
\fill (2,0) circle (1.5pt);
\fill (4,0) circle (1.5pt);
\fill (5,0) circle (1.5pt);
  \draw (0,0) node[below] {$_1$} --
        (1,0) node[below] {$_2$} -- (2,0)node[below] {$_3$}--(2.5,0);
\draw[style=dashed](2.5,0)--(3.5,0);
\draw (3.5,0)--(4,0) node[below] {$_{r-1}$};
\draw (4,0.05) --
        (5,0.05);
\draw (4,-0.05) --
        (5,-0.05) node[below]  {$_r$};
%\draw (4.4,0.1)--(4.6,0);
%\draw (4.4,-0.1)--(4.6,0);
\end{tikzpicture}
\end{center}
\medskip
\begin{center}
\begin{tikzpicture}[scale=1.5]
\fill(-1,0) node {$D_r$:};
\fill (0,0) circle (1.5pt);
\fill (1,0) circle (1.5pt);
\fill (2,0) circle (1.5pt);
\fill (4,0) circle (1.5pt);
\fill (4.8,0.5) circle (1.5pt);
\fill (4.8,-0.5) circle (1.5pt);
  \draw (0,0) node[below] {$_1$} --
        (1,0) node[below] {$_2$} -- (2,0)node[below] {$_3$}--(2.5,0);
\draw[style=dashed](2.5,0)--(3.5,0);
\draw (3.5,0)--(4,0) node[below] {$_{r\!-\!2}$};
\draw (4,0.05) --
        (4.8,0.5) node[below] {$_{r}$};
\draw (4,-0.05) --
        (4.8,-0.5) node[below]  {$_{r-1}$};
%\draw (4.4,0.1)--(4.6,0);
%\draw (4.4,-0.1)--(4.6,0);
\end{tikzpicture}%\vspace{-2ex}
%Figure 1.
\end{center}
Then we have the following sequence of three Coxeter systems:\medskip
\begin{equation}
(\oW,\bar S)\leq(\dW,\cS)\leq (W,S),
\end{equation}
where
$S=\{s_1,\cdots,s_{r-1},s_r\}$,
$\cS=\{s_1,\cdots,s_{r-1},\vsg_r\}$ with $\vsg_r:=s_rs_{r-1}s_r$, and
 $\bar S=\{ s_1,\cdots,s_{r-1}\}$. Note that  $\bar W$ is isomorphic to the symmetric group $\fS_r$ on $r$ letters.

%Let $\ep_1,\ep_2,\ldots,\ep_r$ be the natural basis for $\mathbb R^r$. Then the positive root systems 
%$$\Phi^+=\{\ep_i\pm\ep_j,\ep_k\mid1\leq i< j\leq r,k\in[1,r]\},\quad \check\Phi^+=\{\ep_i\pm\ep_j\mid i,j\in[1,r],i< j\}.$$
%The corresponding simple root systems are
%$$\Delta=\{\al_1,\ldots\al_{r-1},\al_r\},\quad \check\Delta=\{\al_1,\ldots\al_{r-1},\check\al_r\},$$
%where $\al_i=\ep_i-\ep_{i+1}$, for $1\leq i\leq r-1$, $\al_r=\ep_r$, and $\check\alpha_r=\ep_{r-1}+\ep_r$.

Let
$t_r=s_r$, $t_{r-i}=s_{r-i}t_{r-i+1}s_{r-i}$, $u_i=t_rt_{r-i}$ for
$1\le i\le r-1$, and let $C=\langle t_1,\cdots,t_r\rangle$ and
$\check C=\langle u_1,\cdots,u_{r-1}\rangle$.
Then, we have $W=C\rtimes\bar W$ and $\dW=\check C\rtimes\bar W$.
Moreover, we often identify $\oW$ with $\fS_r$ in the sequel.

Let $\ell$ (resp. $\dl$) be the length function on $\wW$ 
(resp. $\dW$) with respect to $S$ (resp. $\dS$) 
and $n_r$ the function giving
the number of $s_r$ in a reduced expression of an element of $\wW$
(cf. \cite[(2.1.1)]{DS1}). Clearly, 
\begin{equation}\label{n_r}
\cW=\{w\in W\mid n_r(w)\in 2\mathbb N\}.
\end{equation}

The inner automorphism $\fkf:w\mapsto s_rws_r$ on $\wW$ induces by restriction a graph automorphism $\fkf$ on $\cW$
which {\it flips} the
Coxeter graph:
\begin{equation}\label{graph auto}
\fkf:\cW\longrightarrow\cW, \;\; w\longmapsto s_rws_r.
\end{equation}
Note that $\fkf$ fixes each element of $\bar S$, and interchanges
the two parabolic copies of $\fS_r$ in $\cW$.

Let $H_{\bsq,\bsq'}=H_{\bsq,\bsq'}^\scb$ be the 2-parameter Hecke algebra of type $B_r$, over the polynomial ring $\sA'=\mathbb Z[\bsq,\bsq']$ in two variables, generated by $T_1,\ldots, T_{r-1},T_r$ ($T_i=T_{s_i}$) which satisfy the relations:
\begin{equation}\label{rels}
\aligned
(1)\quad &T_iT_j=T_jT_i,\;\; |i-j|\geq2,\\
(2) \quad&T_jT_{j+1}T_j=T_{j+1}T_jT_{j+1}\;\; (1\leq j<r-1);\\
 (3)\quad&T_i^2=(\bsq-1)T_i+\bsq\;\; (1\leq i<r),\\
 (4)\quad&T_{r-1}T_rT_{r-1}T_r=T_rT_{r-1}T_rT_{r-1},\;\; \\
 (5)\quad&T_r^2=(\bsq^\prime-1)T_r+\bsq^\prime.
 \endaligned
 \end{equation}
Putting $T_w=T_{i_1}\cdots T_{i_l}$, where $w=s_{i_1}\cdots s_{i_l}\in W$ is a reduced expression, we obtain a basis $\{T_w\mid w\in W\}$ for $H_{\bsq,\bsq'}$. 

Let $H_{\bsq,1}=H_{\bsq,\bsq'}\otimes_{\sA'}\mathbb Z[\bsq]$ be the algebra obtained by base change from $\sA'$ to $\sA:=\mathbb Z[\bsq]$ under the specialization of $\bsq'=1$. Then $H_{\bsq,1}$ is presented by $T_1,\ldots, T_{r-1},T_r$ and relations
 \begin{equation}\label{Hq1 rels}
\text{(1)--(4) in }\eqref{rels}, \quad\text{and }\quad (5')\;\;T_r^2=1.
 \end{equation}

The $\mathbb Z[\bsq]$-algebra
$H_{\bsq,1}$ contains a subalgebra $\check H_\bsq:=H^\scd_\bsq$ which is isomorphic to the Hecke algebra of type $D_r$. Here $\check H_\bsq$  is generated by $T_1,\ldots, T_{r-1},T_{\vsg}:=T_rT_{r-1}T_r$ with defining relations (1)--(3) in \eqref{rels} together with:
$$T_\vsg^2=(\bsq-1)T_\vsg+\bsq,\;\;T_\vsg T_{i}=T_{i}T_\vsg \;(i\neq r-2),\;\;T_{r-2}T_\vsg T_{r-2}=T_\vsg T_{r-2}T_\vsg.
$$
 We will use the same notation as above describing basis $\{T_w\mid w\in W\}$ for
$H_{\bsq,1}$ and basis $\{T_w\mid w\in \cW\}$ for $\check H_\bsq$. We often write $H=H_{\bsq,1}$ and $\cH=H_\bsq$ for notational simplicity (see \S4.1).

%\begin{lem} \label{invo}The algebras $H_{\bsq,1}$ admit the following involutions.
%\begin{itemize}
%\item[(1)] $\fkf_\bsq:H_{\bsq,1}\to H_{\bsq,1}, \;\;T_i\mapsto T_i\;(i\neq r-1),\;\; T_{r-1}\mapsto T_rT_{r-1}T_r.$ (Note that $\fkf_\bsq(\cH_\bsq)=\cH_\bsq$.)
%\item[(2)] $\iota:H_{\bsq,1}\to H_{\bsq,1}, \;\;T_i\mapsto T_i\; ( i\neq r),\;\; T_r\mapsto -T_r.$ (Thus, the fixed-point subalgebra $(H_{\bsq,1})^\iota =\cH_\bsq$, and $H_{\bsq,1}=\cH_\bsq\oplus T_r\cH_\bsq$, where $T_r\cH_\bsq$ is the eigenspace of $\iota$ for eigenvalue $-1$.)
%\end{itemize}
Both $H_{\bsq,1}$ and $\cH_\bsq$ admit an algebra anti-involution (anti-automorphism of order 2):
\begin{equation}\label{tau}
\tau:T_w\longmapsto T_{w^{-1}}.
\end{equation}
%\end{lem}
%Note that $\fkf_\bsq$ is induced from the inner automorphism $w\mapsto s_rws_r$ on $W$. Thus we have, $\fkf_q(T_i)=T_rT_iT_r$, for all $i$ with $1\leq i\leq r$.

We use the following indexing set to label (standard) parabolic subgroups of $W$:
$$\Lambda(n+1,r)=\{\la=(\la_1,\ldots,\la_{n},\la_{n+1})\in \mathbb N^{n+1}\mid \la_1+\cdots+\la_{n+1}=r\}.$$ 
 For 
$\la=(\la_1,\ldots,\la_{n},\la_{n+1})\in\Lambda(n+1,r)$, define %$\la_{\leqslant i}=\la_1+\cdots+\la_i$.
 the parabolic subgroup of $W$ associated with $\la$ as the subgroup
 \begin{equation}\label{parabolic}
W_\la=\langle S\backslash\{s_{\la_1+\cdots+\la_i}\mid i\in[1,n]\}\rangle,
\end{equation}
and let $x_\la=\sum_{w\in W_\la}T_w$. The $x_\la H_{\bsq,1}$ is a right $H_{\bsq,1}$-module with basis $x_\la T_d$, for $d\in\sD_\la$, where $\sD_\la$ is the set of distinguished right coset representative of $W_\la$. Let $\sD_{\la\mu}:=\sD_\la^{-1}\cap\sD_\mu$, for $\la,\mu\in\Lambda(n+1,r)$. This is the set of distinguished representatives of $W_\la$-$W_\mu$ double cosets.

Following \cite[(3.3.2)(2)]{DS1} (or \cite{BWW} in the context of $i$-quantum groups), we define
\begin{equation}
S_{\bsq,1}^\scb(n,r)=S^\jmath_{\bsq,1}(n,r)=\End_{H_{\bsq,1}}\Big(\bigoplus_{\la\in\La(n+1,r)}x_\la H_{\bsq,1}\Big).
\end{equation}

\begin{rem}
The double subscripts in the notation $S_{\bsq,1}^\scb(n,r)$ indicate this unequal parameter setting. 
It is different from the equal-parameter algebra $\mathbf S^\jmath$ or
$\sS^\jmath_\sZ(n,r)$  investigated in \cite{BKLW} or in \cite{DW1}.
\end{rem}

Following \cite[Rem.~2.6(1)]{DS2}, we define, by restriction, {\it the $\bsq$-Schur algebra of type $D$}\footnote{\label{ftone}An alternative definition is given in Definition \ref{S^D}}: %The term will be justified at the end of the section.
\begin{equation}\label{kappa}
S_{\bsq}^\kappa(n,r)=\End_{\check H_\bsq}\big(\bigoplus_{\la\in\La(n+1,r)}x_\la H_{\bsq,1}|_{\check H_\bsq}\big)
\end{equation}

Let $\bsq=\up^2$ and $\sZ:=\mathbb Z[\up,\up^{-1}]$. By base change to $\sZ$, we obtain the Hecke algebras $\sH_{\up,1},\check\sH_\up$ and the $\up$-Schur algebras  $\sS_{\up,1}^\scb(n,r)$ and $\sS_\up^\kappa(n,r)$:
\begin{equation}\label{kappaZ}
\sH_{\up,1}:=H_{\bsq,1}\otimes \sZ,\;\;\check\sH_\up:=\check H_\bsq \otimes\sZ,\;\;\sS_{\up,1}^\scb(n,r):=S_{\bsq,1}^\scb(n,r)\otimes_{\sA}\sZ,\;\;\sS_\up^\kappa(n,r)=S_{\bsq}^\kappa(n,r)\otimes_{\sA}\sZ.
\end{equation}
Though we will not consider these $\sZ$-algebras in the rest of the paper, they play important roles for the study of the Schur--Weyl--Hecke duality mentioned in the introduction.

 \subsection{Identifying $W,W_\la$ with certain fixed-point subgroups of $\fS_{2r}$}
 We may embed $W$ into $\fS_{2r}$ as a fixed-point subgroup $(\fS_{2r})^\theta$, where $\theta$ is the involution 
 $$\theta: \fS_{2r}\longrightarrow \fS_{2r}, (i,j)\longmapsto (2r+1-i,2r+1-j).$$

More precisely, the map 
\begin{equation}\label{iota}
\aligned
\iota: W&\longrightarrow \fS_{2r}, \\
s_i&\longmapsto (i,i+1)(2r+1-i,2r-i)\;\;(1\leq i<r),\quad s_r\longmapsto(r,r+1)\endaligned
\end{equation}
is an injective group homomorphism.
By identifying $W(=\WB)$ with its image $\iota(W)$, $W$ acts on the set $[1,2r]$ and, hence, on the power set of $[1,2r]$. With this identification, $W$ consists of
permutations 
\begin{equation}\label{permB}\begin{pmatrix} 1&2&\cdots&r&r+1&\cdots&2r-1&2r\cr
          i_1&i_2&\cdots&i_r&i_{r+1}&\cdots&i_{2r-1}&i_{2r}\end{pmatrix},
          \end{equation}
where $i_{j}+i_{2r+1-j}=2r+1$ for all $j$. In other words,  $W$ is generated by
$$s_i=\si_i\si_{2r-i}\;(1\leq i\leq r-1),\;\;s_r=\si_r,$$
where $\si_i=(i,i+1)$, for $i=1,2,\ldots,2r-1$ form the Coxeter generators of $\fS_{2r}$.

Moreover, the restriction of $\iota$ to 
 $\check W(=\WD)$ has the image which is the subgroup generated by 
\begin{equation}\label{SD}
s_i=\si_i\si_{2r-i}, \;\;(1\leq i<r),\quad \vsg=\si_r\si_{r-1}\si_{r+1}\si_r=(r-1,r+1)(r,r+2).
\end{equation}
We often identify $\cW$ as a subgroup of $\fS_{2r}$ via $\iota$. Thus, $\cW$ acts on the power set of $[1,2r]$.

%Thus, the parabolic subgroups of $W$ defined in \eqref{parabolic} may be described as the stabliser subgroups of Young subgroups of $\fS_{2r}$ associated with certain compositions of $2r$.

For any composition $\mu=(\mu_1,\ldots,\mu_N)$ of $2r$, define the associated partial sum sequence
\begin{equation}\label{par sum}
\wtmu=(\wtmu_1,\wtmu_2,\cdots,\wtmu_N),\;\; \text{ where } \wtmu_i=\mu_1+\cdots+\mu_i \;(\wtmu_0=0).
\end{equation}

Let $\fS_\mu$ be the Young subgroup of $\fS_{2r}$ associated with $\mu$. This is the stabilizer subgroup of the standard Young tabloid\footnote{If we identify $\mu$ with its Young diagram consisting of $\mu_1$ boxes in row 1, and $\mu_2$ boxes in row 2 and so on, then a $\mu$-tableau is obtained by filling the numbers $1,2,\ldots,2r$ into boxed.
A $\mu$-tabloid is a row-equivalent class of $\mu$-tableaux.} $R^\mu=R^\mu_1\times\cdots\times R_N^\mu$, where $R^\mu_i:=[\tmu_{i-1}+1, \tmu_{i-1}+\mu_i]$, for all $i=1,2,\ldots,N$.
In other words, 
\begin{equation}\label{S_mu}
\fS_\mu=\bigcap_{i=1}^N\text{Stab}_{\fS_{2r}}(R_i^\mu),
\end{equation}
 consisting of permutations which stabilize each of the subsets $R^\mu_i$.  
%In other words, $\fS_\mu=\bigcap_{i=1}^N\text{Stab}_{\fS_{2r}}(R_i^\mu)$.

For $\la\in \Lambda(n+1,r)$, let
$$\widehat{\lambda}:=(\lambda_1,\cdots, \lambda_n,2\lambda_{n+1},\lambda_n,\cdots, \lambda_1).$$
This defines an injective map
\begin{equation}\label{hatmap} \widehat{\  }:\Lambda(n+1,r) \longrightarrow \Lambda(2n+1,2r),\;\;
    \lambda \longmapsto \widehat{\lambda}.
    \end{equation}
Let $\widehat{\Lambda}(n+1,r)$ denote its image.
We now describe the parabolic subgroup $W_\la$ defined in \eqref{parabolic} in terms of stabilizers of the subsets $R_i^\la:=R_i^\wla$, for $1\leq i\leq \sfm$, where $\sfm=\sfm(\la)$ denotes the {\it maximal index} $i$ satisfying $\la_i\neq0$.

%can be identified as the fixed-point subgroup
Let 
\begin{equation}\label{Dnr}
\aligned
 \sD(n,r)&:=\{(\la,d,\mu)\mid \la,\mu\in\Lambda(n+1,r), d\in\sD_{\la,\mu}\},\\
\Xi_{N,2r}&:=\{A=(a_{i,j})\in \text{Mat}_{N,N}(\mathbb N)\mid a_{i,j}=a_{N+1-i,N+1-j}, \sum_{i,j}a_{i,j}=2r\}.\endaligned
\end{equation}

\begin{lem}\label{map kap} Let $\la,\mu\in\Lambda(n+1,r)$, $w\in\fS_{2r}$,  $d\in\sD_{\la,\mu}$, and $\td=\iota(d)$ (see \eqref{iota}).
\begin{itemize}
\item[(1)] $W_\la=(\fS_{\wla})^\theta= \bigcap_{i=1}^\sfm\Stab_W(R^\la_i).$ 
\item[(2)] Suppose the double coset $\fS_{\widehat\la}w\fS_{\widehat\mu}$ defines a $(2n+1)\times(2n+1)$ matrix $A=(a_{i,j})$ over $\mathbb N$ whose entries sum to $2r$.
Then $\fS_{\widehat\la}w\fS_{\widehat\mu}$ is stabilized by $\theta$ if and only if $a_{i,j}=a_{N+1-i,N+1-j}$ for all $i,j\in[1,N]$ ($N=2n+1$). In particular, for $\td=\iota(d)\in\fS_{2r}$,  $\theta$ stabilizes $\fS_\wla \td\fS_\wmu$ and the double coset $W_\la dW_\mu=(\fS_\wla \td\fS_\wmu)^\theta$.
\item[(3)]  There is a bijection 
\begin{equation}\label{fkd}
\mathfrak d:\sD(n,r)\longrightarrow \Xi_{N,2r},\;\;(\la,d,\mu)\longmapsto A=(|R^{\widehat\la}_i\cap dR^{\widehat\mu}_j|).
\end{equation}
Moreover, we have $s_r\in W_{\la}\cap {}^dW_\mu$ if and only if $|R^{\widehat\la}_{n+1}\cap dR^{\widehat\mu}_{n+1}|\geq2$.
\end{itemize}
\end{lem}
\begin{proof} The equality $W_\la=\fS_\wla\cap W=(\fS_{\wla})^\theta$ is clear from the definition of the embedding $\iota$.
To see the second equality, we observe
 $$w(R^\la_i)=R^\la_i,\text{ for all }1\leq i\leq\sfm\iff w(R^\wla_i)=R^\wla_i, w(R^{\wla}_{2n+2-i})=R^{\wla}_{2n+2-i},\text{ for all }1\leq i\leq\sfm.$$
 Thus, we have 
 $$%\text{Stab}_W(R^{\widehat\la}):
  \bigcap_{i=1}^\sfm\Stab_W(R^\la_i)=\bigcap_{i=1}^N\text{Stab}_W(R_i^{\wla_i})=\bigcap_{i=1}^N\big(\text{Stab}_{\fS_{2r}}(R_i^{\wla_i})\cap W\big)=\fS_{\wla}\cap W=(\fS_{\wla})^\theta.$$
 proving (1).
 
The remaining but the final assertion is well-known; see, e.g., \cite[Lem.~3.1]{DW1}. We now prove the last assertion.

For $\la\in\La(n+1,r)$ and $w\in W$, let $R^{\widehat\la}=R^{\widehat\la}_1\times\cdots \times R^{\widehat\la}_N$ and $w.R^{\widehat\la}:=w(R^{\widehat\la}_1)\times\cdots \times w(R^{\widehat\la}_N)$, which are called $\wla$-{\it tabloids}.  This defines a $W$-set $X_N:=\{w.R^{\widehat\la}\mid \la\in\La(n+1,r),w\in W\}$, consisting of all $\wla$-tabloids.  
%$$\text{Stab}_W(R^{\widehat\la}):=\bigcap_{i=1}^N\text{Stab}_W(R_i^{\wla_i})=\bigcap_{i=1}^N\big(\text{Stab}_{\fS_{2r}}(R_i^{\wla_i})\cap W\big)=\fS_{\wla}\cap W=(\fS_{\wla})^\theta.$$
This $W$-action on $X_N$ induces 
a diagonal action of $W$ on $X_N\times X_N$ such that each orbit has a representative $(R^\wla, dR^\wmu)$ for some $\la,\mu\in\La(n+1,r)$ and $d\in\sD_{\la,\mu}$.
For $R,R'\in X_N$ with $R=R_1\times\ldots\times R_N$, $R'=R'_1\times\ldots\times R'_N$, let
$$R\wedge R'=(R_1\cap R'_1)\times (R_1\cap R'_2) \times\cdots\times (R_1\cap R'_N)\times\ldots\times
(R_N\cap R'_1)\times (R_N\cap R'_2) \times\cdots\times (R_N\cap R'_N).$$
Define $X_{N^2}:=\{R\wedge R'\mid R,R'\in X_N\}$. Then the map
$$\wedge:X_N\times X_N\longrightarrow X_{N^2},\;\;(R,R')\longmapsto R\wedge R'$$
is bijective and $W$-equivariant. Thus, $\text{Stab}_W(R,R')=\text{Stab}_W(R\wedge R')=\prod_{i,j\in[1,N]}\text{Stab}_W(R_i\cap R'_j).$ Hence, for $(\la,d,\mu)\in\sD(n,r)$, $R=R^{\widehat\la}$ and $R'=dR^{\widehat\mu}$, 
since $\text{Stab}_W(R,R')=W_\la\cap{}^dW_\mu$, it follows that
$$\aligned
s_r\in W_\la\cap{}^dW_\mu&\iff s_r\in \text{Stab}_W(R^{\widehat\la}\wedge dR^{\widehat\mu})\\
&\iff s_r\in\text{Stab}_W(R^{\widehat\la}_{n+1}\wedge dR^{\widehat\mu}_{n+1})\\
&\iff r,r+1\in R^{\widehat\la}_{n+1}\cap dR^{\widehat\mu}_{n+1}\iff|R^{\widehat\la}_{n+1}\cap dR^{\widehat\mu}_{n+1}|\geq2,
\endaligned$$
as desired.
%Note that the first assertion in (2) follows from the bijections 
%$$\sD(n,r)\overset{\kappa_1}\longrightarrow W\setminus(X\times X)\overset{\kappa_2}\longrightarrow W\setminus X_{N^2},$$ 
%sending $(\la,d,\mu)$ to the orbit $W.(R^{\widehat\la}\wedge dR^{\widehat\mu})$.
\end{proof}

\begin{rem}In the proof above, Lemma \ref{map kap}(1) implies that we may replace the $W$-set $X_N$ of $\wla$-tabloids by its truncated version,
the set of all $\la$-tabloids $w.R^\la=w(R^\la_1)\times\cdots\times w(R^\la_{\sfm(\la)})$ ($\la\in\La(n+1,r), w\in W$), where 
$R^\la_1\sqcup\cdots\sqcup R^\la_\sfm=
[1,r+\la_{n+1}]$.
\end{rem}
%The embedding of $W$ (resp., $\cW$) into $\fS_{2r}$ induces a truncated action of $W$ (resp., $\cW$) on the set $[1, r]$ or $[1, r+\la_{n+1}]$. For $\la\in\La(n+1,r)$, if $\sfm$ denotes the {\it maximal index} satisfying $\la_\sfm\neq0$, then 
% This is clear since, for $w\in W$, 
% $$w(R^\la_i)=R^\la_i\iff w(R^\wla_i)=R^\wla_i, w(R^{\wla}_{2n+2-i})=R^{\wla}_{2n+2-i}\iff w\in \fS_\wla\cap W=(\fS_\wla)^\theta=W_\la.$$
For a matrix $A=(a_{i,j})\in\Xi_{N,2r}$, form two sequences via its rows and columns and their associated partial sum sequences as defined in \eqref{par sum}:
\begin{equation}\label{r(A)}
\aligned
\bfr(A)&=(a_{1,1},a_{1,2},\cdots\cdots,a_{1,N},\cdots,a_{N,1},a_{N,2},\cdots,a_{N,N}),\quad
\widetilde\bfr_A=(\widetilde a^r_{1,1},\cdots\cdots,\widetilde a^r_{N,N})\\
\bfc(A)&=(a_{1,1},a_{2,1},\cdots\cdots,a_{N,1},\cdots,a_{1,N},a_{2,N},\cdots,a_{N,N}),\quad
\widetilde\bfc_A=(\widetilde a^c_{1,1},\cdots\cdots,\widetilde a^c_{N,N}).\endaligned
\end{equation}
If $A=\mathfrak d(\la,d,\mu)$, then we set $d_A=d$, $\wla=\ro(A)$, and $\wmu=\co(A)$.  Thus,
we have
\begin{equation}\label{roco}
\ro(A)=\Big(\sum_{j=1}^Na_{1,j},\sum_{j=1}^Na_{2,j},\ldots,\sum_{j=1}^Na_{N,j}\Big),\quad
\co(A)=\Big(\sum_{i=1}^Na_{i,1},\sum_{i=1}^Na_{i,2},\ldots,\sum_{i=1}^Na_{i,N}\Big).
\end{equation}
Note that $\fS_\wla\cap{\widehat d_A} \fS_\wmu{\widehat d_A}^{-1}=\fS_{{\bf r}(A)}$ and 
${\widehat d_A}^{-1}\fS_\wla{\widehat d_A}\cap \fS_\wmu=\fS_{{\bf c}(A)}$.

By deleting row $n+1$ and column $n+1$ (i.e., the central row and column of $A$), $A$ is divided into four $n\times n$ submatrices: the upper/lower right corner matrices $\urcm/\lrcm$ and the upper/lower left corner matrices $\ulcm/\llcm$. The following number will be very useful as it tells when the element $d_A$ is in $\cW$. Let
\begin{equation}
|\urcm|=\sum_{i,j\in[1,n]}a_{i,n+1+j}=|\llcm|
\end{equation}
 be the entry sum of the upper right corner matrix.

\begin{cor}\label{urcm}
Maintain the notation introduced above. For $A=(a_{i,j})\in\Xi_{N,2r}$, $d_A\in W$ is the following permutation  $\left(\begin{smallmatrix}
\widetilde{a}^c_{i-1,j}+1 &\widetilde{a}^c_{i-1,j}+2&\cdots&\widetilde{a}^c_{i-1,j}+a_{i,j}\\
\widetilde a^r_{i,j-1} +1
& \widetilde a^r_{i,j-1}+2
&\cdots&\widetilde a^r_{i,j-1}+a_{i,j}
\end{smallmatrix}\right)$ in $\fS_{2r}$, for all $(i,j)$ with $a_{i,j}>0$.  Moreover, we have
$d_A\in\cW$ if and only if $ |\urcm|\in 2\mathbb N$.
\end{cor}
\begin{proof}The first assertion follows from \cite[Ex. 8.2]{DDPW}, Thus, if $d_A$ permutes $(1,2,\ldots,2r)$ to $(i_1,i_2,\ldots,i_{2r})$ as displayed in \eqref{permB}, then $\{i_1,i_2,\ldots,i_{r}\}\cap \{r+1,r+2,\ldots,2r\}$ has cardinality $|\urcm|$. However, by \cite[(2.1.1)]{DS1},
$|\urcm|=n_r(d_A)$. Now, the last assertion follows from \eqref{n_r}.
\end{proof}

Each triple $(\la,d,\mu)\in\sD(n,r)$ defines a natural basis element in $S^\scb_{\bsq,1}(n,r)$. We denote this element by $e_A$ if $A=\mathfrak d(\la,d,\mu)$. By definition, 
\begin{equation}\label{e_A}
e_A(x_\nu h)=\delta_{\mu,\nu}\Big(\sum_{w\in W_\la dW_\mu}T_w\Big)h.
\end{equation}
 \begin{prop}\label{stdbsB}
 The set $\{e_A\mid A\in\Xi_{2n+1,2r}\}$ forms a basis for $S^\scb_{\bsq,1}(n,r)$ which we call its natural basis.
 \end{prop}
 
\subsection{Parabolic subgroups of $\cW$ as stabilizers of tabloids associated with signed compositions}
We now extend Lemma \ref{map kap} above to (standard) parabolic subgroups of $\cW$; compare \cite{LL}.
 Let $\la\in\La(n+1,r)$ have maximal index $\sfm$.

\underline{\bf The $\lambda_{n+1}=0$ case.} If $\lambda_{n+1}=0$, then $\sfm\leq n$. Define the {\it signed composition} $\la^+$ and $\la^-$
%labelled by $\la^+$ and $\la^-$ 
by setting 
\begin{equation}\label{defrlam}
\aligned    
\cR^{\la^+}_i&=\cR^{\la^-}_i=R_i^{\lambda}=[\la_{i-1}+1,\la_{i-1}+\la_i]\;\;\text{for }1\leq i\leq \sfm-1,\\
    \cR_\sfm^{\lambda^+}&=[\tillam_{\sfm-1}+1,\tillam_{\sfm-1}+\la_\sfm]=\{\tillam_{\sfm-1}+1,\cdots, r-1,r\}, \\
    \cR_\sfm^{\lambda^-}&=[\tillam_{\sfm-1}+1,\tillam_{\sfm}-1]\cup\{r+1\}=\{\tillam_{\sfm-1}+1,\cdots, r-1,r+1\}.
\endaligned
\end{equation}
Form the {\it tabloids} 
$$\cR^{\la^+}:=\cR^{\la^+}_1\sqcup\cdots\sqcup \cR^{\la^+}_\sfm,\qquad
\cR^{\la^-}:=\cR^{\la^-}_1\sqcup\cdots\sqcup \cR^{\la^-}_\sfm.$$ 
Define, for $\ep\in\{+,-\}$,
$$\cW_{\lambda^\epsilon}:={\Stab}_\cW(\cR_1^{\lambda^\ep})\cap {\Stab}_\cW(\cR_2^{\lambda^\ep})\cap\cdots\cap {\Stab}_\cW(\cR_\sfm^{\lambda^\ep}).$$

%Let $\rlamepsi{\epsilon}=\{R^\lambda_i,R^{\lambda^\epsilon}_m|i\in[1,n+1]\text{ and }i\neq m\}$ for $\epsilon\in\{\pm\}$. 
%Denote by ${\rm Stab}(X)$ the stabilizer in $\check{W}$, for $\ep\in\{\pm \}$, let 
%$$\check W_{\lambda^\epsilon}=\cap_{i=0,i\neq m}^{n+1} {\rm Stab}(R_i^{\lambda})\cap {\rm Stab}( R_m^{\lambda^\epsilon}).$$

\underline{\bf The $\lambda_{n+1}\neq0$ case.} 
If $\lambda_{n+1}\neq0$, then define $\labt$ by setting
\begin{equation}\label{2.5.2} 
\begin{aligned}
    \cR_i^{\lamz}&=
    \begin{cases}
     [\tillam_{i-1}+1,\tillam_{i-1}+\lambda_i]= R^\la_i,&\text{ if }1\leq i\leq n,\\
      [\tillam_n+1,\tillam_n+2\lambda_{n+1}],&\text{ if }i=n+1.
    \end{cases}\\
%    R_{n+1}^{\lambda^0}&=\big\{\tillam_n+1,\cdots,r-1,r,r+1,\cdots,\tillam_n+2\lambda_{n+1}\big\}.\\
\end{aligned}
\end{equation}
Form the {\it tabloid} $\cR^{\labt}:=\cR^{\labt}_1\sqcup\cdots\sqcup \cR^{\labt}_{n+1}$
and define the subgroup of $\cW$:
$$\cW_{\labt}:=\mathop{\bigcap}\limits_{i=1}^{n+1} {\rm Stab}_\cW(\cR_i^{\labt}).$$
The following result is established in \cite[Lem.~A.2.1]{LL} with different indexing.
\begin{lem}\label{parabolicD}
For $\la\in\La(n+1,r)$ and $\ep\in\{\oo,+,-\}$, the subgroups $\cW_{\la^\ep}$ are parabolic subgroups of $\cW=\lr{s_1,\ldots,s_{r-1},\vsg_r}$ with $\vsg_r=s_rs_{r-1}s_r$.
More precisely, we have
\begin{itemize}
\item[(1)] If $\la_{n+1}=0$ (so $\sfm\leq n$) and $\lambda_\sfm>1$, then% we have $\check W_{\lambda^\epsilon}$ is generated by 
\begin{equation}\label{defcheckw1}
\aligned
\cW_{\la^+}  &=\langle  \check{S}-\{s_{{\tillam_1}},s_{\tillam_2},\cdots,s_{\tillam_{\sfm-1}},\vsg_r\}\rangle,\\
\cW_{\la^-} &=\langle   \check{S}- \{s_{\tillam_1},s_{\tillam_2},\cdots,s_{\tillam_{\sfm-1}},s_{r-1}\}\rangle.
\endaligned
\end{equation}
\item[(2)]
If $\la_{n+1}=0$ and $\lambda_\sfm=1$, then $\cR_\sfm^{\lambda^+}=\{r\}$ and $ \cR_\sfm^{\lambda^-}=\{r+1\}$ and \begin{equation}\label{defcheckw2} 
\cW_{\lambda^+}=\cW_{\lambda^-}=\langle\cS- \{s_{\tillam_1},s_{\tillam_2},\cdots,s_{\tillam_{\sfm-1}},s_{r-1},\vsg_r\}\rangle.
\end{equation}
\item[(3)] If $\la_{n+1}\neq0$, then
\begin{equation}\label{defcheckw3} 
\cW_\labt=\begin{cases}
   \lr{ \check{S}- \{s_{\tillam_1},s_{\tillam_2},\cdots,s_{\tillam_{n-1}},s_{r-1},\vsg_r\}}, &\text{ if }\lambda_{n+1}=1,\\
    \lr{\check{S}- \{s_{\tillam_1},s_{\tillam_2},\cdots,s_{\tillam_{n-1}}\}}, &\text{ if }\lambda_{n+1}>1.
\end{cases}\end{equation}
\end{itemize}
Moreover, every parabolic subgroup of $\cW$ is one of the forms above.
 \end{lem}

 For the parabolic subgroups $W_\la$ of $\WB$ defined in \eqref{parabolic}, we naturally define
\begin{equation}\label{cW_la}
\check W_\la=\begin{cases}W_\la,&\text{ if }\la_{n+1}=0;\\
W_\la\cap \check W,&\text{ if }\la_{n+1}\neq0. \end{cases}
\end{equation}
We may describe the parabolic subgroups of $\WD$ given in Lemma \ref{parabolicD} in terms of the subgroups $\cW_\la$.
Recall the flipping graph automorphism $\fkf$ in \eqref{graph auto}.
 \begin{cor}\label{par cor}Let $\la\in\La(n+1,r)$. 
 \begin{enumerate}
 \item If $\la_{n+1}=0$ and $\lambda_\sfm>1$, then
 $
 \cW_{\la^+}=W_\la$ and 
 $\cW_{\la^-}=s_r\cW_\la s_r=\fkf(\cW_\la).$
  \item If $\la_{n+1}=0$ and $\lambda_\sfm=1$, then 
  $\cW_{\la^+}=W_\la=\fkf(W_\la)=\cW_{\la^-}.$
  \item If $\la_{n+1}\neq0$, then $\cW_{\la^\oo}=\cW_{\la}=\fkf(\cW_{\la^\oo})$.
 \end{enumerate}
 In particular, $\cW_\la$ is a parabolic subgroup of $\cW$.\footnote{The reader should not be confused the notation $\cW_\la$ with $\cW_{\la^\ep}$. Note that $\cW_\la=\cW_{\la^+}$ or $\cW_\la=\cW_{\la^\oo}$, depending on $\la_{n+1}=0$ or not.} 
 Moreover, for every parabolic subgroup $\lr{J}$ of $\check W$ ($J\subseteq\check S$), there exists a $\la\in\La(n+1,r)$ such that $\lr{J}=\check W_\la$
or $\lr{J}=s_r\check W_\la s_r=\fkf(\cW_\la)$.
 \end{cor}
 
 \begin{rems} (1) The case (2) above shows that distinct tabloids may have the same stabilizer subgroups.
 This phenomenon is similar to the type $A$ fact that distinct compositions may define the same Young subgroup of the symmetric group.
 
(2)  Consider the parabolic subgroup $W_\mu=\langle s_1,\ldots,s_{r-1}\rangle$ of $W$, where $\mu={(r,0^n)}$. This is also a parabolic subgroup of $\chW$. That is $W_\mu=\chW_{\mu}$. Clearly, $\vsg=s_rs_{r-1}s_r\in\sD_{\mu,\mu}$, but $s_{r-1}=s_r\vsg s_r$ is not distinguished in the double coset $\chW_{\mu}(s_rds_r)\chW_{\mu}=\chW_{\mu}$.
\end{rems}

\subsection{The $\bsq$-Schur algebras of type $D$}
We are now ready to introduce the $q$-Schur algebra of type $D$ in terms of $\bsq$-permutation modules of $\cH_\bsq$ (see \cite[(A.3.3)]{LL} and compare \cite[Def.~2.5]{DS2}). Let
\begin{equation}\label{laD}
\La^\scd=\Lambda^\scd(n,r):=\La^\bullet(n+1,r)\sqcup\La^\circ_+(n+1,r)\sqcup \La^\circ_-(n+1,r),
\end{equation}
where $\Lambda^\oo(n+1,r):=\{\lambda^\oo\mid\lambda\in\Lambda(n+1,r),\la_{n+1}\neq0\}$ and $\Lambda^\circ_\pm(n+1,r):=\{\la^\pm\mid\lambda\in\La^\circ(n+1,r),\la_{n+1}=0\}$.

\begin{rem}We remark that, if $\sP(\cS)$ denotes the power set of $\cS$ and $n\geq r$,  then the following map is surjective:
$$\Lambda^\scd(n,r)\longrightarrow \sP(\cS),\;\;\al\longmapsto \cS\cap\cW_\al.
$$

As promised in footnote \ref{ftone}, the following is an alternative definition for the $\bsq$-Schur algebra of type $D$.
\end{rem}
\begin{defn}\label{S^D}[\cite{LL}]
For $\al\in\La^\scd$, let $\cx_{\al}=\sum_{w\in\cW_{\al}}T_w$ and define {\it the $\bsq$-Schur (or $\up$-Schur) algebra of type $D$}:
\begin{equation*}
\aligned
S_\bsq^\scd(n,r)&:=\End_{\check H_\bsq}\Big(\bigoplus_{\al\in\La^\scd(n,r)}\cx_\al{\check H_\bsq}\Big)\quad(\text{over }\sA=\mathbb Z[\bsq]);\\
\sS_\up^\scd(n,r)&:=\End_{\csH_\up}\Big(\bigoplus_{\al\in\La^\scd(n,r)}\cx_\al\csH_\up\Big)\quad(\text{over }\sZ=\mathbb Z[\up,\up^{-1}]).
\endaligned
\end{equation*}
\end{defn}
The algebras defined above involve only $q$-permutation modules of type $D$ Hecke algebras. They are isomorphic to the algebras defined in \eqref{kappa} and \eqref{kappaZ}. This will be proved in Proposition \ref{kappa-D}.

For every $\al\in\La^\scd(n,r)$, the identity map $1_\al$ on $\cx_\al\cH$ extends to an idempotent, called an {\it weight idempotent} in $S_\bsq^\scd(n,r)$. We record the following.

\begin{lem}\label{idemp}
The set $\{1_\al\mid\al\in\La^\scd(n,r)\}$ forms a set of idempotents such that $1=\sum_{\al\in\La^\scd(n,r)}1_\al$ is the identity element of $S_\bsq^\scd(n,r)$.
\end{lem}

We will introduce the natural basis for $S_\bsq^\scd(n,r)$ and discuss its geometric intepretation in Section 4.

%\begin{rem}
%In some cases such as in Lemma~\ref{parabolicD}(2) when $\lambda_\sfm=1$, we have $\cW_{\lambda^+}=\cW_{\lambda^-}$ and thus $\cx_{\lamp}=\cx_{\lamm}$. But the identity maps on $\cx_{\lamp}\cH$ and $\cx_{\lamm}\cH$ define distinct weight idempotent elements $1_{\la^+}$ and $1_{\la^-}$ in $S_\bsq^\scd(n,r)$.
%still two different elements in type $D$ condition since they represent different meanings. Therefore, for any $^{\ep_1}\!\!{A}^{\ep_2}, ^{\ep_3}\!\!{A}^{\ep_4}\in \xid$, if $\ep_1\neq \ep_3$ or $\ep_2\neq \ep_4$, then $^{\ep_1}\!\!{A}^{\ep_2}, ^{\ep_3}\!\!{A}^{\ep_4}$ represent different elements in $\xid.$
%\end{rem}
\section{A geometric setting for the $\bsq$-Schur algebras $S_{\bsq,1}^\scb(n,r)$}

In \cite{BKLW}, a geometric framework is introduced for  the $q$-Schur algebra $S^\jmath_q(n,r)$ of type $B$ in equal parameter. We now modify their construction using $(2n+1)$-step filtrations for the $2r$-dimensional space $\mathbb F_q^{2r}$ (see \cite{FL}) and the action by the orthogonal group $\O_{2r}(q)$. This results in a $q$-Schur algebra in unequal parameters. Restricting the action to the subgroup $\SO_{2r}(q)$  leads to a geometric setting for the $q$-Schur algebra of type $D$; see Theorems \ref{geosettingB} and \ref{geosettingD} in the next two sections.

\subsection{A geometric setting for the endomorphism algebra of a permutation module}

Let $G$ be a finite group acting on a finite set
$X$ (not necessarily faithfully or transitively), i.e., $X$ is a $G$-set, and let $\mathcal R$ be
a commutative ring with 1. Let $\fkF_G(X,\sR)$ be the $\sR$-free module of all functions $f:X\to\sR$ which are constant on the orbits of $G$. If $X'$ is another $G$-set and $\pi:X\to X'$ is a $G$-set map, then $\pi$ induces 
$\sR$-module maps $\pi_!:\fkF_G(X,\sR) \to\fkF_G(X',\sR)$ and $\pi^*:\fkF_G(X',\sR)\to\fkF_G(X,\sR)$ defined by
$$
\aligned
(\pi_!f)(x')&=\sum_{x\in\pi^{-1}(x')}f(x),\;\;\text{ for all }f\in\fkF_G(X,\sR),x'\in X';\\
(\pi^*f')(x)&=f'(\pi(x)),\;\;\text{ for all }f'\in\fkF_G(X',\sR),x\in X.
\endaligned
$$

The $G$-action is extended diagonally to $X\times X$. Consider three projection maps $\pi_{i,j}:X\times X\times X\to X\times X$ sending $(x_1,x_2,x_3)$ to $(x_i,x_j)$ for $(i,j)\in \{(1,2),(2,3), (1,3)\}$. Then $\fkF_G(X\times X,\sR)$ becomes an associative algebra with multiplication defined by 
$f*f'=(\pi_{1,3})_!(\pi_{1,2}^*f\cdot\pi_{2,3}^*f')$. In other words,
$$(f*f')(x,y)=\sum_{z\in X}f(x,z)f'(z,y),\;\;\text{ for all }(x,y)\in X\times X.$$

 Let $\Om$ % $\sO=\{\mathcal O_A\}_{A\in\Xi}$ 
be the set of all $G$-orbits %with respect to the diagonal action of $G$ 
on $ X\times X$. For each orbit $\sO$, there is an associated {\it (characteristic) orbital function} $f_\sO$ defined by setting 
\begin{equation}\label{orbifun}
f_\sO(x,y)=\begin{cases}
1, &\text{if } (x,y)\in\sO;\\  0,&\text{elsewhere.}
\end{cases}
\end{equation}%, and, for $A,A'\in\Xi$
Then $\fkF_G(X\times X,\sR)$ has a basis $f_\sO, \sO\in\Om$, and
\begin{equation}\label{f_Of_O'}
f_\sO*f_{\sO'}=\sum_{\sO''}c_{\sO,\sO',\sO''}f_{\sO''},\;\;\text{ where }\;c_{\sO,\sO',\sO''}=\#\{z\in X\mid (x,z)\in\sO,(z,y)\in\sO'\},
\end{equation}
for $(x,y)\in\sO''$. (This number is independent of the selection of $(x,y)\in\sO''$.) The identity element $\bf 1$ is the sum of the idempotent functions $f_\sO$ associated with $\sO$ of the form $G.(x,x)$.

Let $\sR X$ be the free $\sR$-module with basis $X$. Then $\sR X$ becomes an $\sR G$-module via the $G$-action on $X$. 
The following result is well-known.
\begin{lem}\label{scott} Let $G$ be a finite group.

(1) If $H,K$ are subgroups of $G$, then there is a bijection between the set $H\backslash G/K$ of all double cosets $HgK$ ($g\in G$) and the set $\Omega$ of all $G$-orbits in the $G$-set $G/H\times G/K$ with diagonal action.

(2) For a finite $G$-set $X$,
the endomorphism algebra
$\End_{\mathcal R G}(\mathcal R X)^\op$ is isomorphic to $\fkF_G(X\times X,\sR)$. 
\end{lem}

\subsection{Finite orthogonal groups and flag varieties}
We now look at the case, where $G=O_{2r}(q)$ or $SO_{2r}(q)$ is the finite orthogonal or special orthogonal group, and $X=\mathcal F_{n,r}^{\jmath}$ is the set of $n$-step filtrations (or partial flag variety) of isotropic subspaces.

Let $\mathbb{F}_q$ be a finite field of $q$ elements and of {\it odd} characteristic.\footnote{The odd characteristic guarantees that $\O_{2r}(q)=\O^+_{2r}(q)$ and $\SO_{2r}(q)=\SO^+_{2r}(q)$. See \cite[\S1.3.16]{Ge} and \cite[Th.~1.7.8, \S1.7.9]{Ge}.}
On the $2r$-dimensional vector space $\mathbb{F}_q^{2r}$, we fix a nondegenerate symmetric bilinear 
form $\langle-,-\rangle_J$ whose associated $2r\times 2r$-matrix is 
$$
\sfJ=\sfJ_{2r}=\begin{bmatrix}
0&\sfJ_r\\\sfJ_{r}&0
\end{bmatrix}\;\;
\text{ where }\;\;
\mathsf J_r=\begin{bmatrix}
0 & \cdots & 1\\
\vdots &%\begin{sideways} $\ddots$ \end{sideways} 
\rotatebox[origin=c]{90}{$\ddots$}&\vdots \\
1 & \cdots & 0
\end{bmatrix}_{r\times r}
$$
(thus, $\sfJ_{ij}=\delta_{i,2r+1-j}$ for any $1\leq i,j\leq 2r$). In other words, for $x,y\in \mathbb{F}_q^{2r}$,
\begin{equation}\label{formJ}
\langle x,y\rangle_{\sfJ}=x^t\sfJ y=x_1y_{2r}+x_2y_{2r-1}+\cdots+x_{2r}y_1.
\end{equation}
%Note that this non-degenerate symmetric bilinear form is equivalent to the one defined by the matrix 
%\begin{equation}\label{FLQ}\begin{bmatrix}0&I_r\\ I_{r}&0\end{bmatrix}=\begin{bmatrix}I_r&0\\0& \sfJ_{r}\end{bmatrix}^tJ_{2r}\begin{bmatrix}I_r&0\\0& \sfJ_{r}\end{bmatrix}.\end{equation}

The orthogonal group is defined by
$$\mathrm{O}_{2r}(q):=\{g\in \GL_{2r}(q)\ |\ J=g^tJg\}=\{g\in \GL_{2r}(q)\mid\langle gx,gy\rangle_\sfJ=\langle x,y\rangle_\sfJ\}.$$
This is a group with a split BN-pair, where $B=\widehat B(q)\cap \O_{2r}(q)$ with $\widehat B(q)$ being the Borel subgroup of $\GL_{2r}(q)$ consisting of upper triangular matrices and $N=\{\dot w\mid w\in W\}$ (see, e.g., \cite[p.80]{Ge}), where
\begin{equation}\label{dot w}
\text{$\dot w$ is the permutation matrix $\dot w=(\delta_{k,i_l})$ if $w$ is given as in \eqref{permB}, sending $k$ to $i_k$.}
\end{equation}

We also have the special orthogonal group
$$\mathrm{SO}_{2r}(q):=\{g\in \O_{2r}(q)\mid \det(g)=1\}.$$
Note that $\mathrm{SO}_{2r}(q)$ is also a group with a split BN-pair; see, e.g., \cite[Th.~1.7.8]{Ge}.

By convention, $W^{\perp}$ stands for the orthogonal complement of a vector subspace $W\subset\mathbb{F}_q^{2r}$ with respect to the bilinear form \eqref{formJ}.
We call a vector subspace $W\subset \mathbb{F}_q^{2r}$ isotropic if $W\subseteq W^{\perp}$. Note that a maximal isotropic subspace $M$ in $\mathbb{F}_q^{2r}$ has dimension $r$ and $M^\perp=M$. Also, the restriction of $\lr{\;\quad}_\sfJ$ to $W^\perp/W$ is a form \eqref{formJ} with rank $2r'$, where $r'=r-\dim W$.

We fix a positive integer $n$ and let $N=2n+1$. 
For a sequence of $n$ isotropic subspaces
\begin{equation}\label{n-isot}
W_1\subseteq W_2\subseteq\cdots\subseteq W_n,
\end{equation}
 we may extend it to an $N$-step filtration
\begin{equation}\label{N-isot}
F=F_\centerdot:\;\;0=F_0\subseteq F_1\subseteq F_2\subseteq \cdots \subseteq F_N\equiv
 \mathbb{F}_q^{2r},
 \end{equation}
 by setting $F_i=W_i$ and $F_{n+i}=W_{N-n-i}^\perp$, for all $1\leq i\leq n$. We call \eqref{N-isot} 
an $n$-step {\it isotropic flag} (extended to an $N$-step flag) and \eqref{n-isot} a {\it pure} $n$-step isotropic flag. If $n=r$ and $\dim W_i=i$,
then \eqref{n-isot} is called a complete isotropic flag. In this case, since $F_r=F_{r+1}$, \eqref{N-isot} becomes a ($2r$-step) complete isotropic flag after reindexing.

 Consider the following finite sets: 
\begin{itemize}
    \item The {\it $n$-step isotropic flag variety} is the set 
    $$\mathcal F_{n,r}^{\jmath}=\sF_n^\jmath(\mathbb F_q^{2r})=\{0=F_0\subseteq F_1\subseteq F_2\subseteq \cdots \subseteq F_N\equiv
 \mathbb{F}_q^{2r} \mid F_i\subseteq F_i^\perp, F_i^{\perp}=F_{N-i}, \forall i\in[1,n] \},$$ consisting of all $n$-step isotropic flags (extended to $N$-steps) in $\mathbb{F}_q^{2r}$. 
 \item The  {\it complete isotropic flag variety} is the set 
 $$\mathcal B_{r}=\sB_r(\mathbb F_q^{2r})=\{0\subset \sfF_1\subset \sfF_2\subset \cdots \subset \sfF_{2r}\equiv
\mathbb{F}_q^{2r}\mid\sfF_i\subseteq\sfF_i^\perp, \text{dim}\sfF_i=i, \sfF_i^{\perp}=\sfF_{2r-i}, \forall i\in[1,r] \},$$ consisting of all complete ($r$-step) isotropic flags in $\mathbb{F}_q^{2r}$.
\end{itemize}

We often set $\sX:=\mathcal F_{n,r}^\jmath$ and $\sB:=\mathcal B_{r}$ for notational simplicity. Both admit naturally an ${\mathrm{O}}_{2r}(q)$-action from the left. Moreover, ${\mathrm{O}}_{2r}(q)$ acts transitively on $\mathcal B$.

\subsection{Parabolic subgroups of $\O_{2r}(q)$}
Consider the natural basis $e_1,e_2,\ldots,e_{2r}$ for $\mathbb F_q^{2r}$ and form the (standard) pure isotropic complete flag $\sfF^s$ with 
\begin{equation*}\label{std complete flag}
\sfF^s_1=\lr{e_1},\sfF^s_2=\lr{e_1,e_2},\ldots,\sfF^s_r=\lr{e_1,e_2,\ldots,e_r}.
\end{equation*} 
%Then the subgroup
%$B=B(q):=\text{Stab}_{G(q)}(\sfF^s)$ is a Borel subgroup of $G=G(q):=\O_{2r}(q)$. 
Then we have $B(q) = \text{Stab}_{G(q)}(\sfF^s)$.
%Note that if $\widehat B(q)$ denotes the Borel subgroup of $\GL_{2r}(q)$ consisting upper triangular matrices, then $B(q)=\widehat B(q)\cap G(q)$. 
Moreover, the mapping $G(q)\to\mathcal B_r$ sending $g$ to $gF$ induces a bijection $G(q)/B(q)\to\mathcal B_r$ between the two $G$-sets.  Thus, by Lemma \ref{scott}(1), the $G$-orbits in $\sB_r\times \sB_r$ correspond bijectively to the set of double coset $B\setminus G/B$. Hence,  the Bruhat decomposition, 
every $G$-orbit $G.(\sfF,\sfF')$ in $\sB_r\times \sB_r$  has the from 
\begin{equation}\label{BwB}
G.(\sfF,\sfF')=G.(\sfF^s,\dot w\sfF^s), \text{ for some }w\in W^\scb,
\end{equation}
where $\dot w$ is corresponding permutation matrix in $G$. For example, $\dot s_i$ is the matrix obtained from the identity matrix $I_{2r}$ by swapping its rows $i$ and $i+1$, and rows $2r+1-i$ and $2r-i$. Hence, putting
$\dot W:=\{\dot w\mid w\in W^\scb\}\subset\Xi_{2r,2r}$, and $(\sB_r\times \sB_r)/{{\mathrm{O}}_{2r}(q)}$ to be the set of ${{\mathrm{O}}_{2r}(q)}$-orbits in  $\sB_r\times\sB_r$,
we have a bijective map
\begin{equation}\label{map p}
\fkp:(\mathcal B_r\times \mathcal B_r)/{{\mathrm{O}}_{2r}(q)}\longrightarrow \dot W.
\end{equation}

In general, for $\la\in \Lambda(n+1,r)$, where $2n+1\leq 2r$, let $P_{\widehat \la}(q)$ be the standard parabolic subalgebra of $\GL_{2r}(q)$ associated with $\widehat\la$, consisting of upper quasi-triangular matrices with blocks of sizes $\widehat\la_i$ on the diagonal.  Let
$$P_\la(q)=P_{\widehat\la}(q)\cap G(q).$$
Then $G(q)$ acts on the set $G(q)/P_\la(q)$ of left cosets $gP_\la(q)$ in $G(q)$. 

Associated with $\la$, there is a pure isotropic $n$-step partial flag $F^\la$ such that
\begin{equation}\label{std flag}
F^\la_1=\lr{e_1,\ldots,e_{\widetilde \la_1}},F^\la_2=\lr{e_1,\ldots,e_{\widetilde \la_2}},\ldots,F^{\la}_n=\lr{e_1,\ldots,e_{\widetilde \la_n}}.
\end{equation}
Then $P_\la(q)=\text{Stab}_{G(q)}(F^\la)$ and there is a $G(q)$-set isomorphism
$G(q)/P_\la(q)\cong G(q).F^\la$.

\begin{lem}\label{lab}
Maintain the notation above. We have $P_\la(q)\subseteq\SO_{2r}(q)$ if and only if $\la_{n+1}=0$.
\end{lem}

\subsection{Parametrizing ${\mathrm{O}}_{2r}(q)$-orbits} %Let $G={\mathrm{O}}_{2r}(q)$. 

The set $\sF_{n,r}^\jmath$ is a finite $G(q)$-set.
If $\sF_{n,r}^\jmath/G(q)$ denotes the set of $G(q)$-orbits in $\sF_{n,r}^\jmath$, then there is a surjective map
 onto the set $\widehat \La(n+1,r)$:
\begin{equation}\label{dimvec}
\aligned
&{\bf dim}: \sF_{n,r}^\jmath\longrightarrow \widehat \La(n+1,r),\;\; F\longmapsto{\bf dim}(F),\;\; \text{where}\\
{\bf dim}(F):=(&\dim F_1/F_0, \dim F_2/F_1,\ldots,
\dim F_n/F_{n-1},\dim F_{n+1}/F_n,\ldots,\dim F_N/F_{N-1})
\endaligned
\end{equation}
is called the {\it dimension sequence} (or {\it dimension vector}),  induces a bijective map from $\sF_{n,r}^\jmath/G(q)$ to $\widehat \La(n+1,r)$.
Thus, the $G$-orbits are indexed by the set $\La(n+1,r)$.

For each pair $(F,F')$ in $\sF_{n,r}^\jmath\times \sF_{n,r}^\jmath$, let $F_{i,j}=F_{i-1}+F_i\cap F'_j$ and $F'_{i,j}=F'_{i-1}+F'_i\cap F_j$. Then the filtration $F$ is refined to 
\begin{equation}\label{Fij}
\aligned
%F_{\!\centerdot\centerdot}
(F, F')_{\!\centerdot\centerdot}
:\;\;0&=F_{1,0}\subseteq F_{1,1}\subseteq F_{1,2}\subseteq\cdots\subseteq F_{1,N}(=F_1\\
&=F_{2,0})\subseteq F_{2,1}\subseteq F_{2,2}\subseteq\cdots\subseteq F_{2,N}(=F_2\\
&\qquad\qquad\cdots\cdots\\
&=F_{N,0})\subseteq F_{N,1}\subseteq F_{N,2}\subseteq\cdots\subseteq F_{N,N}(=F_N).\endaligned
\end{equation}
%A similar refinement $F_{\!\centerdot\centerdot}'$ for $F'$ can be defined by swapping the roles of $F$ and $F'$. 
For simplicity, we write $F_{\!\centerdot\centerdot}$ for $(F, F')_{\!\centerdot\centerdot}$
and $F'_{\!\centerdot\centerdot}$ for $(F', F)_{\!\centerdot\centerdot}$.

\begin{lem}\label{N^2} 
The subspace filtration \eqref{Fij} is an $n(N+1)$-step isotropic (partial) flag. In other words,
$F_{\!\centerdot\centerdot}\in\sF^\jmath_{n(N+1),r}$.
\end{lem}
\begin{proof}If we set $F_{\!\centerdot\centerdot}=(V_1\subseteq V_2\subseteq\cdots\subseteq V_{N^2})$, then $F_{i,j}=V_{(i-1)N+j}$. We need to prove that $(F_{i,j})^\perp=V_{N^2-(i-1)N-j}=F_{N-i+1,N-j}$ since
$N^2-(i-1)N-j=(N-i)N+N-j$. This can be checked easily. 

We now prove that all $V_i$, for $1\leq i\leq n(N+1)$, are isotropic subspaces of $\mathbb F_q^{2r}$. Since 
$F_{i, j}\subseteq V_i$, for all $i\in[1,n]$, are clearly isotropic, it remains to prove that $F_{n+1,j}$, for $j\in[1,n]$ are isotropic. This is because
$(F_n+(F_{n+1}\cap F'_j))^\perp=
F_n^\perp\cap(F_{n+1}^\perp+F_j^{\prime\perp})=F_{n+1}\cap(F_n+F'_{N-j})=F_n+F_{n+1}\cap F'_{N-j}\supseteq F_{n+1,j}.$
\end{proof}

The dimension vector {\bf dim} associated with $F_{\!\centerdot\centerdot}$ (see \eqref{dimvec}) 
gives rise to a dimension matrix
%defines a matrix
\begin{equation}\label{m(F,F')}
\fkm(F,F'):=
{\bf dim} (F, F')_{\!\centerdot\centerdot
} =
(a_{i,j}), \quad\text{where } a_{ij} = \dim (F_{i,j}/F_{i,j-1}).
\end{equation}
Note that $\dim (F_{i,j}/F_{i,j-1})=\dim (F'_{j,i}/F'_{j,i-1})$, which 
 follows from the vector space isomorphism 
 $$
 F_{i,j}/F_{i,j-1}\cong\frac{F_i\cap F_j^{'}}{F_{i-1}\cap F_j^{'}+F_i\cap F_{j-1}^{'}}\cong F'_{j,i}/F'_{j,i-1}
 $$ 
 by the Zassenhaus Lemma. 

For $A\in\Xi_{N,2r}$ with $\wla=\rm{ro}(A)$ and $\wmu=\co(A)$, we partition $[1,2r]$ into $N^2$ subsets $I_{j,l}$ for $1\leq j,l\leq N=2n+1$,
where 
\begin{equation}\label{Ijl}
I_{j,l}=I_{j,l}(A):=\big[\widetilde a^r_{j,l-1}+1, \widetilde a^r_{j,l-1}+a_{j,l}\big]=R^\wla_j\cap d_AR^\wmu_l,
\end{equation} 
where $\widetilde a^r_{j,l-1}$ is defined in \eqref{r(A)} (and \eqref{par sum}).

\begin{lem}\label{MforB}
(1) If $w\in W^\scb$ is the permutation of $[1, 2r]$ sending $k$ to $i_k$ and suppose that the complete flags $(\sfF,\sfF^w)\in\mathcal B_{r}\times \mathcal B_{r}$ have the form $\sfF_j=\lr{v_1,\ldots,v_j}$ and $\sfF^w_j=\lr{v_{i_1},\ldots,v_{i_j}}$, for $1\leq j\leq 2r$, then $\fkm(\sfF,\sfF^w)=\dot w$.

(2) For $A\in\Xi_{N,2r}$ with $\la=\rm{ro}(A)$, suppose that $v_1,v_2,\ldots,v_{2r}$ form a basis for $\mathbb F_q^{2r}$ and the partial flag $F\in \mathcal F_{n,r}^\jmath$ is defined by
$F_i=\lr{v_k\mid k\in\bigcup_{j\in[1,i]}\bigcup_{l\in[1,N]}I_{j,l}}\;\;(1\leq i\leq N).$
If we define $F^A\in \mathcal F_{n,r}^\jmath$ by setting $F^A_i=\lr{ v_k\ | \  k\in \bigcup_{j\in[1, N]}\bigcup_{l\in[1,i]}I_{j,l}}$, for $1\leq i\leq N$,
% $$\aligned
% F'_1&=\lr{v_k\mid k\in[1,a_{1,1}]\cup[\widetilde\la_1+1,\widetilde\la_1+a_{2,1}]\cup\cdots\cup[\widetilde \la_{N-1}+1,\widetilde\la_{N-1}+a_{N,1}]},\\ \endaligned,$$
 then $\fkm(F,F^A)=A$. In particular, if $\wla=\ro(A)$, $\wmu=\co(B)$, and $d_A\in\sD_{\la,\mu}$, then  $(F^\wla,\dot {d}_AF^\wmu)\in\sO_A$, where $F^\wnu$ denotes the standard flag defined in $\eqref{std flag}$.
 \end{lem}
\begin{proof} (1)  By the transitivity of the action of $\O_{2r}(q)$ in $\sB_r$,  
there exists a $g\in\O_{2r}(q)$ such that $(\sfF,\sfF')=g.(\sfF^{s},\dot w\sfF^s)$, for some $w\in W$ (see \eqref{BwB}), where $\sfF^{\sf s}$ is the standard flag, i.e., $\sfF^{s}_i=\lr{e_1,\ldots,e_i}$. We may simply  take $(\sfF,\sfF^w)=(\sfF^{s},\dot w\sfF^s)$ and assume $v_i=e_i$ for all $i\in[1,2r]$. %Thus, $\sfF'=g.\sfF^{s\prime}$.

We first prove that $\sfF^w=(\sfF^{\sf s})^w$ is isotropic.
For any $1\leq m \leq 2r,$ by the definition of $w\in W$ (see \eqref{permB}), we have $i_m+i_{2r+1-m}=2r+1$ which, by \eqref{formJ}, implies that
$e_{i_m}\perp e_{i_p}$ for all $p\in[1,2r]\setminus\{2r-m+1\}$.
Thus, $e_{i_m}\perp e_{i_k}$ for all $m\in[1,j]$ ($j\leq r$) 
and $k\in[1,2r-j]$ since $m+k<2r+1$.
Hence,  ${\sfF^{w}_j}^\perp\supseteq \sfF^{w}_{2r-j}$.
Conversely, %to prove ${F^\prime_j}^\perp\subseteq F^\prime_{2r-j}$, Assume 
suppose there is a nonzero element $v=\sum_{m\in[1,2r]}c_me_{i_m}\in {\sfF^{w}_j}^\perp\setminus \sfF^{w}_{2r-j}$. Then there must exist some $k\geq 2r+1-j$ such that $c_k\neq0$. Thus, $e_{i_{2r+1-k}}\in \sfF^{w}_{j}$ and
$$\lr{v,e_{i_{2r+1-k}}}=c_k\lr{e_{i_k},e_{i_{2r+1-k}}}\neq 0$$
 which contradicts to $v\in {\sfF^{w}_{j}}^\perp$. Hence, we must have ${\sfF^{w}_j}^\perp\subseteq \sfF^{w}_{2r-j}$, proving ${\sfF^{w}_j}^\perp=\sfF^{w}_{2r-j}$. So, $\sfF'$ is an isotropic flag.

By the definition of $\fkm(\sfF,\sfF^\prime)=A=(a_{k,l})$, we have $a_{k,l}= \dim\frac{\sfF_{k}\cap \sfF^\prime_l}{\sfF_{k-1}\cap \sfF^\prime_l+\sfF_{k}\cap \sfF^\prime_{l-1}}$. Thus,
\begin{equation}\nonumber\begin{aligned}
   a_{k,l}&=\dim\frac{\lr{v_1,\cdots,v_k}\cap \lr{v_{i_1},\cdots,v_{i_l}}}{\lr{v_1,\cdots,v_{k-1}}\cap \lr{v_{i_1},\cdots,v_{i_l}}+\lr{v_1,\cdots,v_k}\cap \lr{v_{i_1},\cdots,v_{i_{l-1}}}} \\
&=\begin{cases}
    \dim\frac{\lr{v_1,\cdots,v_{k-1}}\cap \lr{v_{i_1},\cdots,v_{i_l}}+\lr{v_1,\cdots,v_k}\cap \lr{v_{i_1},\cdots,v_{i_{l-1}}}}{\lr{v_1,\cdots,v_{k-1}}\cap \lr{v_{i_1},\cdots,v_{i_l}}+\lr{v_1,\cdots,v_k}\cap \lr{v_{i_1},\cdots,v_{i_{l-1}}}},&\text{ if }k\neq i_l,\\
    \\
    \dim\frac{\lr{v_1,\cdots,v_{k-1}}\cap \lr{v_{i_1},\cdots,v_{i_{l-1}}}+\lr{v_k}}{\lr{v_1,\cdots,v_{k-1}}\cap \lr{v_{i_1},\cdots,v_{i_{l-1}}}}, &\text{ if }k=i_l,
\end{cases}\\
&=\delta_{k,i_l},
\end{aligned}\end{equation}which is the same as the matrix $\dot w$.

(2)  
Associated with  $A\in\Xi_{N,2r}$, there is a permutation, a distinguished double coset representative, sending $k$ to $i_k$,
where $(i_1,i_2,\cdots,i_{2r})$ is the permutation of $1,2,\ldots, 2r$ obtained by concatenating $I_{1,1}$ with $I_{2,1}$, down column 1 and then down column 2, and so on. See Lemma \ref{map kap} and \cite[Exer.8.2(1)]{DDPW}. Then, by definition, $F^A_i=\langle v_{i_1}, \cdots, v_{i_{\tilde{\mu}_i}}\rangle$, where $\mu=(\mu_1,\ldots,\mu_N)=\text{co}(A)$. By part (1), $F^A$ is isotropic as it can be obtained by dropping some steps from a corresponding complete isotropic flag.

Now, $F_k\cap F_l^{'}$ is spanned by $\scb_{k,l}:=\{v_k\mid k\in \bigcup_{i\in[1,k],j\in[1,l]} I_{i,j}\}$. Thus,
 $F_{k-1}\cap F_l^{'}+F_k\cap F_{l-1}^{'}$ has basis $\scb_{k-1,l}\cup \scb_{k,l-1}$. Hence,
 $\scb_{k,l}\setminus(\scb_{k-1,l}\cup \scb_{k,l-1})=\{v_k\mid k\in I_{k,l}\}$. Consequently,
 $a_{k,l}=\dim\frac{F_k\cap F_l^{'}}{F_{k-1}\cap F_l^{'}+F_k\cap F_{l-1}^{'}}$. The last assertion follows from Lemma \ref{urcm}. 
\end{proof}
We observe from the proof that the $n(N+1)$-step isotropic flag $F_{\!\centerdot\centerdot}$ associated with $(F,F')$ has the following filtration terms
\begin{equation}\label{explicitFij}
F_{i,j}=F_{i-1}+F_i\cap F'_j=\lr{v_1,v_2,\cdots, v_{\widetilde a^r_{i,j}}}\end{equation}
ordered on the lexicographic ordering of $[1,N]\times[1,N]$.

\begin{example}Let 
$$A=\begin{pmatrix}1&2&3\\ 3&0&{\bf3}\\{\bf3}&{\bf2}&{\bf1}\end{pmatrix}.\quad
$$
Then $n=1$, $N=3$, $r=9$, $\wla=\ro(A)=(6,6,6),\wmu=\co(A)=(7,4,7)$, and $d_A=\big(\begin{smallmatrix}1&2&3&4&5&6&7&8&9&\cdots\\1&7&8&9&13&14&15&2&3&\cdots
\end{smallmatrix}\big).$

Breaking the natural basis $(e_1,e_2\ldots,e_{18})$ into subsequences of lengths 
$a_{1,1},a_{1,2},a_{1,3}, a_{2,1},\ldots,a_{3,3}$ and form
the matrix of the subsequences
$$A(e_1,\ldots,e_{18})=\begin{pmatrix}e_1&(e_2,e_3)&(e_4, e_5,e_6)\\(e_{7},e_{8},e_9)&-&(e_{10},e_{11},e_{12})\\
(e_{13},e_{14},e_{15})&(e_{16},e_{17})&e_{18}\\\end{pmatrix}$$
Then $F=F^\wla$ with $F^\wla_i$ being spanned by all rows $j$, $1\leq j\leq i$, and $F'=d_AF^\wmu$ with $F'_i$ being spanned by all columns $j$, $1\leq j\leq i$: 
$$\aligned
F_1&=\lr{e_1,\ldots,e_6}, \;\;F_2=F_1+\lr{e_7,\ldots,e_{12}},\;\;F_3=F_2+\lr{e_{13},\ldots,e_{18}},\\  
F_1'&=\lr{e_1,e_7,e_8,e_9,e_{13},e_{14},e_{15}}, \;\;F'_2=F_1'+\lr{e_2,e_3,e_{16},e_{17}},\;\;F'_3=F'_2+\lr{e_4,e_5,e_6,e_{10},e_{11},e_{12},e_{18}},\endaligned$$
and $F_{\!\centerdot\centerdot}$ has the form
 $$F_{i,j}=F^\wla_{i-1}+\text{subspaces spanned by subsequences of lengths $a_{i,1},\ldots,a_{i,j}$}=\lr{e_1,e_2,\ldots,e_{\widetilde a^r_{i,j}}}.$$
%The general case is similar.
\end{example}

For vector spaces $U,V$,  $n$-step isotropic flags $F,F^\prime \in\mathcal F_{n,r}^\jmath$, and $h\in[1,n]$, we set
\begin{equation}\label{UsubsetV}
\aligned
U\overset a\subseteq V\;\text{ or }\;V\overset a\supseteq U
&\iff U\subseteq V,\;\text{ and }\;\dim V/U=a, \\
F\overset 1\subset_hF'\;(\text{resp.}, F\overset 1\supset_hF')&\iff F_h\stackrel{1}{\subset }{F'}_{h}\;(\text{resp.},F_h\stackrel{1}{\supset }{F'_h}),\; F_i=F'_i,\;\; \forall i\in[1,n]-\{ h\}.
\endaligned
\end{equation}

We set $E_{i j}^{\theta}=E_{i j}+E_{N+1-i, N+1-j},$ where $E_{i j}$ is the $N\times N$ matrix whose $(i,j)$-entry is $1$ and all other entries are 0. 

\begin{cor}\label{Cormulti}Let $E,F,F'\in\mathcal F_{n,r}^\jmath$ and $h\in[1,n]$. 
We have 
$$F\overset 1\subset_hE\iff \fkm(F, E)-E^\th_{h+1,h}=\diag(\wga),\text{ for some }\ga\in\La(n+1,r).$$
 Moreover, if $F\overset 1\subset_hE$, then  $\fkm(E, F')=A\implies\fkm(F, {F'})=A-E_{h,p}^\theta+E^{\theta}_{h+1,p}$, for some $p\in[1,N]$. Conversely, if $\fkm(F, {F'})=A-E_{h,p}^\theta+E^{\theta}_{h+1,p}=A'$, then there exists $E\in\sF^\jmath_{n,r}$ such that $F\overset 1\subset_hE$ and $\fkm(E,F')=A'+E^\th_{h,p}-E^\th_{h+1,p}=A$.

Here $p$ satisfies the conditions  $F_{h}\cap F_j'=E_h\cap F_j'$ for $j<p$ and $F_h\cap F_j'\neq E_h\cap F_j^{'}$ for $j\geq p$.
\end{cor}
\begin{proof} The first assertion follows from Lemma~\ref{MforB}(2) immediately.%, noting
%  $$F\overset 1\subset_hE\iff F_{\!\centerdot\centerdot}:{{\cdots\subseteq F_{h,h-1}=F_{h-1}\subseteq F_{h,h}=F_h\subseteq F_{h,h+1}=F_h\subseteq}\cdots\atop{\cdots\subseteq F_{h+1,h-1}=F_{h}\subseteq F_{h+1,h}=E_{h}\subseteq F_{h+1,h+1}=F_{h+1}\subseteq }}$$

For the second assertion, let $A'=\fkm(F, {F'})$.  We want to prove that $A'=A-E_{h,p}^\theta+E^{\theta}_{h+1,p}$. By the condition $F\overset 1\subset_hE$, we may assume that there exists a basis $\sB=\{v_1,\ldots,v_{2r}\}$ for $\mathbb F_q^{2r}$ such that the isotropic subspaces have the form $F_j=E_j=\lr{v_{1}, \cdots, v_{{{\tla}_j}}}$, for all $j\in[1,n],j\neq h$, $F_h=\lr{v_{1}, \cdots, v_{{{\tla}_i}}}$, and $E_h=\lr{v_{1}, \cdots, v_{{{\tla}_i}},v_{{{\tla}_i}+1}}$, where $\la={\bf dim} F$ (see \eqref{dimvec}). 

Suppose $F'=E^A$. Then $a_{kl}=\dim\frac{E_k\cap F_l'}{E_{k-1}\cap F_l'+E_k\cap F_{l-1}'}$ for all $k,l$. We compute
$a_{kl}'=\dim\frac{F_k\cap F_l'}{F_{k-1}\cap F_l'+F_k\cap F_{l-1}'}$. Clearly, we have
\begin{equation}\label{xxx}a_{kl}=a'_{kl}\;\;\text{ for all $l$ and $k\neq h,h+1,N-h,N+1-h$}.
\end{equation}

With the notation in the proof of Lemma \ref{MforB}, we have bases $\scb_{k,l}\subseteq\sB$ for $E_k\cap F'_l$ and $\scb'_{k,l}\subseteq\sB$ for $F_k\cap F'_l$. Then 
\begin{equation}\label{yyy}
\aligned a'_{h,j}&=|\scb'_{h,j}|-|\scb'_{h,j-1}|-|\scb'_{h-1,j}|+|\scb'_{h-1,j-1}|\\
&=|\scb_{h,j}|-|\scb_{h,j-1}|-|\scb_{h-1,j}|+|\scb_{h-1,j-1}|+|\scb'_{h,j}|-|\scb'_{h,j-1}|-|\scb_{h,j}|+|\scb_{h,j-1}|\\
&=a_{h,j}+(|\scb'_{h,j}|-|\scb_{h,j}|)-(|\scb'_{h,j-1}|-|\scb_{h,j-1}|)\;\;\;\text{and similarly,}\\
a'_{h+1,j}&=a_{h+1,j}-(|\scb'_{h,j}|-|\scb_{h,j}|)+(|\scb'_{h,j-1}|-|\scb_{h,j-1}|).
\endaligned
\end{equation}

By the hypothesis $F\overset 1\subset_hE$,
there exists $p\in[1,N]$ such that $F_h\cap F_j'=E_h\cap F_j'$ when $j<p$ and $F_h\cap F_j'\neq E_h\cap F_j^{'}$ when $j\geq p$. Then 
\begin{equation}\label{zzz}
|\scb_{h,j}|-|\scb'_{h,j}|=\begin{cases}0,&\text{ if } j<p;\\
1,&\text{ if } j\geq p.\end{cases}
\end{equation}
(Replacing 1 by $-1$ gives the relation under the condition $F\overset 1\supset_hE$.) Substituting \eqref{zzz} into \eqref{yyy} together with \eqref{xxx} gives $A'=A-E_{h,p}^\theta+E^{\theta}_{h+1,p}$.

Conversely, let $F$ be given as above and suppose $F'=F^{A'}$ with $A'=\apu$. We want to construct an $E\in\sF^\jmath_{n,r}$ such that $F\overset 1\subset_hE$ and $\fkm(E,F')=A'+E^\th_{h,p}-E^\th_{h+1,p}=A$. Let $I'_{j,l}=I_{j,l}(A')$
  as in \eqref{Ijl}. Then $I'_{n+1,p}\neq\emptyset$ as $a'_{h+1,p}>0$. If we choose any $m\in 
I'_{h+1,p}$, then necessarily, by symmetry, ${2r+1-m}\in I'_{N-h,N+1-p}$.  Define 
$$\aligned
I_{h,p}&=I'_{h,p}\cup\{m\},\qquad I_{N+1-h,N+1-p}=I'_{N+1-h,N+1-p}\cup\{{2r+1-m}\};\\
I_{h+1,p}&=I'_{h+1,p}\setminus\{m\},\quad I_{N-h,N+1-p}=I'_{N-h,N+1-p}\setminus\{{2r+1-m}\}\;\;(I_{j,l}=I'_{j,l}, \text{ for other }j,l),
\endaligned
$$
and set $E_i=\lr{v_j\mid j\in\cup_{k=1}^nI_{i,k}}$. Then $E_h=F_h+\lr{v_m}$ and $E_i=F_i$ for other $i$. One checks easily that $E$ is isotropic using the fact $\lr{v_m,v_{2r+1-m}}_\sfJ\neq0$. Finally, by an argument around \eqref{xxx}, \eqref{yyy}  and \eqref{zzz}, one checks easily that $\fkm(E,F')=A$.
%This is clear since is modified by changing 1 to $-1$. The rest is the same.
\end{proof}

Let $(\mathcal F_{n,r}^\jmath\times \mathcal F_{n,r}^\jmath)/G(q)$ %(reps., $(\sB_r\times \sB_r)/{{\mathrm{O}}_{2r}(q)}$)
 denote the set of $G(q)$-orbits in $\mathcal F_{n,r}^\jmath\times \mathcal F_{n,r}^\jmath$. %(resp., $\sB_r\times\sB_r$). 
Then, we obtain a map
\begin{equation}\label{fkm}\fkm: (\mathcal F_{n,r}^\jmath\times \mathcal F_{n,r}^\jmath)/{{\mathrm{O}}_{2r}(q)}\longrightarrow\Xi_{N,2r},\;\;G(q).(F,F')\longmapsto
\fkm(F,F').
\end{equation} 
The ``equal parameter'' counterpart of the following result can be found in \cite[Lem.~2.1]{BKLW}.
\begin{cor} \label{bij fkm}
The map $\fkm$ is bijective.
\end{cor}
\begin{proof}The following proof modifies the proof of \cite[Lem.~1.2]{BKLW}. First, the proof
for a well-defined $\fkm$ uses the argument there. The surjectivity follows from part (2) of the lemma above. 
Finally, the injectivity of $\fkm$ follows from Lemma \ref{scott}(2) (with $X=\sX$) and the bijection \eqref{fkd}, noting the relation
between $n$-step isotropic flags and parabolic subgroups introduced in the next subsection.
\end{proof}
%In this way, we obtain a bijection
%$$\fkm: {{\mathrm{O}}_{2r}(q)}\setminus(\mathcal F_{n,r}^\jmath\times \mathcal F_{n,r}^\jmath)\longrightarrow\Xi_{N,2r},\;\;G.(F,F')\longmapsto\fkm(F,F').$$ 

For an orbit $(F,F')\in\mathcal F_{n,r}^\jmath\times \mathcal F_{n,r}^\jmath$, if $\fkm(F,F')=A$, we write $\sO_A:=G.(F,F')$. Similarly, for $(\sfF,\sfF')\in\sB_r\times\sB_r$,
if $\fkm(F,F')=\dot w$, we set $\sO_w=G(q).(F,F')$.

For any $A=(a_{i,j})\in\Xi_{N,2r}$, $h\in[1,n]$ and $p\in[1,N]$, define
\begin{equation}\label{hp}
\apl =A+E_{h,p}^\theta-E^{\theta}_{h+1,p}\;\;\text{ and }\;\;\apu=A-E_{h,p}^\theta+E^{\theta}_{h+1,p}.
\end{equation}Thus, $\apl,\apu\in\Xi_{N,2r}$ imply $a_{h+1,p}\geq1$ and $a_{h,p}\geq1$, respectively.
For later use, we record the following result which follows from Corollary \ref{Cormulti}.

\begin{prop}\label{new} 
Let $h\in[1,n]$, $p\in[1,N]$, $A,B,C\in\Xi_{N,2r}$ with $B-E^\th_{h,h+1}$ and $C-E^\th_{h+1,h}$ diagonal. % and assume $\wla=\ro(B), \wmu=\co(B)=\ro(A)$ and $\wnu=\co(A)$. Then $\la=\ro(\apl)=\ro(\apu)$ and $\nu=\co(\apl)=\co(\apu)$.
Let $(F,F')\in\sO_{A'}$ with $A'={\apl}$ (resp., $\apu$) and $\al=\ro(A')$. Suppose $\{v_1,v_2,\ldots,v_{2r}\}$ is a basis for $\mathbb F_q^{2r}$ such that $F_i=\lr{v_1,v_2,\ldots,v_{\widetilde\al_i}}$ for all $i\in[1,N]$ and $F'=F^{A'}$. Then the flag $E\in\sF^\jmath_{n,r}$, obtained from $F$ by moving a vector $v_m$ from $\{v_k\mid k\in I'_{h,p}\;(\text{resp., }I'_{h+1,p})\}$ to $\{v_k\mid k\in I'_{h+1,p}\;(\text{resp., }I'_{h,p})\}$ and moving $v_{2r+1-m}$ from $\{v_k\mid k\in I'_{N+1-h,N+1-p}\;(\text{resp., }I'_{N-h,N+1-p})\}$ to $\{v_k\mid k\in I'_{N-h,N+1-p}\;(\text{resp., }I'_{N+1-h,N+1-p})\}$, satisfies $\fkm(F,E)=B$ (resp., $C$) and $\fkm(E,F')=A$.
%\item[(2)] Let $(F,F')\in\sO_{A'}$ with $A'={\apu}$. Suppose $\{v_1,v_2,\ldots,v_{2r}\}$ is a basis for $\mathbb F^{2r}$ such that $F_i=\lr{v_1,v_2,\ldots,v_{\wla_i}}$ for all $i\in[1,N]$ and $F'=F^{A'}$. Then the flag $E\in\sF^\jmath_{n,r}$ obtained my moving a vector $v_m$ from $\{v_k\mid k\in I_{h+1,p}\}$ to $\{v_k\mid k\in I'_{h,p}\}$ satisfies $\fkm(F,E)=C$ and $\fkm(E,F')=A$.
%$E_i=F_i$ for $i\in[1,n],i\neq h$ and $E_h=\lr{v_1,v_2,\ldots,v_{\wla_h+1}}$. Then $\fkm(F,E)=C$ and $\fkm(F,F')=A$.
%\end{itemize} 
\end{prop}

\subsection{A geometric setting for {\rm$H^\scb_{\bsq,1}$}}

Let $G(q)={\mathrm{O}}_{2r}(q)$ and 
let $\fkF_{G(q)}(\mathcal B_{r}\times \mathcal B_{r})=\fkF_{G(q)}(\mathcal B_{r}\times \mathcal B_{r},\mathbb Z)$. Let $\tau_w$ be the orbital function associated with orbit $\sO_w$; see \eqref{orbifun}. Then,
for any $1\leq j\leq r$ and $\sfF,\sfF'\in\mathcal B_{r}$,  $\tau_j=\tau_{s_j}\in \fkF_{G(q)}(\mathcal B_{r}\times \mathcal B_{r})$ has the form
$$\tau_j(\sfF,\sfF^{'})=\begin{cases}
    1, &\text{ if $\sfF,\sfF'$ differ only at $j$-th (and $(2r-j)$-th) subspace};\\
    % F_i=F_i^{'} \quad \forall i\in [1,r] \setminus \{j\} \text{ and } F_j\neq F_j^{'};\\
    0, &\text{ otherwise.}
\end{cases}
$$
In other words, if two isotropic complete flags $\sfF$ and $\sfF'$ differ only at the $j$-th subspace (i.e, $\sfF_i=\sfF'_i$, for all $i\in[1,r]$,  $i\neq j$, and  $\sfF_j\neq \sfF'_j$), then the orbit $\sO_j=G(q).(\sfF,\sfF')$ has image $\fkp(\sfF,\sfF')=\dot s_j$, for all $j\in[1,r]$ (i.e., $\sO_j=\sO_{s_j}$). (Note that each $\sfF$ is an $r$-step isotropic filtration extended to a $2r$-step filtration.) 

We also write ${\bf 1}=\tau_w$, for $w=1$. This is the function associated with the orbit $\sO_1:=G(q).(\sfF,\sfF)$ for the identity $1\in W$.

Recall the Hecke algebra $H_{\bsq,1}=H_{\bsq,1}^\scb$ of type $B_r$ from \S2.
Let $H_{q,1}:=H_{\bsq,1}|_{\bsq=q}:=H_{\bsq,1}\otimes_\sA\mathbb Z$ be the $\mathbb Z$-algebra obtained by specializing $\bsq$ to the prime power $q$. Then 
$$H_{q,1}\cong H_{\bsq,1}/(\bsq-q)H_{\bsq,1}.$$
Note that $H_{q,1}$ can be presented by \eqref{rels} with $\bsq$ replaced by $q$. The following result can be proved directly, using convolution product, or via Lemma \ref{scott} (see \cite{Iw}).

\begin{prop} 
Let $G(q)={\mathrm{O}}_{2r}(q)$. There is a $\mathbb Z$-algebra isomorphism
$$H_{q,1}\cong\fkF_{G(q)}(\mathcal B_{r}\times \mathcal B_{r})$$
defined by sending $T_j$ to $\tau_j$, for all $j\in[1,r]$.
\end{prop} 

\begin{proof}
We only verify the degenerate relation $\tau_r^2=1$. Observe that,
for $\sfF,\sfF^\prime\in \mathcal B_{r}$,
$$\tau_r^2(\sfF,\sfF^\prime)=\sum_{{\sfF''}\in\mathcal B_{r}}\tau_r(\sfF,{\sfF''})\tau_r({\sfF''},\sfF^\prime)\neq0$$
implies $\exists \sfF''$ s.t. $(\sfF,\sfF'')\in\sO_{r}$ and $(\sfF'', \sfF')\in\sO_{r}$. Thus, $\sfF_r\neq \sfF''_r\neq \sfF'_r$ are all maximal isotropic. Suppose $\sfF_r=\sfF_{r-1}+\lr{u}$, $\sfF''_r=\sfF_{r-1}+\lr{v}$ and $\sfF'_r=\sfF_{r-1}+\lr{\al u+\be v}$. Then $\lr{u,u}_\sfJ=0=\lr{v,v}_\sfJ$ and $\lr{u,v}_\sfJ=1$. So $0=\lr{\al u+\be v,\al u+\be v}=2\al\be$, forcing $\be=0$. Hence, $\sfF=\sfF'$ and in this case $\tau_r^2(F,F)=1$.
Consequently, $\tau_r^2={\bf 1}.$
\end{proof}

%For any commutative ring $R$, this action induces a permutation representation over $R$ which is isomorphic to the induced representations $\text{Ind}_{P_{\la}(q)}^{G(q)}1_R$ of the trivial representation $1_R$ to $G(q)$ 

\subsection{A geometric setting for {\rm $S^\scb_{\bsq,1}(n,r)$}}%the $\bsq$-Schur algebra of type $B$}%$S^{\scB}_{\bsq,1}(n,r)$ and $S^\kappa_{\bsq}(n,r)$}

Let $S^\scb_{q,1}(n,r):=S^\scb_{\bsq,1}(n,r)|_{\bsq=q}=S^\scb_{\bsq,1}(n,r)\otimes_\sA\mathbb Z$ be the $\mathbb Z$-algebra obtained by specializing $\bsq$ to the prime power $q$. Then 
$$S^\scb_{q,1}(n,r):=S^\scb_{\bsq,1}(n,r)|_{\bsq=q}\cong S^\scb_{\bsq,1}(n,r)/(\bsq-q)S_{\bsq,1}(n,r).$$
Let $G(q)=\O_{2r}(q)$, $\sX=\mathcal F_{n,r}^\jmath$ and $\fkF_{G(q)}(\sX\times \sX)=\fkF_{G(q)}(\sX\times \sX,\mathbb Z)$. Recall the natural basis $\{e_A|_{\bsq=q}\mid A\in\Xi_{N,2r}\}$ for $S^\scb_{q,1}(n,r)$ and the orbital basis $\{f_\sO\}$ for $\fkF_{G(q)}(\sX\times \sX)$

\begin{thm} \label{geosettingB}
There is a $\mathbb Z$-algebra isomorphism
$$S^\scb_{q,1}(n,r)\cong\fkF_{G(q)}(\sX\times \sX),$$
sending $e_A|_{\bsq=q}$ to $f_{\sO_A}$ for all $A\in\Xi_{N,2r}$.
\end{thm} 
The proof is standard by Lemma \ref{scott}(2) and an argument similar to that of \cite[Th.~13.15]{DDPW}.
%\begin{proof}By Lemma \ref{scott}(2), we have $\End_{\mathbb ZG}(\mathbb Z \sX)^\op\cong\fkF_{G}(\sX\times \sX)$. Since, as $G$-module, 
%$$\mathbb ZX\cong\oplus_{\la\in\La(n+1,r)}\mathbb Z\sO_\la\cong \oplus_{\la\in\La(n+1,r)}\mathbb Z(G/P_\la).$$
%It suffices to prove that $S^\scb_{q,1}(n,r)\cong\End_{\mathbb ZG}(\oplus_{\la\in\La(n+1,r)}\mathbb Z(G/P_\la))^{\rm op}$. The remaining argument is similar to that of \cite[Th.~13.15]{DDPW}, using the BN-pair structure of $G$. 
%\end{proof}
%Note that the orbital basis $\{f_\sO\mid \sO\in\sX\times\sX/G\}$ for $\fkF_{G}(\sX\times \sX)$ is the specialisation of the natural basis for $S^\scb_{q,1}(n,r)$ given in Proposition \ref{stdbsB}.

\section{A natural basis and geometric setting for $S^\scd_\bsq(n,r)$}
We describe the natural (or double coset, or orbital) basis for the $\bsq$-Schur algebra of type $D$ according to its algebraic or geometric setting.

\subsection{The algebraic definition of the natural basis}Recall the Hecke algebra $H=H_{\bsq,1}$ of type $B_r$ and its subalgebra $\cH=\cH_\bsq$ of type $D_r$. By restriction, every $H$-module $M$ is an $\cH$-module, denoted by $M|_{\cH}$. For any $\al\in\La^\scd(n,r)$, let
$\cx_\al=\sum_{w\in\cW_\al}T_w$.
%Clearly, for any subset $I\subseteq \check S$, the parabolic subgroup $\check W_\la$, there exists  $\la\in\La(n+1,r)$ such that $\check W_I=\check W_\la$ (or $\tau\check W_\la$ if $t_r\in H$ but $s_{r-1}\not\in I$).

\begin{lem}
 For any $\la\in\La(n+1,r)$, there are $\cH$-module isomorphisms:
$$x_\la H|_{\cH} \cong \begin{cases} 
\cx_{\la^+}\cH\oplus \cx_{\la^-}\cH,&\text{ if }\la_{n+1}=0;\\
\cx_{\la^\oo}\check H,&\text{ if }\la_{n+1}\neq0.\end{cases}$$
%(2) For any $\la,\mu\in\La(n+1,r)$ and $d\in\sD_{\la,\mu}$, if $d\in\check W$, then $\cW_\la d\cW_\la$ and
%$\cW_\la s_r ds_r\cW_\mu$ are distinct double cosets; if $d\not\in\check W$, then $\cW_\la s_r d\cW_\mu,\cW_\la s_rd\cW_\mu$ are distinct double coset.
\end{lem}

\begin{proof} 
Since $W=\check W\cup s_r\check W$ is a disjoint union, we have $H=\cH\oplus T_r\cH$. %(see Lemma \ref{invo}(2)).
Thus,
%$W_\la=\cW_\la\cup s_r\cW_\la$ which is disjoint if $\la_{n+1}\geq1$.
$$x_\la H=\begin{cases}
x_{\la}\cH\oplus x_\la T_r\check H,&\text{ if }\la_{n+1}=0;\\
 x_{\la}\check H,&\text{ if }\la_{n+1}\neq0.\\\end{cases}$$
If $\la_{n+1}=0$, then $x_\la=\cx_{\la^+}$ and $T_r x_\la T_r=\cx_{\la^-}$ by Corollary \ref{par cor}(1)\&(2).
We have $x_\la \cH=\cx_{\la^+}\cH$ and $x_\la T_r\check H=T_r\cx_{\la^-}\cH\cong \cx_{\la^-}\cH$. If $\la_{n+1}\neq0$, by Corollary \ref{par cor}(3), we have either $x_\la=\cx_{\la^\oo}$ or  $x_\la=(1+T_r)\cx_{\la^\oo}$. In both cases, we have $x_\la H|_{\cH}\cong \cx_{\la^\oo}\cH$.
%$W_\la=\cW_{\la^\oo}\cup s_r\cW_{\la^\oo}$.
%(2) It is clear since
%$$d\in\sD_{\la,\mu}\implies ds_r(\check\Phi_\la)\subseteq 
%\begin{cases}s_rds_r,&\text{ if }d\in\check W;\\
%\check\Phi^+, s_rd(\check\Phi_\la)\subseteq \check\Phi^+,&\text{ otherwise.}\end{cases}$$
\end{proof}

\begin{prop}\label{kappa-D} 
There is an $\sA$-algebra  isomorphism
$S_\bsq^\kappa(n,r)\cong S_\bsq^\scd(n,r)$. By base change, it extends to a $\sZ$-algebra isomorphism
$\sS_\up^\kappa(n,r)\cong \sS_\up^\scd(n,r).$
\end{prop}

\begin{proof} By the above lemma, there is an $\cH$-module isomorphism
$$g:\bigoplus_{\la\in\La(n+1,r)}x_\la H|_{\cH}\longrightarrow \bigoplus_{\al\in\La^\scd(n,r)}\cx_\al\cH.$$
This $\cH$-module isomorphism induces the required $\sA$-algebra isomorphism
$$\widetilde g: S_\bsq^\scd(n,r)\longrightarrow S_\bsq^\kappa(n,r),\;\; \phi\longmapsto g^{-1}\phi g.$$
The proposition is proved.
\end{proof}

Similar to Corollary \ref{stdbsB}, the above isomorphism allows us to define the natural basis for $S_\bsq^\ka(n,r)$.

For $\al,\be\in\La^\scd(n,r)$, let $\csD_{\al,\be}$ denote the set of distinguished $\cW_\al$-$\cW_\be$ double coset representatives. Let
\begin{equation}\label{dnr}
\csD(n,r)=\{(\al,d,\be)\mid \al,\be\in\Lambdad, d\in\csD_{\al,\be}\}.
\end{equation}
For each $(\al,d,\be)\in\Lambdad$, define $\phi_{\al,\be}^d\in S_\bsq^\scd(n,r)$ by setting
$$\phi_{\al,\be}^d(\cx_\ga h)=\de_{\be,\ga}\sum_{w\in\cW_\al d\cW_\be}T_wh=\de_{\be,\ga}T_{\cW_\al d\cW_\be}h,\text{ for all $h\in\cH$.}$$

\begin{cor}\label{stdbsD} The set $\{\phi_{\al,\be}^d\mid (\al,d,\be)\in\csD(n,r)\}$ forms a basis for $S_\bsq^\scd(n,r)$.
We call this basis the {\it natural basis} for $S_\bsq^\scd(n,r)$. Moreover, the idempotents $1_\al$ considered in Lemma \ref{idemp} take the form $1_\al=\phi_{\al,\al}^1$.
\end{cor}

\subsection{Parametrising $\SO_{2r}(q)$-orbits}
We now construct the natural basis for $S_\bsq^\kappa(n,r)$ defined in \eqref{kappa} via the isomorphism in Proposition \ref{geosettingB}.

Let $\cG=\cG(q)=\SO_{2r}(q)$ (and $G=G(q)=\O_{2r}(q)$). We start with parametrizing $\cG$-orbits in the complete flag variety $\sB_r$ introduced in Section 3.  Let $\sfF=\sfFs$ be the standard complete flag with $\sfF_i=\lr{e_1,\ldots,e_i}$ and let 
$$\mathcal{B}_r^{+} :=\cG(q).\mathsf F^s,\qquad \mathcal{B}_r^{-}:= \cG(q)\dot s_r.\mathsf F^s.$$ 
If $M_r=\lr{e_1,e_2,\ldots,e_{r}}$ is the {\it standard maximal isotropic subspace}  in $\mathbb{F}_q^{2r}$ (of dimension $r$),
then 
\begin{equation}\label{max iso}
\mathcal{B}_r^{+}=\{\sfF\in\mathcal B_r\mid \dim(\sfF_r \cap M_r)\equiv r \text{ mod }2\},\quad
\mathcal{B}_r^{-}=\{\sfF\in\mathcal B_r\mid \dim(\sfF_r \cap M_r)\equiv r-1 \text{ mod }2\}.
\end{equation}
We have $\mathcal{B}_r =\mathcal{B}^{+}_r\sqcup\mathcal{B}_r^{-}$ and
$\mathcal{B}_r^{+}\cong \cG(q)/B(q),\; \mathcal{B}_r^{-}\cong \cG(q)/B(q)^{\dot s_r}$, where $B(q)=\text{Stab}_{\cG(q)}(\mathsf F^s)$ and
$B(q)^{\dot s_r}=\text{Stab}_{\cG(q)}(\dot s_r\mathsf F^s)$. 

To parametrize $\cG(q)$-orbits in the $n$-step isotropic (partial) flag variety $\sX:= \mathcal{F}_{n,2r}^\jmath $,
we use the $G(q)$-orbits.
Let $\sX/G(q)$ denote the set of all $G(q)$-orbits. Then \eqref{dimvec} implies a bijection
$$f: \sX/G(q)\longrightarrow \La(n+1,r),\;\;G(q).F\longmapsto \big(\dim(F_1/F_0),\ldots,\dim(F_n/F_{n-1}),\dim(F_{n+1}/F_n)/2\big).$$
Let $\sO_\la=f^{-1}(\la)$. Then $\sX$ is a disjoint union of orbits $\sX=\sqcup_{\la\in\La(n+1,r)}\sO_\la$.

\begin{lem}
An orbit $\sO_\la=G(q).F$ splits into two $\cG(q)$-orbits if and only if $\la_{n+1}=0$.
\end{lem}

\begin{proof}
Since $G(q)=\cG(q)\sqcup\cG(q)\dot s_r$, it follows that 
$G(q).F=\cG(q).F\cup \cG(q)\dot s_r.F$.
Thus, by Lemma \ref{lab},
$$\la_{n+1}>0\iff \dot s_r\in\text{Stab}_{G(q)}(F)\iff G(q).F=\cG(q).F.$$
Hence, $G(q).F=\cG(q).F\sqcup \cG(q)\dot s_r.F\iff \la_{n+1}=0$. The lemma follows.
\end{proof}

Let
$$\csO_{\la^\oo}=\sO_\la,\;\;\csO_{\la^+}=\cG(q).\sfF^s,\;\; \text{ and }\;\;\csO_{\la^-}=\cG(q)\dot s_r.\sfF^s.$$
By the Lemma, we may decompose $\sX$ into a disjoint union of $\cG(q)$-orbits:
\begin{equation}\label{SOorb}
\sX=\Big(\bigsqcup_{\la\in\La^\oo(n+1,r)} \csO_{\la^\oo}\Big)\sqcup\Big(\bigsqcup_{\la\in \La^\circ(n+1,r)} (\csO_{\la^+}\sqcup\csO_{\la^-})\Big).
\end{equation}

%for some $g\in {\mathrm{O}}_{2r}(q)\setminus \mathrm{SO}_{2r}(q)$. 
We now describe the decomposition of $\sX\times\sX$ into $\cG(q)$-orbits.

If $G(q).(F,F^{'})$ is a $G(q)$-orbit in $\sX\times \sX$, then it can be written as $G(q).(F,F^{'})=\cG(q).(F,F^{'})\cup\cG(q)g.(F,F^{'})$ for some $g\in G(q)\setminus \cG(q)$. 
Since the map $\fkm$ is a bijection (Corollary 5.7), we shall denote a $G(q)$-orbit by $ \mathcal O_A:=G(q).(F,F^{'})$ if $\fkm(F,F^\prime)=A$, for some $A\in \Xi_{N,2r}$.
%For any matrix $A=(a_{ij})_{1\leq i,j\leq N}\in\Xi_{N,2r}$ ($N=2n+1$), let % 
%\begin{equation}
%    \text{ro}(A)=\Big(\sum_{j\in[1,N]}a_{1,j},\ldots,\sum_{j\in[1,N]}a_{N,j}\Big),\quad 
%\text{co}(A)=\Big(\sum_{i\in[1,N]}a_{i,1},\ldots,\sum_{i\in[1,N]}a_{i,N}\Big).
%\end{equation}

\begin{prop}\label{split}
Let $A=(a_{ij}) \in \Xi_{N,2r}$. The $G(q)$-orbit $ \mathcal O_A$ splits into two $ \cG(q)$-orbits if and only if $a_{n+1,n+1}=0$.
\end{prop}
\begin{proof} Recall Lemma \ref{N^2} and consider the map
        \begin{equation}\label{phi}
           \Phi: \mathcal F_{n,r}^\jmath \times \mathcal F_{n,r}^\jmath \longrightarrow \mathcal F_{n(N+1),r}^\jmath,\ ({F}, F^{'})\longmapsto  F_{\!\centerdot\centerdot}=(F,F')_{\!\centerdot\centerdot}=( F_{i,j})_{1\leq i,j\leq N},
        \end{equation}where $F_{i,j}=F_{i-1}+ F_i\cap  F^{'}_j$ is defined as in \eqref{Fij}. This map is $G(q)$-equivariant since $\Phi(g(F,F^{'}))=g\Phi(F,F^{'})$ for any $g\in G(q).$  %
        %Also, \eqref{Fij} implies that this map is injective. 
            %  Moreover, $\mathcal F_{N^2,r}^\jmath(A)$ is isotropic since
  %  \nonumber\begin{align}
 %   {F}_{k,l}^{\perp}&=( {F}_k+{F}_{k+1}\cap {F^{'}_l})^{\perp}= {F}_k^{\perp}\cap ({F}_{k+1}^{\perp}+ {{F^{'}_l}}^{\perp})
%   = {F}_{N-k}\cap ({F}_{N-k-1}+ {{F}^{'}_{N-l}})\\
 %  &= {F}_{N-k-1}+({F}_{N-k}\cap {{F}^{'}_{N-l}})
 %  ={F}_{N-k-1, N-l}\end{align}
%   and $N^2-(kN+l)=(N-k-1)N+(N-l)$.   
         %  Suppose $\fkm(F,F^\prime)=A$ and let $\mathcal F_{n(N+1),r}^\jmath(A)=\Phi_A(\mathcal O_A)$. Restricting $\Phi$ to $\mathcal O_A$ induces an $G(q)$-equivariant map $\Phi_A: \mathcal O_A\to \mathcal F_{N^2,r}^\jmath(A).$ 
      Hence, 
$\text{Stab}_{G(q)}(F,F')$ is a subgroup of %$\text{Stab}_{G(q)}(\Phi(F,F'))=
$\text{Stab}_{G(q)}(\{F_{i,j}\}_{i,j}).$ Thus, if we choose $(F,F')\in\sO_A$ such that $F=F^\la$ and $F'=\dot dF^\mu$, where $\la=\co(A)$, $\mu=\co(A)$ and $d\in \sD_{\la,\mu}$,  then, with the notation in \S3.4, $\text{Stab}_{G(q)}(F^\la,\dot wF^\mu)=P_\la \cap{}^{\dot w}P_\mu$ and $\text{Stab}_{G(q)}(\{F_{i,j}\}_{i,j})=P_\nu$, where $\nu$ is the composition defined by $W_\nu=W_\la\cap{}^dW_\mu$. By \cite[Prop.~2.8.4]{Ct}, $U_\la(P_\la \cap{}^{\dot w}P_\mu)=P_\nu$, where $U_\la$ is a unipotent subgroup of $P_\la$ in the Levi decomposition of $P_\la=L_\la U_\la$.

%, where $F,F'$ are constructed  as in Lemma \ref{MforB}. %Without loss of generality, we may assume that $F$ is the standard flag defined by the natural basis $e_1,e_2,\ldots,e_{2r}$.

       If $a_{n+1,n+1}\neq 0,$ then $a_{n+1,n+1}$ is positive even and $a_{n+1,n+1}=|I_{n+1,n+1}|\geq 2$ by the notation in Lemma \ref{MforB}(2). Moreover, $r,r+1\in I_{n+1,n+1}$. Hence, $\dot s_r\in\text{Stab}_{G(q)}(\{F_{i,j}\}_{i,j})$ is not contained in $\cG(q)$. In other words, there exists $\dot s_r\in G(q)\setminus \cG(q)$ such that $\dot s_r\in \text{Stab}_{G(q)}(F,F')$. Hence, $\dot s_r.(F,F^{'})=(F,F')$ and so $\sO_A=\cG(q).(F,F^{'})$. 
       
       %Furthermore, $$a_{n+1,n+1}=\#\{ e_m|m\in \big[\tillam_{n}+\hspace{-3mm}\sum_{1\leq j \leq n}\hspace{-3mm}a_{n+1,j}+1, \tillam_{n}+\hspace{-3mm}\sum_{1\leq j \leq n}\hspace{-3mm}a_{n+1,j}+2,\cdots, \tillam_{n}+\hspace{-3mm}\sum_{1\leq j \leq n}\hspace{-3mm}a_{n+1,j}+a_{n+1,n+1}\big]\}\geq 2.$$  
      %On the other hand, $\sO_A=\cG(q).(F,F^{'})\sqcup\cG(q)\dot s_r.(F,F^{'})$.
 %The condition $r,r+1\in I_{n+1,n+1}$ implies $\dot s_r.(F,F^{'})=(F,F')$. Hence,  $\sO_A=\cG(q).(F,F^{'})$. 
       
       If $a_{n+1,n+1}= 0,$ then $\text{dim}({ F}_{n+1,n+1}/{F}_{n+1,n})=a_{n+1,n+1}=0$. Thus,
   %    Thus, any flag in $\mathcal F_{n(N+1),r}^\jmath(A)$ contains a step, e.g., $F_{n+1,n+1}$, which is maximal isotropic. In this case, 
       $$\text{Stab}_{G(q)}(F,F')\leq \text{Stab}_{G(q)}(\{F_{i,j}\}_{i,j})\subseteq \cG(q).$$ Hence, $\sO_A$ splits into two $\cG(q)$-orbits 
       $$\cG(q).(F,F')\;\;\text{ and }\;\;\cG(q)g.(F,F'),$$ where $g\in G(q)\setminus \cG(q)$. This finishes the proof.
 \end{proof}

Recall the standard maximal isotropic subspace $M_r$ used in \eqref{max iso}. By the proof above and the map $\Phi$ in \eqref{phi},  we define, for the orbit $\sO_A$ containing $(F,F')$,
\begin{equation}\label{d_FF'}
\aligned
\mathcal F_{n(N+1),r}^\jmath(A)&=\Phi(\mathcal O_A)=G(q).F_{\!\centerdot\centerdot}\\
d_{F,F'}=\dim&\big((F_n+F_{n+1}\cap F'_{n+1})\cap M_r\big).
\endaligned
\end{equation} Restricting the map $\Phi$ to $\mathcal O_A$ yields a surjective $G(q)$-equivariant map $\Phi_A: \mathcal O_A\to \mathcal F_{n(N+1),r}^\jmath(A).$

\begin{defn}\label{split2}
Suppose $A=(a_{i,j})\in \Xi_{N,2r}$ and $\sO_A=G(q).(F,F')$ for some $F,F'\in\sF_{n,r}^\jmath$ and $F_{\!\centerdot\centerdot}$ as in \eqref{Fij}.
\begin{itemize}
\item[(1)] If $a_{n+1,n+1}\neq 0$, define $\dot\sO_A:=\sO_A$ to be the non-split $G(q)$-orbit.
\item[(2)]
If $a_{n+1, n+1}=0$, then we label the two $\cG(q)$-orbits as follows.
Since $F_{n+1,n}=F_{n+1,n+1}$ is a maximal isotropic subspace, the $G(q)$-orbit $\mathcal F_{n(N+1), r}^{\jmath}(A)=G(q).F_{\!\centerdot\centerdot}$ splits into two $\cG(q)$-orbits
$$\mathcal F_{n(N+1), r}^{\jmath}(A)= \mathcal F_{n(N+1), r}^{\jmath}(A)^{+} \sqcup\ \mathcal F_{n(N+1), r}^{\jmath}(A)^{-}$$
where
 $\mathcal F_{n(N+1), r}^{\jmath}(A)^{\pm}=\{\Phi(F,F')\in \mathcal F_{n(N+1), r}^{\jmath}(A)\mid d_{F,F'}\equiv r^\pm\text{ mod } 2\}$, for $\begin{cases}r^+=r;\\
r^-=r-1.\end{cases}$
Define
$$
^{+} \hspace{-0.5mm}\mathcal{O}_A:=\Phi_A^{-1}\left(\mathcal F_{n(N+1), r}^\jmath(A)^{+}\right)\text{ and  }{}^{-} \hspace{-0.5mm}\mathcal{O}_A:=\Phi_A^{-1}\left(\mathcal F_{n(N+1), r}^{\jmath}(A)^{-}\right).$$
This decomposition $\mathcal{O}_A={^{+}}\hspace{-0.5mm}\mathcal{O}_A\sqcup {^{-}}\hspace{-0.5mm}\mathcal{O}_A$ is further refined to
\begin{equation}\label{O_A}
   \mathcal{O}_A=\begin{cases}
    ^{+}\hspace{-0.5mm}\mathcal O_A^{+}\sqcup {}^{-}\hspace{-0.5mm}\mathcal O_A^{-}, &\text{ if }\sgnura=+,\\
    ^{+}\hspace{-0.5mm}\mathcal O_A^{-}\sqcup {}^{-}\hspace{-0.5mm}\mathcal O_A^{+}, &\text{ if }\sgnura=-,
\end{cases}
\end{equation} 
where $      \sgnura=
        \begin{cases}+, &\text{ if $|\urcm|$ is even;}\\
        -, &\text{ otherwise.}
    \end{cases}$
    \end{itemize}
\end{defn}
We remark that there is an alternative interpretation for the above double signed labelling, using a  
similarly defined map for the transpose $A^t$ of $A$:
$$
\Phi'_{A}:\mathcal{O}_A \longrightarrow \mathcal F_{n(N+1), r}^{\jmath}(A^t),\ ({F}, F^{'})\longmapsto  F'_{\!\centerdot\centerdot}=\big(F'_{i,j}:=F^{'}_{i-1}+ F^{'}_i\cap  F_j\big)_{1\leq i,j\leq N}.$$
Thus, $\sF_{n(N+1),r}^\jmath(A^t)=G(q).F'_{\!\centerdot\centerdot}$ splits into two orbits 
$$\aligned
\mathcal F_{n(N+1), r}^{\jmath}(A^t)&= \mathcal F_{n(N+1), r}^{\jmath}(A^t)^{+} \sqcup \mathcal F_{n(N+1), r}^{\jmath}(A^t)^{-},
\endaligned
$$
where
 $\mathcal F_{n(N+1), r}^{\jmath}(A^t)^{\pm}=\{\Phi(F,F')\in \mathcal F_{n(N+1), r}^{\jmath}(A)\mid d_{F',F}\equiv r^\pm\text{ mod } 2\}$.
 Thus, $\mathcal{O}_A=\mathcal{O}_A^{+}\sqcup \mathcal{O}_A^{-},$ where
$$\mathcal{O}_A^{+}:=\Phi_{A}^{\prime-1}\left(\mathcal F_{n(N+1), r}^\jmath(A^t)^{+}\right)\text{ and }\mathcal{O}_A^{-}:=\Phi_{A}^{\prime-1}\left(\mathcal F_{n(N+1), r}^\jmath(A^t)^{-}\right). $$

\begin{lem}\label{O_A2}
Maintain the notations given above for $A$ with $a_{n+1,n+1}=0$. We have
$$ \begin{cases}
    \mathcal{O}_A^{+}={}^{+}\hspace{-0.5mm}\mathcal{O}_A,\ \mathcal{O}_A^{-}={}^{-}\hspace{-0.5mm}\mathcal{O}_A,  &\text{ if } \sgnura=+;\\
    \mathcal{O}_A^{+}={}^{-}\hspace{-0.5mm}\mathcal{O}_A,\ \mathcal{O}_A^{-}={}^{+}\hspace{-0.5mm}\mathcal{O}_A,  &\text{ if }\sgnura=-,
\end{cases}
$$
and hence, $^{\epsilon _1}\hspace{-0.5mm}\mathcal{O}_A^{\epsilon_2}:=^{\epsilon _1}\hspace{-1.5mm}\mathcal{O}_A\cap \mathcal{O}_A^{\epsilon _2}$, for the four selections of $\epsilon_1,\epsilon_2\in\{+,-\}$.
Moreover, the $\pm$ signs $\ep_1$ and $\ep_2$ are determined by $\ep_11=(-1)^{d_{F,F'}-r}$ and $\ep_21=(-1)^{d_{F',F}-r}$.
\end{lem}
\begin{proof}Let $\wla=\ro(A)$ and $\wmu=\co(A)$. We may assume $F=F^\wla$ and $F'=\dot d_AF^\wmu$ and define $F_{i,j}$ (resp., $F'_{i,j}$) with respect to $(F,F')$ (resp., $(F',F)$).
Then, by \eqref{explicitFij}, $F_{n+1,n+1}=M_r$, the standard maximal isotropic subspace. On the other hand, 
\begin{equation}\label{MFF'}
\aligned
M_{F',F}&:=F'_{n+1,n+1}\cap M_r=(F'_{n}+F'_{n+1}\cap F_{n+1})\cap M_r\\
&=\dot d_A\big((F'_n+F'_{n+1}\cap\dot d_A^{-1}F_{n+1})\cap \dot d_A^{-1} M_r\big)=\dot d_A\big(F'_{n+1,n+1}\cap \dot d_A^{-1} M_r\big)
\endaligned
\end{equation} 
has dimension congruent to $r$ (mod 2) if $\sgnura=+$, since $F'_{n+1,n+1}=M_r$. This proves  $\mathcal{O}_A^{+}={}^{+}\hspace{-0.5mm}\mathcal{O}_A$. The other cases can be proved similarly.
\end{proof}

For the convenience of type $D$ theory, we introduce the following ``signed matrices''. Let 
$$\hhh:=\{A\in\Xi_{N,2r}\mid a_{n+1,n+1}\neq0  \},\;\; \quad
 \overset\circ\Xi:=\{A\in\Xi_{N,2r}\mid a_{n+1,n+1}=0  \}$$
which partition $\Xi=\Xi_{N,2r}$ into a disjoint union $\Xi=\hhh\sqcup \overset\circ\Xi$. 
\begin{defn}\label{csO}
(1) For $A\in \overset\circ\Xi$ and $\epsilon_1,\epsilon_2\in\{+,-\}$, we use
 $^{\epsilon _1}\hspace{-0.5mm}A^{\epsilon_2}$ to label the $\cG(q)$-orbit $^{\epsilon _1}\hspace{-0.5mm}\mathcal{O}_A^{\epsilon_2}$ and write 
 \begin{equation}\label{SOorbit1}
 \csO( ^{\epsilon _1}\hspace{-0.5mm}A^{\epsilon_2}):=^{\epsilon _1}\!\!\mathcal{O}_A^{\epsilon_2};
 \end{equation}
 
 (2) For $A\in\hhh$, we use $\dot A$ to label the $\cG$-orbit ${\dot\sO}_A=\sO_A$ and write 
  \begin{equation}\label{SOorbit2}\csO(\dot A):=\dot{\sO}_A.\end{equation}
 \end{defn}
%Further, if we let
%$$\overset\circ\Xi_+:=\{A\in \overset\circ\Xi\mid |\urcm| \text{ is even}\},\qquad\qquad
%\overset\circ\Xi_-:=\{A\in \overset\circ\Xi\mid |\urcm| \text{ is odd}\},$$
%Then $\Xi=\hhh\sqcup \overset\circ\Xi_+\sqcup \overset\circ\Xi_-$. 
Hence, as an $\cG(q)$-set, $\sX\times \sX$ is a disjoint union of orbits:
\begin{equation}\label{SOorbit}\aligned
\sX\times \sX&=\Big(\bigsqcup_{A\in\hhh}\dot\sO_A\Big)\sqcup\Big(\bigsqcup_{A\in\overset\circ\Xi_+}({}^{+}\hspace{-0.5mm}\mathcal O_A^{+}\sqcup {}^{-}\hspace{-0.5mm}\mathcal O_A^{-})\Big)
\sqcup\Big(\bigsqcup_{A\in\overset\circ\Xi_-}({}^{+}\hspace{-0.5mm}\mathcal O_A^{-}\sqcup {}^{-}\hspace{-0.5mm}\mathcal O_A^{+})\Big)\\
&=\Big(\bigsqcup_{A\in\hhh}\csO(\dot A)\Big)\sqcup\Big(\bigsqcup_{A\in\overset\circ\Xi,\ep,\ep'\in\{+,-\}}\csO({}^\ep\!\!A^{\ep'})\Big).
\endaligned
\end{equation}

We further introduce the following notation. Recall the row/column sum vectors in \eqref{roco}.

\begin{defn}
    \label{defsepxinr} Let $\Xi:= \Xi_{N,2r}$ and,
 for any $A=(a_{i,j})\in\Xi$, let $\wla=\ro(A)$ and $\wmu=\co(A)$.% and $\kappa(\la,d_A,\mu)=A$. 

(1) We further partition $ \overset\circ\Xi$ into a disjoint union $ \overset\circ\Xi=\bbb\sqcup \bbh\sqcup\hbb\sqcup\hbh$, where
\begin{equation}\begin{aligned}\nonumber
  \bbb&=\{A\in\Xi_{N,2r}|\lambda_{n+1}=0, a_{n+1,n+1}=0  \text{ and } \mu_{n+1}=0 \},\\
    \bbh&=\{A\in\Xi_{N,2r}|\lambda_{n+1}=0, a_{n+1,n+1}=0  \text{ and } \mu_{n+1}\neq0 \},\\
    \hbb&=\{A\in\Xi_{N,2r}|\lambda_{n+1}\neq0, a_{n+1,n+1}=0  \text{ and } \mu_{n+1}=0 \},\\
    \hbh&=\{A\in\Xi_{N,2r}|\lambda_{n+1}\neq0, a_{n+1,n+1}=0  \text{ and } \mu_{n+1}\neq0 \}.
\end{aligned}\end{equation}
Thus, we have $\Xi_{N,2r}=\hhh\sqcup\bbb\sqcup \bbh\sqcup\hbb\sqcup\hbh$.% since neither $\lambda_{n+1}$ nor $\mu_{n+1}$ equals $0$ when $a_{n+1,n+1}\neq0$.

(2) For $\star, *\in \{\circ,\bullet\}$, let 
\begin{equation}\label{corner}
\aligned
\sbs\hspace{-0.7mm}_{\scriptscriptstyle{+}}&:=\{{\pap{A}},{\mam{A}} \mid A\in\sbs\text{ and }
\sgnura=+\},\quad\\
\sbs\hspace{-0.7mm}_{\scriptscriptstyle{-}}&:=\{{\pam{A}},{ \map{A}} \mid A\in\sbs\text{ and }\sgnura=-
\}, \\
\endaligned
\end{equation}
and
$$\hhhd:=\{\dot A\mid A\in\hhh\},\quad \overset{\circ}\Xi_+:=\bbb_{\!+}{\sqcup}\bbh_{\!+}\sqcup\hbb_{\!+}\sqcup \hbh_{\!+},\quad
\overset\circ\Xi_-:=\bbb_{\!-}{\sqcup}\bbh_{\!-}\sqcup\hbb_{\!-}\sqcup \hbh_{\!-}.$$
(Compare \cite[(A.4.2)]{LL}. Define
\begin{equation}\label{xid}
\xid:=\hhhd \mathop{\sqcup}(\bbb_{\!+}{\sqcup}\bbb_{\!-})\sqcup (\bbh_{\!+}\sqcup \bbh_{\!-})\sqcup (\hbb_{\!+}\sqcup \hbb_{\!-})\sqcup (\hbh_{\!+}\sqcup \hbh_{\!-})=\hhhd\sqcup\overset\circ\Xi_+\sqcup \overset\circ\Xi_-.
\end{equation}
\end{defn}

\begin{rem} Consider the diagonal embedding $\sX\to\sX\times\sX, F\mapsto (F,F)$. This map is $\cG(q)$-equivariant as well as $G(q)$-equivariant. Thus, the parametrization in \eqref{SOorb} is compatible with that in Definition \ref{split2}. In other words, it induces injective maps
$$\La(n,r)\longrightarrow\Xi_{N,2r}, \la\longmapsto D_\la:=\text{diag}(\la)\;\;\text{ and }\;\;
\La^\scd(n,r)\longrightarrow\check \Xi(n,r).$$
The latter sends, for $\ep\in\{+,-\}$, $\la^\ep$ to $^\ep D_\la^\ep$ if $\la_{n+1}=0$, or $\la^\bullet$ to $\dot D_\la$ if $\la_{n+1}\neq0$.
\end{rem}

\subsection{The geometric definition of the natural basis}
Recall from Lemma \ref{map kap} the bijective $\mathfrak d(\lambda,d,\mu)=(|R^\lambda_i\cap dR^\mu_j|)_{ij}$. If $\mathfrak d(\lambda,d,\mu)=A$, then we write $\mathfrak d^{-1}(A)=(\lambda,d_A,\mu)$, where $\lambda=\ro(A)$, $\mu=\co(A)$. Note that  $d_A\in\cW$ if and only if the entries of the upper right corner matrix $\urcm$ sum to an even number.

We now extend the map $\fkd$ to the type $D$ case. Observe first that, for $A\in\Xi$ with $a_{n+1,n+1}\neq0$, 
the associated distinguished double coset representative $d_A$ fixes $r$ and $r+1$. This implies that $s_rd_As_r=d_A$.
%We will keep using these notations.
%Now we could give the definition of matrix in type D.
\begin{defn}\label{defkd}
Define a map $\eta:\xid \rightarrow \dnr$ by the following rules.  %, we will give the definition on each disjoint subset of $\xid$.
For $A\in\Xi_{N,2r}$ such that $\fkd(\la,d_A,\mu)=A$ (so $\la=\ro(A)$ and $\mu=\co(A)$), let $W_\la d_AW_\mu$ be  the associated double coset 
    in $W=\WB$. We look for the corresponding double cosets in $\cW=W^\scd$ associated with $W_\la d_AW_\mu$.% and 
%    \begin{equation}\label{ladmu} \aligned W_\la d_AW_\mu&=(\cW_\la\cup s_r\cW_\la)d_A(\cW_\mu\cup \cW_\mu s_r)\\  &=(\cW_\la d_A\cW_\mu)\cup
%    (s_r\cW_\la d_A\cW_\mu s_r)\cup (\cW_\la d_A\cW_\mu s_r)\cup (s_r\cW_\la d_A\cW_\mu)
 %   \endaligned.\end{equation}
\begin{itemize}
      \item[(0)] If $A\in\hhh$, then we set  $\eta(\dot A)=(\lamz,d, \muz)$ where 
$$d=\begin{cases}
    \begin{aligned}
      &d_A, &\text{ if }  d_A\in \check W,\\
      &\text{the distinguished double coset representative for }\check W_{\lambda^\oo} s_rd_A \check W_{\mu^\oo}, &\text{ if } d_A\notin \check W.
    \end{aligned}
\end{cases}$$  
    \item[(1)] If $A\in\bbb$, then, for $d_A\in\cW$,   we set $\eta({\pap{A}})=(\lamp,d_1,\mup)$ and $\eta({\mam{A}})=(\lamm,d_2,\mum)$, where $d_1=d_A$ and $d_2$ is the distinguished  double coset representative in $\check W_{\lambda^{-}} s_rd_As_r\check W_{\mu^{-}}=s_r(\cW_\la d_A\cW_\mu) s_r$, and, for $d_A\not\in\cW$, we set $\eta({\pam{A}})=(\lamp,d_1,\mum)$ and $\eta({\map{A}})=(\lamm,d_2,\mup)$, where   $d_1$ and $d_2$ is the distinguished  double coset representatives in $\check W_{\lambda^{+}} d_As_r\check W_{\mu^{-}}$ and $\check W_{\lambda^{-}} s_rd_A\check W_{\mu^{+}}$, respectively.
     \item[(2)] If $A\in\bbh$ then,  for $d_A\in\cW$,  we set $\eta({\pap{A}})=(\lamp,d_1,\muz)$ and $\eta({\mam{A}})=(\lamm,d_2,\muz)$, where $d_1=d_A$ and $d_2$ is the distinguished  double coset representative for $\check W_{\lambda^{-}} s_rg_As_r\check W_{\mu^{\oo}}$, and,
     for $d_A\not\in\cW$, we set $\eta({\pam{A}})=(\lamp,d_1,\muz)$ and $\eta({\map{A}})=(\lamm,d_2,\muz)$ where $d_1$ and $d_2$ is the distinguished  double coset representative for $\check W_{\lambda^{+}} d_As_r\check W_{\mu^{\oo}}$ and $\check W_{\lambda^{-}} s_rd_A\check W_{\mu^{\oo}}$ respectively. 
      \item[(3)]If $A\in\hbb$, then the definition is symmetric, replacing $\la^+,\la^-$ by $\la^\bullet$ and $\mu^\bullet$ by $\mu^+,\mu^-$ in (2).
  %   for $d_A\in\cW$,  we set $\eta({\pap{A}})=(\la^\bullet,d_1,\mu^+)$ and $\eta({\mam{A}})=(\la^\bullet,d_2,\mu^-)$, where $d_1=d_A$ and $d_2$ is the distinguished  double coset representative for $\check W_{\la^\bullet} s_rg_As_r\check W_{\mu^-}$, and 
  %   for $d_A\not\in\cW$, we set $\eta({\pam{A}})=(\la^\bullet,d_1,\mu^-)$ and $\eta({\map{A}})=(\la^\bullet,d_2,\mu^+)$ where $d_1$ and $d_2$ is the distinguished  double coset representative for $\check W_{\la^\bullet} d_As_r\check W_{\mu^-}$ and $\check W_{\la^\bullet} s_rd_A\check W_{\mu^+}$ respectively. 
      
      \item[(4)] If $A\in\hbh$ then, for $d_A\in\cW$, we set $\eta({\pap{A}})=(\lamz,d_1,\muz)$ and $\eta({\mam{A}})=(\lamz,d_2,\muz)$, where $d_1=d_A$ and $d_2$ is the distinguished  double coset representative for $\check W_{\lambda^{\oo}}s_rd_As_r\check W_{\mu^{\oo}}$, and
for $d_A\not\in\cW$, we set $\eta({\pam{A}})=(\lamz,d_1,\muz)$ and $\eta({\map{A}})=(\lamz,d_2,\muz)$, where $d_1$ and $d_2$ is the distinguished  double coset representative for $\check W_{\lambda^{\oo}} d_As_r\check W_{\mu^{\oo}}$ and $\check W_{\lambda^{\oo}} s_rd_A\check W_{\mu^{\oo}}$ respectively.
\end{itemize}
\end{defn}

The next lemma shows that $\eta$ is well-defined and bijective. See a different bijection from $\csD(n,r)$ to $\xid$ in \cite[Lem.~A.4.1]{LL}.

\begin{lem}\label{eta}
The map $\eta$ is well-defined and is a bijective map. Hence, we obtain a bijection
$$\check{\mathfrak d}:=\eta^{-1}: \csD(n,r)\longrightarrow\xid.$$
%In particular, we have
    %\begin{equation}
      %  \dim\sS^\scd(n,r)=\binom{2n^2+2n+r}{r}+\binom{2n^2+2n+r-1}{r}.
   % \end{equation}
\end{lem}

\begin{proof}We first prove that $\eta$ is a well-defined injective map.
 
 (0) If $A\in\hhh$, then, by Lemma \ref{map kap}(3) and Corollary \ref{par cor}(3), $s_rd_A=d_As_r$, $\cW_\la=W_\la\cap \cW=\cW_{\la^\bullet}$,
      and $\cW_\mu=\cW_{\mu^\bullet}$. Thus,
      \begin{equation}\label{(0)}
      W_\la d_AW_\mu=
      \cW_{\la^\oo} d_A\cW_{\mu^\oo}\sqcup
      s_r \cW_{\la^\oo} d_A\cW_{\mu^\oo}.
      \end{equation}
One of the disjoint subsets is a double coset in $\cW$.

(1) If $A\in\bbb$ then, by Corollary \ref{par cor}(1) and \eqref{cW_la}, $\cW_\la=\cW_{\la^+}=W_\la$,
    $\cW_\mu=\cW_{\mu^+}=W_\mu$, and
    \begin{equation}\label{(1)} W_\la d_AW_\mu=\cW_{\la^+} d_A\cW_{\mu^+}.\end{equation}
    Thus, if $d_A\in\cW$ (resp., $d_A\not\in\cW$), then $\cW_{\la^+} d_A\cW_{\mu^+}$ and $\fkf(\cW_{\la^+} d_A\cW_{\mu^+})=
    \cW_{\la^-} s_rd_As_r\cW_{\mu^-}$ (resp., $\cW_{\la^+} d_A\cW_{\mu^+}s_r=\cW_{\la^+} d_As_r\cW_{\mu^-}$ and $\fkf(\cW_{\la^+} d_A\cW_{\mu^+}s_r)=
    \cW_{\la^-} s_rd_A\cW_{\mu^+}$)
    are distinct double cosets in $\cW$. 

 (2) If $A\in\bbh$ then, by Corollary \ref{par cor}, $W_\la=\cW_{\la^+}$ and $\cW_\mu=\cW_{\mu^\bullet}$,
     \begin{equation}\label{(2)}W_\la d_AW_\mu=(\cW_{\la^+} d_A\cW_{\mu^\oo})\sqcup (\cW_{\la^+} d_A\cW_{\mu^\oo}s_r).
     \end{equation}
    Thus, if $d_A\in\cW$ (resp., $d_A\not\in\cW$), then $\cW_{\la^+} d_A\cW_{\mu^\bullet}$ and $\fkf(\cW_{\la^+} d_A\cW_{\mu^\bullet})=
    \cW_{\la^-} s_rd_As_r\cW_{\mu^\bullet}$ (resp., $\cW_{\la^+} d_A\cW_{\mu^\bullet}s_r=\cW_{\la^+} d_As_r\cW_{\mu^\bullet}$ and $\fkf(\cW_{\la^+} d_A\cW_{\mu^\bullet}s_r)=
    \cW_{\la^-} s_rd_A\cW_{\mu^\bullet}$)
    are distinct double cosets in $\cW$. 

 (3) This case is symmetric to (2).
 
 (4) If $A\in\hbh$ then, by Corollary \ref{par cor}(3), 
        \begin{equation}\label{(4)}
        W_\la d_AW_\mu=(\cW_{\la^\oo} d_A\cW_{\mu^\oo})\sqcup
    (\cW_{\la^\oo}s_r d_As_r\cW_{\mu^\oo})\sqcup (\cW_{\la^\oo} d_As_r\cW_{\mu^\oo})\sqcup (\cW_{\la^\oo}s_r d_A\cW_{\mu^\oo}).\end{equation}
    Thus, the two distinct double cosets in $\cW$ can be seen easily.
      
  We now prove that $\eta$ is surjective.
 
Let $(\lambda^{\ep}, d, \mu^{\iota})\in \dnr$, where $\ep,\iota\in\{\oo,+,-\}$ and $d\in\csD_{\lambda^{\ep},\mu^{\iota}}$. We now show the existence of $A\in\xid$ satisfying $\eta(A)=(\lambda^{\ep}, d, \mu^{\iota})$ corresponding to the cases in Definition \ref{defkd}.

{\bf Case (0)\&(4).} We first assume $\ep=\iota=\oo$ and consider the double coset $\cW_{\la^\oo} d\cW_{\mu^\oo}=\cW_\la d\cW_\mu$ in $\cW$ which is contained in the double coset $ W_\la dW_\mu$ in $W$. Suppose $A\in\Xi$ defines the double coset $W_\la dW_\mu=W_\la d_AW_\mu$. Then $\la=\ro(A)$, $\mu=\co(A)$, and $d_A$ is the distinguished double coset representative.

If $d_As_r=s_rd_A$ then, by Lemma \ref{map kap} (2),  $a_{n+1,n+1}>0$. Thus, $A\in\hhh$ and one of the two disjoint sets in \eqref{(0)} is a double coset in $\cW$, depending on $d_A\in\cW$ or not.

If $d_As_r\neq s_rd_A$, then $a_{n+1,n+1}=0$. Thus, $A\in\hbh$, $W_\la d_AW_\mu$ has a decomposition as in \eqref{(4)} (and  one of the two disjoint sets in \eqref{(4)} is a double coset, depending on $d_A\in\cW$ or not).
    
The remaining cases, where Case (1) for  $\ep,\iota\in\{+,-\}$ and Cases (2)\&(3) for $\ep =\oo$ or $\iota=\oo$, but not both, we may use \eqref{(1)} or \eqref{(2)} to prove the existence similarly.
  \end{proof}

Similar to the row/column sum vectors $\ro(A),\co(A)$ defined in \eqref{roco}, we now
  define {\it row weights} and {\it column weights} of  elements in $\Xi(n,r)$:
  \begin{equation}\label{rwcw}
 \rw,\cw:\check \Xi(n,r)\longrightarrow\La^\scd(n,r)
 \end{equation}
  by setting $\eta(\cBA)=(\rw(\cBA),d,\cw(\cBA))$, for any $\cBA\in\check \Xi(n,r)$.
  
\begin{lem}\label{crw}Let $\ep,\ep'\in\{+,-\}$.
\begin{enumerate}
\item[(0)] If $\cBA$ is associated with $A\in\hhh$, then $\rw(\cBA),\cw(\cBA)\in\La^\bullet(n+1,r)$.
\item If $\cBA={}^\ep\!A^{\ep'}$  with $A\in\bbb$, then $\rw({}^\ep\!A^{\ep'})\in\La^\circ_\ep(n+1,r)$ and $\cw({}^\ep\!A^{\ep'})\in\La^\circ_{\ep'}(n+1,r)$.
\item If $\cBA={}^\ep\!A^{\ep'}$  with $A\in\bbh$, then $\rw({}^\ep\!A^{\ep'})\in\La^\circ_\ep(n+1,r)$ and $\cw({}^\ep\!A^{\ep'})\in\La^\bullet(n+1,r)$.
\item If $\cBA={}^\ep\!A^{\ep'}$ with $A\in\hbb$, then $\rw({}^\ep\!A^{\ep'})\in\La^\bullet(n+1,r)$ and $\cw({}^\ep\!A^{\ep'})\in\La^\circ_\ep(n+1,r)$.
\item If $\cBA={}^\ep\!A^{\ep'}$  with $A\in\hbh$, then $\rw(\cBA),\cw(\cBA)\in\La^\bullet(n+1,r)$.
\end{enumerate}
\end{lem}
Note that only in case (0) and (1), the signs ($\pm$ or $\cdot$) on $A$ are inherited by the row/column weights.

Recalled the natural basis $\{\phi_{\al,\be}^d\mid(\al,d,\be)\in\csD(n,r)\}$ for $S_{\bsq}^\scd(n,r)$ given in Corollary \ref{stdbsD}. With the bijection $\eta$ in Lemma \ref{eta}, we may also label the basis element by signed matrix notation: 
$$
\{\phi_{\cBA}\mid\cBA\in\check\Xi(n,r)\}\;\;\text{ where }
\phi_{\cBA}=\phi_{\al\be}^d\;\;\text{ if }\eta(\al,d,\be)=\cBA.
$$
%For $\la,\mu\in \La^\scd$ and $d\in \check{\mathscr{D}}_{\la\mu}$, we construct a right $\check H_\bsq$-linear map $\phi_{\la\mu}^{ d}\in \text{Hom}_{\check{ H}}(\cx_\mu\check H_\bsq,\cx_\la\check H_\bsq)$ which sends $\cx_\mu$ to $T_{\check W_{\la}d\check W_{\mu}}$. Therefore, by Lemma~\ref{eta}, for $A=\eta(\la, d, \mu)$, define
We have also, for $\cBA,\cBB\in\check\Xi(n,r)$,
\begin{equation}\label{phiAB}
\phi_\cBA\phi_\cBB=0\;\;\text{ unless }\cw(\cBA)=\rw(\cBB).
\end{equation}

\begin{rem}
For $A\in\overset\circ\Xi$ such that $A={\frak d}(\la,d_A,\mu)$, denote $\sfm_1=\sfm(\la)$ is the  maximal index $i\leq n+1$ satisfying $\la_i\neq0$. Let $\sfm_2=\sfm(\mu)$. Define $\sfm$ to be the maximal index $j$ satisfying $a_{\sfm_1,\sfm}\neq 0$ if $\sfm_1<n+1$, or $a_{n+1,\sfm}\neq0$ if $\sfm_1=n+1$, and $\sfm\leq n$ (such $\sfm$ exists since $\la_{n+1}\neq0$). We now use the subsets $\cR^{\la^\ep}_i$ defined in \eqref{defrlam} to justify Definition \ref{defkd} above. %We follow the notation in .
%\begin{itemize}
%    \item [(1)] 

(1)    For $A\in\bbb$ and $\ep\in\{+,-\}$, if $A$ defines
     $\paep{A},\mamep{A}\in\bbb_+\sqcup\bbb_-$, then $\eta(\paep{A})=(\la^+,d_1,\mu^\ep)$ and $\eta(\mamep{A})=(\la^-,d_2,\mu^{-\ep})$
     (see \eqref{+-} for notational convention and Definition \ref{defkd} for the definition of $d_1,d_2$) which define distinct natural basis elements $\phi_{\paep{A}}$ and $\phi_{\mamep{A}}$ according their row or column weights given in \eqref{rwcw}. This distinction can be reflected by the distinction of  the $\sfm_1\times \sfm_2$ subset matrices $\big(\cR_i^{\lambda^+}\cap d_1\cR_j^{\mu^\ep}\big)$ and $\big(\cR_i^{\lambda^-}\cap d_2\cR_j^{\mu^{-\ep}}\big)$ since
      $$\cR_{\sfm_1}^{\lambda^+}\cap d_1\cR_{\sfm}^{\mu^\ep}=I_{\sfm_1,\sfm}\quad\text{ and }\quad
\cR_{\sfm_1}^{\lambda^-}\cap d_2\cR_{\sfm}^{\mu^{-\ep}}=(I_{\sfm_1,\sfm}\setminus\{r\})\cup\{r+1\}$$ are distinct.     However,
     as subsets of $\cW$, the double cosets $\cW_{\la^+} d_1\cW_{\mu^\ep}$ and $\cW_{\la^-}d_2\cW_{\mu^{-\ep}}$ as well as their
     conjugate intersections
     $$\aligned
\cW_{\la^+}\cap d_1\cW_{\mu^\ep}{d_1^{-1}}&=\bigcap_{i,j}\text{Stab}_{\cW}(\cR_i^{\lambda^+}\cap d_1\cR_j^{\mu^\ep}),\\ 
\cW_{\la^-}\cap d_2\cW_{\mu^{-\ep}}{d_2^{-1}}&=\bigcap_{i,j}\text{Stab}_{\cW}(\cR_i^{\lambda^-}\cap d_2\cR_j^{\mu^{-\ep}}),\endaligned
$$
may be the same.

Similar statement holds for $A\in\bbh$ with all $\mu^\ep$ and $\mu^{-\ep}$ replaced by $\mu^\bullet$. This also gives a similar statement for $A\in\hbb$ by symmetry (with $\sfm_1=n+1$ and $\sfm\leq n$).

%\item [(2)] 
(2) For $A\in\hbh$, we also have $\paep{A},\mamep{A}\in\hbh_+\sqcup\hbh_-$. In this case, the natural basis elements $\phi_{\paep{A}}$ and $\phi_{\mamep{A}}$ have the same domain and codomain, but are defined by distinct double cosets $\cW_{\la^\bullet} d_1\cW_{\mu^\bullet}$ and $\cW_{\la^\bullet}d_2\cW_{\mu^\bullet}$. Of course, the conjugate intersections
$\cW_{\la^\bullet}\cap (\cW_{\mu^\bullet})^{d_1^{-1}}$ and  $\cW_{\la^\bullet}\cap (\cW_{\mu^\bullet})^{d_2^{-1}}$  are also distinct in this case.
%\end{itemize}
\end{rem}

We are now ready to compute the dimension $\#\check\Xi(n,r)$ of $S^\scd_\bsq(n,r)$.

\begin{thm}\label{geosettingD}Let $\sX=\mathcal{F}_{n,2r}^\jmath $ and $\cG={\mathrm{SO}}_{2r}(q)$. Then there is an algebra isomorphism
$$S^\scd_q(n,r):=S_\bsq^\scd(n,r)|_{\bsq=q}\cong \fkF_{\cG}(\sX\times \sX)$$
sending $\phi_\cBA|_{\bsq=q}$ to $f_{\csO(\cBA)}$.
Moreover, the following dimension formula holds
    \begin{equation}
    \label{geosettingD-2}
        \#\xid=\binom{2n^2+2n+r}{r}+\binom{2n^2+2n+r-1}{r}.
    \end{equation}
\end{thm}
\begin{proof}By Lemma \ref{scott}, the convolution algebra $\fkF_{{\mathrm{SO}}_{2r}(q)}(\sX\times \sX)$ is isomorphic to the endomorphism algebra $\End_{\cG(q)}(\mathbb Z\sX)^{\text{\rm op}}$ of the permutation $\cG(q)$-module $\mathbb Z\sX$.

       We now calculate $\#\xid$. By the definition of $\xid$ in \eqref{xid}, we have, for $N=2n+1$, 
    $$\#\xid-\#\Xi_{N,2r}=
    \#\{A\in \Xi_{N,2r}\mid a_{n+1,n+1}=0\}=\binom{2n^2+2n+r-1}{r}=\binom{n(N+1)+r-1}{r},$$
(see  Lemma \ref{N^2}).   Moreover, $$\begin{aligned}
        \#\Xi_{N,2r}&=\#\Big\{A=(a_{ij})\in\Xi_{N,2r}\;\Big|\;\sum\limits_{i\leq n:j}a_{ij}+\sum\limits_{j\leq n}a_{n+1,j}=r-\frac{a_{n+1,n+1}}{2}\Big\}\\
        &=\sum \limits_{l=0}^r \binom{2n^2+2n+r-l-1}{2n^2+2n-1}=\binom{2n^2+2n+r}{r}.
    \end{aligned}$$
    Combining two equations gives $\#\xid=\binom{2n^2+2n+r-1}{r}+\binom{2n^2+2n+r}{r}.$
   \end{proof}

%\begin{cor}
 %   The set $\{\phi_A|A\in \xid\}$ forms an $\sA$-basis of $S_{\bsq}^\scd(n,r),$ which we call the natural basis.
%\end{cor}

\begin{rem}\label{error}
(1). 
The theorem shows that the $\sA$-algebra $S_\bsq^\scd(n,r)$ is the ``quantumization'' (in the sense of \cite[p.17]{DDPW}) of the algebras $\{\fkF_{{\mathrm{SO}}_{2r}(q)}(\sX\times \sX)\}_{q\in\mathcal P}$ which is the same as the algebra defined in \cite[(5)]{FL}.
It should be pointed out that the dimension formula $\#\Xi_{\scd}$ given in \cite[Lem.~4.2.1]{FL} is incorrect as the definition of the index set $\Xi_{\scd}=\Xi^+\sqcup\Xi^0\sqcup\Xi^-$ mistakenly used $\Xi^0$ instead of $\hhh$.
In fact, we have $\Xi^0=\hhh\sqcup\hbh$. Consequently, 
\begin{equation}\label{FL}\#\xid=\#\Xi_{\scd}+\#\hbh.\end{equation}
Indeed, this can be seen from $\overset\circ\Xi=\hbh\sqcup(\overset{\circ\!\circ\text{-}}\Xi\cup\overset{\text{-}\circ\!\circ}\Xi)$, where $\overset{\circ\!\circ\text{-}}\Xi=\{A\in\Xi\mid \la_{n+1}=0=a_{n+1,n+1}\}$. Thus,
$\#\hbh=\#\overset\circ\Xi-2\overset{\circ\!\circ\text{-}}\Xi+\#\bbb=\binom{2n^2+2n+r-1}{r}-2\binom{2n^2+n+r-1}{r}+\binom{2n^2+r-1}{r}$, giving \eqref{FL}.

(2). The anonymous referee pointed out to us that 
by using coordinate algebra approach developed in~\cite{LNX}, 
the dimension formula (\ref{geosettingD-2})  of $S^\scd_\bsq(n,r)$ was also  obtained by Ziqing Xiang in an  unpublished manuscript.
\end{rem}

\section{Multiplication formulas in the $(\bsq,1)$-Schur algebra {\rm$S_{\bsq,1}^\scb(n,r)$}}
%We now derive certain multiplication formulas first in $S^\scb_{\bsq,1}(n,r)$ and then in $S^\scd_\bsq(n,r)$.

%\subsection{Multiplication formulas in {\rm$S^{\scb}_{\bsq,1}(n,r)$}} 
We first follow the idea in \cite{BLM} or \cite{BKLW} to derive the multiplication formulas in $S_{\bsq,1}^\scb(n,r)$.
Due to the unequal parameter nature, we will see the differences between the structure constants here and those occurring in the (equal parameter) multiplication formulas for $S^\jmath$ in \cite{BKLW}.

Let $\Gr(1,2r)=\Gr_1(\mathbb F_q^{2r})$ be the Grassmannian of isotropic lines in $\mathbb F_q^{2r}$:
$$\Gr(1,2r)=\{\lr{x}\subseteq \mathbb{F}_q^{2r}\mid x\neq0, \lr{x,x}_\sfJ=0\}.$$
 The following result is known. Part (1) is stated in \cite[Lem.~3.1.2]{FL}; while part (2) is \cite[Lem.~3.1.3]{FL}. For completeness and later use in Corollary \ref{claim}, we include a proof for part (1).

\begin{lem}\label{grass}
    (1) The cardinality of the set $ \Gr(1,2r)$ of isotropic lines in $\mathbb{F}_q^{2r}$ is 
    $$\# \Gr(1,2r)=\frac{\left(q^{r}-1\right)\left(q^{r-1}+1\right)}{q-1}=\frac{q^{2r-1}-1}{q-1}+q^{r-1}.$$ 

    (2) Let %$0\overset{a_1}\subset F_1\overset{a_2}\subset F_2\overset{a_3}\subset F_3\overset{a_4}\subset F_4\overset{a_5}\subset F_5=\mathbb F_q^{2r}$ 
$F=({{F}_i})_{1\leq i\leq 5}\in \mathcal F_{2,r}^\jmath$ be a $2$-step isotropic flag with $\text{\bf dim}(\hat{F})=(a_1,a_2,a_3,a_4,a_5)$ and, for $i=3,4$, let
     $Z_i=\{L\subset {F}_i\mid L\in\Gr(1,2r), L\not\subseteq {F}_{i-1} \}$. Then
    $$({\romannumeral 1}) \ \#Z_3=q^{a_1+a_2}\Big(\frac{q^{a_3-1}-1}{q-1}+q^{\frac{a_3}{2}-1}\Big), \qquad
    ({\romannumeral 2}) \ \#Z_4=q^{a_1+a_2+a_3-1}\frac{q^{a_4}-1}{q-1}.$$
\end{lem}

\begin{proof}
We prove (1) (and (2) is given in loc. cit.). For $x=x_{1} e_{1}+\cdots+x_re_r+x_{r+1}e_{r+1}+\cdots+x_{2 r} e_{2 r} \in \mathbb{F}_{q}^{2 r}-\{0\}$, let
   $${\textbf{a}}={\bf a}(x)=(x_1, x_2, \cdots, x_r),\;\; {\textbf{b}}={\bf b}(x)=(x_{2r}, x_{2r-1}, \cdots, x_{r+1})\in \mathbb F_q^r.$$
 Then, $\mathbb F_q$ has odd characteristic implies 
 $$\aligned
 \lr{x,x}_\sfJ=0&\iff {\bf a}{\bf b}^t=0\\
 &\iff\text{either }{\bf a}=0,\text{ or }{\bf a}\neq0,\text{ but }{\textbf{b}}^{t}\in \text{Ker }{\textbf{a}}.\endaligned
 $$
Thus,
$$\Gr(1,2r)=\{\lr{x}\subseteq \mathbb{F}_q^{2r}\mid {\bf a}=0, {\bf b}\in\mathbb F_q^r-\{0\}\}
\sqcup \{\lr{x}\subseteq \mathbb{F}_q^{2r}\mid {\bf a}\in\mathbb F_q^r-\{0\}, {\bf b}^t\in\text{Ker}({\bf a})\}.$$
Clearly, we have
$$\#\{\lr{x}\subseteq \mathbb{F}_q^{2r}\mid {\bf a}=0, {\bf b}\in\mathbb F_q^r-\{0\}\}=\frac{q^r-1}{q-1}.$$
On the other hand, for ${\bf a}\neq0$, since $ \text{Ker }{\textbf{a}}\cong \mathbb{F}_{q}^{r-1}$, it follows that
$$\#\{\lr{x}\subseteq \mathbb{F}_q^{2r}\mid {\bf a}\in\mathbb F_q^r-\{0\}, {\bf b}^t\in\text{Ker}({\bf a})\}=
\frac{q^r-1}{q-1} q^{r-1}.$$
 Hence, $$\# \Gr(1,2r)=\frac{q^{r}-1}{q-1}+\frac{q^{r}-1}{q-1}\cdot q^{r-1}
=\frac{\left(q^{r}-1\right)\left(q^{r-1}+1\right)}{q-1},$$
as desired.
%Now we will prove (2).  
 %    To prove ({\romannumeral 1}), consider $Z^\prime_3=\{W\subset \hat{F}_3\ |\  \hat{F}_2\stackrel{1}{\subset }W, W\text{ is isotropic}\}$. Let $\phi :Z_3\rightarrow Z_3^\prime$ be the map defined by $U\mapsto \hat{F}_2+U$ which is surjective, then $Z^\prime_3\simeq \{\text{isotropic lines in }\hat{F}_3/\hat{F}_2\}$ and by Lemma~\ref{grass}, $\#Z^\prime_3=\frac{q^{a_3-1}}{q-1}+q^{\frac{a_3}{2}-1}$ which proves ({\romannumeral 1}).
 %    As for ({\romannumeral 2}), consider $Z^\prime_4=\{W\subset \hat{F}_4\ |\  \hat{F}_1\stackrel{1}{\subset }W, W\not\subseteq \hat{F}_3, W\text{ is isotropic}\}$. Let $\phi^\prime:Z_4\rightarrow Z_4^\prime$ be the map defined by $U\mapsto \hat{F}_1+U$ which is surjective. Therefore, 
%     $$\begin{aligned}
%         \#Z_4&=\#\tilde Z_4-\#\tilde Z_3-(\#\tilde Z_3\setminus \#\tilde Z_2+\#\tilde Z_2)\\
%        &=\frac{q^{a_2+a_3+a_4-1}-1}{q-1}+q^{a_2+\frac{a_3}{2}-1}-\big(q^{a_2}(\frac{q^{a_3-1}-1}{q-1}+q^{\frac{a_3}{2}-1})+\frac{q^{a_2-1}}{q-1} \big)\\
 %       &=q^{a_2+a_3-1}\frac{q^{a_4}-1}{q-1},
 %    \end{aligned}$$
%    and ({\romannumeral 2}) follows.
\end{proof}

%We set $E_{i j}^{\theta}=E_{i j}+E_{N+1-i, N+1-j},$ where $E_{i j}$ is the $N\times N$ matrix whose $(i,j)$-entry is $1$ and all other entries are 0. 
Let $e_A=f_{\sO_A}$ be the characteristic function of the ${\mathrm{O}_{2r}}(q)$-orbit corresponding to $A\in \Xi_{N, 2r}$. Therefore, $\{e_A|A\in \Xi_{N, 2r}\}$ forms a basis of $\sj$. For convenience, set $e_A=0,\text{ if }A\notin \Xi_{N,2r}.$

%%%%%%%%%%%%%%
%The matrices $A$, $\apl$ and $\apu$ define distinguished double coset representatives. The following lemma connect them when they are regarded as permutations in $\fS_{2r}$.

%\begin{lem} For $A=(a_{i,j}),\apl,\apu\in\Xi_{N,2r}$ defined as above and $p\in[1,N]$, let $\mu=\ro(A)$, $a=\sum_{j=1}^{p-1}a_{h+1,j}$, and $b=\sum_{j=p+1}^Na_{h,j}$.
%\begin{itemize}
%\item[(1)]Let $\la=\ro(apl)$
%\end{itemize}
%\end{lem}
%%%%%%%%%%%%%%%%

Now we formulate some multiplication formulas in $S^\scb_{\bsq,1}(n,r)$. 
These formulas are similar to (but different from) those in \cite[Lem.~3.2]{BKLW}, ({cf. \cite[4.3.2(c)]{FL}, \cite[A.5.16]{LL}}).
Also, the proof, motivated from that of \cite[Prop.~4.3.2]{FL}, is much shorter than the one in \cite{BKLW}. Recall the matrices $\apl$ and $\apu$ associated with $A$ defined in \eqref{hp}. % and the convention of $e_A=0$ when $A$ has a negative entry.

\begin{thm}\label{sjcon} Maintain the notation above and suppose that $h\in[1,n]$, $N=2n+1$, and $A=(a_{i,j})\in \Xi_{N, 2r}$.
\begin{itemize}
\item[(1)] If $B=E_{h, h+1}^{\theta}+\wla\in \Xi_{N, 2r}$, for some $\la\in\La(n+1,r)$, and $\operatorname{ro}(A)=\operatorname{co}(B)$, then we have in $S_{\bsq,1}^\scb(n,r)$
$$ e_B * e_A=\sum_{p\in[1, N],a_{h+1,p}>0} \bsq^{\sum_{j > p} a_{h,j}} \ggi{a_{h,p}+1}e_{\apl}.$$
\item[(2)] If $C=E_{h+1,h}^{\theta}+\wmu\in \Xi_{N, 2r}$, for some $\mu\in\La(n+1,r)$, and $\operatorname{ro}(A)=\operatorname{co}(C)$, then, for $\de=1-\de_{a_{n,n+1},0}$, we have in $S_{\bsq,1}^\scb(n,r)$
\begin{equation}\nonumber
    e_C* e_A=\begin{cases}
    \sum \limits_{p\in[1,N],a_{h,p}>0} \bsq^{\sum_{j<p} a_{h+1,j}} \ggi{a_{h+1,p}+1} e_{\apu},&\text{ if }h\neq n;\\
 \sum \limits_{p\in[1,N], p\neq n+1\atop a_{n,p}>0} \bsq^{\sum\limits_{j<p} a_{n+1,j}} \ggi{a_{n+1,p}+1} e_{\apu}    +\de\bsq^{\sum\limits_{j<n+1} a_{n+1,j}}\big(\!\ggi{a_{n+1,n+1}+1}+\bsq^{\frac{a_{n+1,n+1}}2}\big) e_{_n\!A_{\overline{n+1}}},&\text{ if }h=n.
    \end{cases}
\end{equation}
%\begin{equation}\nonumber
%    e_C* e_A=\sum \limits_{1 \leqslant p \leqslant N} q^{2\sum_{j<p} a_{h+1,j}} \frac{(q^{a_{h+1,p}+1-\delta_{p,n+1}\delta_{h,n}}+1)(q^{a_{h+1,p}+1+\delta_{p,n+1}\delta_{h,n}}-1)}{q^{2}-1} e_{\apu}
%\end{equation}
\end{itemize}
\end{thm}
\begin{proof}
     We only prove (2). The proof consists of two cases: (1) $h<n$ or $h=n, p\leq n$ and (2) $h=n, p>n$. The first case is similar to the original type $A$ proof given in \cite[Lem.~3.2]{BLM}, while the second case is parallel to that of \cite[Lem.~3.2]{BKLW}. Note that the proof of (1) is similar to the first case in (2). 
     
     By the isomorphism $e_A|_{\bsq=q}\mapsto f_A=f_{\sO_A}$ in Theorem \ref{geosettingB}, it suffices to compute the coefficients in the multiplication formula 
     \begin{equation}\label{CA}f_C*f_A=\sum_{A'}\#Z_{C,A,A'}f_{A'}\end{equation}
      in \eqref{f_Of_O'}, where,
 for $(F,F')\in\sO_{A'}$,  
 \begin{equation}\label{CAA'}
 Z_{C,A,A'}=\{ E\in\sF_{n,r}^\jmath\mid (F,E)\in\sO_C,(E,F')\in\sO_A\}.
 \end{equation}   
   By Corollary~\ref{Cormulti}(2), if $\#Z_{C,A,A^\prime}\neq 0$ then $A^\prime= \apu=(a'_{i,j})$ for some $p\in[1,N]$. 
 %  We have $a_{h+1,p}=a^\prime_{h+1,p}-1$, $a_{h,p}=a^\prime_{h,p}+1$.
 %  Thus, putting $Z_{h,p}=Z_{C,A,\apu}$, \eqref{CA} becomes%Therefore, we could rewrite the multiplication formula as 
 %  $$f_{E_{h+1,h}^\th+\wmu}* f_A=\sum\limits_{p\in[1,N]}\#Z_{h,p}f_{\apu}.$$ 
We now compute the numbers $\#Z_{C,A,A^\prime}.$
   
 Let $Z_{h,p}=Z_{C,A,A^\prime}$.  Suppose $\fkm(F, F^\prime)\in\sO_{\apu}$ and 
   %$F\in \mathcal F_{n,r}^\jmath$ and define 
  let $Z_h$ be the set of all isotropic subspaces $S$ such that $F_h\stackrel{1}{\subset }S\subseteq F_{h+1}$. 
  Note that if $h<n$ then $F_{h+1}$ is isotropic. Thus, every subspace of $F_{h+1}$ is isotropic.
 Hence, $h<n$ implies $\#Z_h=\#\{\text{1-dimensional subspaces of }F_{h+1}/F_h\}$.
  
  Each $S\in Z_h$ defines an $n$-step isotropic flag $E\in\sF^\jmath_{n,r}$ such that $E_h=S$ and $F\overset{1}\subset_h E$. By Lemma~\ref{Cormulti}, we have $\fkm(F, E)=C$, $\fkm(E, F^\prime)=A$
  and $S$ belongs to the set
  \begin{equation*}
Z'_{h,p}=\{S\in Z_h\mid  F_h\cap F_j^{'}=S\cap F_j^{'} \text{ if }j<p \text{ and } F_h\cap F_j^{'}\neq S\cap F_j^{'} \text{ if }j\geq p\}.
  \end{equation*}
In particular, we have $\#Z_{h,p}=\#Z'_{h,p}$.

Note that if we write $S=F_h+\lr{v}$, then $F_h\cap F_{p-1}^{'}=S\cap F_{p-1}^{'}\iff v\not\in S\cap F_{p-1}^{'}$ which is equivalent to $S\not\subseteq F_h+F_{h+1}\cap F'_{p-1}$. Similarly, $F_h\cap F_p^{'}\neq S\cap F_p^{'}\iff
v\in  S\cap F_p^{'}\iff S\subseteq F_h+F_{h+1}\cap F'_{p}$. Thus, with the notation $F_{i,j}:=F_{i-1}+F_i\cap F'_j$ used in the $n(N+1)$-step isotropic flag \eqref{Fij}, we have
\begin{equation}\label{Z'}
Z'_{h,p}=\{S\in Z_h\mid S\subseteq F_{h+1,p}, S\not\subseteq F_{h+1,p-1}\}.
\end{equation}
% For each $p\in[1,n]$, find $F^\prime\in \mathcal F_{n,r}^\jmath$ such that $\fkm(F, F^\prime)=\apu$.  If we define $\tilde F\in \mathcal F_{n,r}^\jmath$ satisfies ${\tilde F}_{h}=S$ where $S\in Z_h$ and $F_i={\tilde F}_i$ when $i\neq h$, then $\fkm(F, \tilde{F})=C$ and $\fkm(\tilde{F}, F^\prime)=A$ by Lemma~\ref{Cormulti}(2). Therefore, $Z_{hp}$ be the subset of $Z_{h}$ satisfies   $$F_h\cap F_j^{'}=S\cap F_j^{'} \text{ if }j<p \text{ and } F_h\cap F_j^{'}\neq S\cap F_j^{'} \text{ if }j\geq p.$$

%Since $\apu=A'=(a^\prime_{i,j})$.
   If $h<n$, then $A$ and $A'$ differ only at $a'_{h,p}=a_{h,p}-1$, $a'_{h+1,p}=a_{h+1,p}+1$. Thus,
    $$\begin{aligned}
   \#Z_{h,p}&=\#\{S\in Z_h \mid S\subseteq F_{h+1,p} \}-\#\{S\mid  S\subseteq F_{h+1,p-1}\}\\
   &=(q-1)^{-1}(q^{\dim(F_{h}+F_{h+1}\cap F_p^\prime/F_h)}-q^{\dim(F_{h}+F_{h+1}\cap F_{p-1}^\prime/F_h)})\\
   &=(q-1)^{-1}(q^{\sum_{j\leq p}a^\prime_{h+1,j}}-q^{\sum_{j\leq p-1}a^\prime_{h+1,j}})=q^{\sum_{j<p}a_{h+1,j}}\frac{q^{a_{h+1,p}+1}-1}{q-1}.
   \end{aligned}$$
 If $h=n$, then $S\in Z_n$ may not be isotropic. However, $p\leq n$ implies $F_{n+1,p}=
  F_n+F_{n}^\perp\cap F'_p$ is isotropic. Thus, every $S\in Z'_{n,p}$ is isotropic. Hence, the counting formula above continues to hold for $h=n$ and $p\leq n$.

  % since all $F_n+F_{n+1}\cap F_p^\prime$ are isotropic when $F_n$ and $F_{p}^\prime$ are isotropic.

 It remains to prove the case for $h=n$ and $p\geq n+1$. We extract the central section (the $(n+1)$th row) from \eqref{Fij}:
 $$F_n\subseteq F_{n+1,1}\subseteq F_{n+1,2}\subseteq \cdots\subseteq F_{n+1,n}\subseteq F_{n+1,n+1}\subseteq \cdots\subseteq F_{n+1,N}=F_{n+1}.$$
By Lemma \ref{N^2}, we have $F_{n+1,i}^\perp=F_{n+1,N-i}$, for all $i\in[1,n]$. Since $F_n^\perp=F_{n+1}$ and, for $W:=F_{n+1}/F_n$, $\lr{\quad\;}_\sfJ|_W$ is equivalent to $\lr{\quad\;}_{\sfJ_{2r'}}$ by \cite[Lem.~3.1.1]{FL}, %(see \eqref{FLQ}), 
where $2r'=\dim W$, the above filtration induces an $n$-step isotropic flag in $W$:
$$0\subseteq W_1\subseteq W_2\subseteq\cdots\subseteq W_{N-1}\subseteq W_N =W\cong\mathbb F_q^{2r'},$$
 where $W_i=\frac{F_n+F_{n+1}\cap F_i^\prime}{F_n}$ for $i\in[1,N]$.
 
 Now we finish the computation with simple applications of Lemma ~\ref{grass}(2).
%  then $W_i^\perp=W_{N-i}$ since 
%    $$(F_n+F_{n+1}\cap F_i^\prime)^\perp=F_n^\perp \cap (F_{n+1}\cap F_i^\prime)^\perp=F_{n+1}\cap (F_n+F^\prime_{N-i})=F_n+F_{n+1}\cap F^\prime_{N-i}.$$

 If $p=n+1$, then consider the $2$-step isotropic flag
    $$0\overset{a_1}\subseteq 0\overset{a_2}\subseteq W_n\overset{a_3}\subseteq W_{n+1}\overset{a_4}\subseteq W\overset{a_5}\subseteq W$$ in $W\cong\mathcal F_{2,r'}^\jmath$.
    Clearly, since $A'=A-E^\th_{n,n+1}+E^\th_{n+1,n+1}$, we have $a_1=0$,
   $$a_2=\text{dim }W_n=\sum_{j< n+1}a_{n+1,j}^\prime=\sum_{j< n+1}a_{n+1,j},\; \text{ and }\; a_3=\text{dim }(W_{n+1}/W_n)=a_{n+1,n+1}^\prime=a_{n+1,n+1}+2.$$ 
    Thus, \eqref{Z'} becomes in this case
    $Z'_{n,n+1}=\{S\in Z_n\mid S\subseteq F_{n+1,n+1}, S\not\subseteq F_{n+1,n}\}$. Hence,
  $\#Z'_{n,n+1}=\#Z_3$ where $Z_3=\{S\subset W_{n+1}\mid S\in\Gr(1,2r'), S\not\subseteq W_{n}\}$. By Lemma ~\ref{grass}(2)({\romannumeral 1}), we obtain
  $$\#Z_{n,n+1}=q^{\sum_{j<n+1}a_{n+1,j}}\Big(\frac{q^{a_{n+1,n+1}+1}-1}{q-1}+q^{\frac{a_{n+1,n+1}}2}\Big),$$ 
 as desired. 
 
  Finally, 
    suppose $p>n+1$. The computation is similar.  Consider the $2$-step isotropic flag in $W$:
    $$0\overset{a_1}\subseteq  W_{N-p}
    \overset{a_2}\subseteq W_{N-p+1}\overset{a_3}\subseteq W_{p-1}\overset{a_4}\subseteq W_{p}\overset{a_5}\subseteq W_{N}\cong  \mathbb{F}_q^{2r'}.$$
 Here, with $A'=A-E^\th_{n,p}+E^\th_{n+1,p}$ ($p>n+1$), we have
 $$a_1+a_2+a_3=\text{dim }W_{p-1}=\sum_{j< p}a_{n+1,j}^\prime=\sum_{j<p}a_{n+1,j}\; \text{ and }\; a_4= \text{dim }(W_{p}/W_{p-1})=a_{n+1,p}^\prime=a_{n+1,p}+1.$$
    Thus,  by Lemma~\ref{grass}(2)({\romannumeral 2}),
     $\#Z'_{n,p}=\#Z_4$, where $Z_4=\{S\subset W_{p}\mid S\in\Gr(1,2r'),S\not\subseteq W_{p-1}\}$. Hence,
     for $p>n+1$,
      $$\#Z_{n,p}=q^{\sum_{j<p}a_{n+1,j}}\frac{q^{a_{n+1,p}+1}-1}{q-1}.$$ The theorem is proven.
\end{proof}
The proof above implies immediately the following which will be used in \S7.
\begin{cor}\label{claim}Maintain the notation on $A,C,A'=\apu$ etc. set in Theorem \ref{new}.
Suppose $(F,F')\in\sO_{A'}$ and $\al=\ro(A')$. If $\{v_1,v_2,\ldots,v_{2r}\}$ forms a basis for $\mathbb F^{2r}$ such that $F_i=\lr{v_1,v_2,\ldots,v_{\widetilde\al_i}}$ for all $i\in[1,N]$,
then the set $Z_{C,A,A'}$ defined in \eqref{CAA'} has the following description
$$Z_{C,A,A'}=\begin{cases}\{F_h+\lr{v}\mid v=\sum_{i=1}^mb_iv_{\widetilde\al_h+i}+\sum_{j=1}^{a_{h+1,p}}b'_jv_{\widetilde\al_h+m+j}\},&
\text{ if }h<n\text{ or }h=n, p\neq n+1;\\
\{F_h+\lr{v}\mid v=\sum_{i=1}^mc_iv_{\widetilde\al_n+i}+\sum_{j=1}^{a_{h+1,n+1}}c'_jv_{\widetilde\al_n+m+j}\}, &\text{ if }h=n, p=n+1,
\end{cases}$$
 where $m=\begin{cases}a_{h+1,1}\cdots+a_{h+1,p-1},&\text{ if }h<n\text{ or }h=n,p\neq n+1;\\
 a_{n+1,1}\cdots+a_{n+1,n},&\text{ if }h=n, p=n+1,
 \end{cases}$ and the coefficients $b'_j$ satisfy that
 not all $b_j'=0$, and $c'_j$ satisfy that either ${\bf c}=(c'_r,c'_{r-1},\ldots,c'_{r-\frac a2+1})\neq{\bf0}$ $(a=a_{n+1,n+1})$ with $(c'_{r+1},\ldots,c'_{r+\frac a2})\in\text{\rm ker}({\bf c})$, or ${\bf c}={\bf0}$ but not all $c'_{r+1},\ldots,c'_{r+\frac a2}=0$.
\end{cor}

\section{Multiplication formulas in the $\bsq$-Schur algebra $S_\bsq^\scd(n,r)$, I} %(\uppercase\expandafter{\romannumeral1})}

We now use the multiplication formulas given in Theorem \ref{sjcon} to derive the corresponding  multiplication formulas in $S_\bsq^\scd(n,r)$. %and sign change on the $\cG$-orbit. 
We first look for the counterpart of Theorem \ref{sjcon}(1)
\begin{equation}\label{hneqn}
    \e{E^\th_{h,h+1}+\wga} * \e{A}=\de_{\wga,\ro(A)-e^\th_{h+1}}\sum_{1\leq p\leq N}g_{h,A,p}\e{\apl},
\end{equation}
where  $h\in[1,n]$, $\ga\in\La(n+1,r)$, $A=(a_{i,j}) \in \Xi_{N, 2r}$, $\apl=A+E_{h,p}^\theta-E^{\theta}_{h+1,p}$ for $p\in[1,N]$, and 
\begin{equation}\label{coefg}
g_{h,A,p}=\begin{cases}\bsq^{\sum_{j>p} a_{h,j}} \ggi{a_{h,p}+1},&\text{ if }a_{h+1,p}>0;\\
0,&\text{ if }a_{h+1,p}=0.\end{cases}
\end{equation}
\begin{rem}\label{notn} 
 Observe from \eqref{O_A} that, if an $\O_{2r}$-orbit $\sO_A$ in $\sX\times\sX$ splits into two $\SO_{2r}$-orbits $\csO({}^+\!A^\ep)\sqcup\csO({}^-\!A^{-\ep})$, then the orbit function $f_A=f_{{}^+\!A^\ep}+f_{{}^-\!A^{-\ep}}\in \fkF_\cG(\sX\times\sX)$. Thus, we have a $\mathbb Z$-algebra embedding $\fkF_G(\sX\times\sX)\subset\fkF_\cG(\sX\times\sX)$ which is the specialization of the $\sA$-algebra embedding $S_{\bsq,1}^\scb(n,r)\subset S_{\bsq}^\kappa(n,r)$
 by Theorem \ref{geosettingB}. Composing with the isomorphism in Proposition \ref{kappa-D} gives an algebra embedding $\iota:S_{\bsq,1}^\scb(n,r)\hookrightarrow S_{\bsq}^\scd(n,r)$. 
We make the convention that, when we write $e_A=\phi_{{}^+\!A^\ep}+\phi_{{}^-\!A^{-\ep}}$ for $\ep\in\{+,-\}$, we mean that $\iota(e_A)=\phi_{{}^+\!A^\ep}+\phi_{{}^-\!A^{-\ep}}$. Here and below we use the sign convention 
\begin{equation}\label{+-}
++=+=--, \qquad+-=-=-+.
\end{equation}
\end{rem}
 
 Now, in \eqref{hneqn}, let $B=E^\th_{h,h+1}+\wga$ and assume $\la=\ro(B)$ and $\mu=\operatorname{ro}(A)=\operatorname{co}(B)$. Then the upper right corner matrix
 $B^{\text{\normalsize $\llcorner$}}=0$ (and so, $\text{sgn}(B^{\text{\normalsize $\llcorner$}})=+$). Thus, by Proposition \ref{split} and notations in Definitions \ref{split2} and \ref{csO}, we have %\eqref{corner}, 
% the elements $\cBB$ in $\xid$ defined by $B$ have the form
 $$\sO_B=\begin{cases}\begin{aligned}
    \csO(\epsiaepsi{B}{+}{+}) \sqcup\csO(\epsiaepsi{B}{-}{-}), &\text{   if }b:=b_{n+1,n+1}=0,\\
   \csO(\bzero),\qquad\ &\text{   otherwise.}
\end{aligned}\end{cases}$$
Hence, we have
$$\begin{cases}
B\in \bbb\sqcup\hhh,&\text{ if }h<n;\\
B\in\bbh\sqcup\hhh,&\text{ if }h=n.\end{cases}$$
Since we only need to consider $A=(a_{i,j})\in\Xi_{N, 2r}$ with $\mu=\ro(A)$ in \eqref{hneqn}, the three selections of $B$ above determine the column weights $\cw(\cBB)$ of $\cBB\in\{ ^+\!B^+,{}^-\!B^-,\dot B\}$, which further determine certain selections of $A$ and, consequently, of $\apl$. We now derive multiplication formulas in $S_\bsq^\scd(n,r)$ associated with the three selections of $B$ in three subsections below. 

Observe first the following general relations. Since $\apl=A+E^\th_{h,p}-E^\th_{h+1,p}$ (and $\apu=A-E_{h,p}^\theta+E^{\theta}_{h+1,p}$), it follows that
 \begin{equation}\label{sgnhapsgna}
   \text{sgn}(\apl^{\text{\normalsize $\llcorner$}})(=\text{sgn}(\apu^{\text{\normalsize $\llcorner$}}))=\begin{cases}
       \text{sgn}(\urcm), \ &\text{ if $h<n$ or $h=n$ and $1\leq p\leq n+1$},\\
       -\text{sgn}(\urcm), \ &\text{ if $h=n$ and $p>n+1$}.
   \end{cases}
\end{equation}
Further, if $a:=a_{n+1,n+1}=0$ and $\ep=\text{sgn }\urcm$, then we have $G$-orbits  $\sO_A,\sO_{\apl}$ splitting into $\cG$-orbits:
 \begin{equation}\label{sgnhap}
\aligned
(1)\qquad\sO_A&=\csO({}^+\!A^\ep)\sqcup\csO({}^-\!A^{-\ep}),\\
(2)\;\,\quad\sO_{\apl}&=    \begin{cases}
    \csO(\paep{\apl})\sqcup\csO(\mamep{\apl}), &\text{ if $h<n$ or $h=n$ and $1\leq p\leq n+1$},\\
    \csO(\pamep{\apl})\sqcup\csO(\maep{\apl}), &\text{ if $h=n$ and $p>n+1$}.
\end{cases}    \endaligned
\end{equation}
For simplicity, we fix the following notational abbreviation for the {\it central entries} of $A$ and $B$:
$$a:=a_{n+1,n+1}=0,\qquad b:=b_{n+1,n+1}.$$

%Suppose now $\cw(\cBB)=\la^{\ep}$, for some $\la\in\La(n+1,r),\ep\in\{+,-,\bullet\}$. We want to compute the product $\phi_\cBB\phi_\cBA$ where $\cBA$ has $\rw(\cBA)=\la^\ep$. (Note that $\cBA$ does not necessarily have the form ${}^{\ep}\!A^{\ep'}$ if $\ep\neq\bullet$, or $\dot A$ if $\ep=\bullet$.) 

\subsection{The $B\in\hhh$  case} Since $\sO_B=\csO(\dot B)$ in this case and $\cw(\dot B)=\mu^\bullet\in\La^\bullet(n+1,r)$, \eqref{phiAB} and Lemma \ref{crw} imply that it suffices to consider those $A$ with $A\in\hhh\sqcup\hbb\sqcup\hbh$. %Assume $\ro(A)=\wmu$ and $\co(A)=\wnu$ for some $\mu,\nu\in\La(n+1,r)$. 

\begin{thm}\label{thmhhh}
Let $A,B\in\Xi_{N,2r}$ with $B=E^\th_{h,h+1}+\wga$, and assume $\wla=\ro(B)$, $\wmu=\co(B)=\ro(A)$, and $\wnu=\co(A)$, where  $h\in[1,n]$, and $\ga,\la,\mu,\nu\in\La(n+1,r)$. Let $\ep=\text{\rm sgn }(\urcm)$. If $B\in\hhh$, then $\sO_B=\csO(\dot B)$ and the following multiplication formulas hold in $S_\bsq^\scd(n,r)$:
    \begin{itemize}
        \item [(1)]If $A\in\hhh$,
then $\sO_A=\csO(\dot A)$ and 
\begin{equation}\nonumber
    \phi_{\bzero} * \phi_{\azero}=\begin{cases}
  \sum_{1 \leqslant p \leqslant N} \gba \phi_{\apzero}, &\text{ if $h<n$ or $h=n$ and $a\neq2$},\\
  \\
  \sum\limits_{1\leq p\leq N \atop p\neq n+1}\gba \phi_{\apzero}+g_{n,A,n+1}(\phi_{\paep{(_n\! A_{n+1})}}+\phi_{\mamep{(_n \! A_{n+1})}}),&\text{ if $h=n$ and $a=2$};
  \end{cases}
\end{equation} 
\item [(2)] If $A\in \hbb\sqcup\hbh$, then
 \begin{equation}\begin{aligned}\label{bzeropaep}
 \phi_{\bzero} * \phi_{\paep{A}}
=&\begin{cases}\sum_{1\leq p\leq N}\gba \phi_{\papsgn}, &\text{ if }h<n,\\
\sum\limits_{1\leq p<n+1}\!\gba \phi_{\paep{(\apl)}}+\!\sum\limits_{ p>n+1}\!\gba \phi_{\maep{(\apl)}}, &\text{ if }h=n,\end{cases}
\\
\phi_{\bzero} * \phi_{\mamep{A}}
=&\begin{cases}
\sum_{1\leq p\leq N}\gba \phi_{\mapmsgn}, &\text{ if }h< n,\\
\sum\limits_{1\leq p<n+1}\!\gba \phi_{\mamep{(\apl)}}+\hspace{-2mm}\sum\limits_{ p>n+1}\hspace{-2mm}\gba \phi_{\pamep{(\apl)}}, &\text{ if }h=n.
\end{cases} \end{aligned}\end{equation}
    \end{itemize}
    Here all coefficients $\gba$ are given in \eqref{coefg}.
\end{thm}
\begin{proof}
(1) For  $A\in\hhh$, by the hypotheses, we have $(\rw( \azero), \cw(\azero))=(\mu^\bullet,\nu^\bullet)$.
Since, for $h<n$ or $h=n$ and $a_{n+1,n+1}\neq 2$, the central entry of $\apl$ is nonzero, it follows that $\sO_{\apl}=\csO(\dot\apl)$ is non-split. Hence, we have $e_B=\phi_{\dot B}$, $e_A=\phi_{\dot A}$, and $e_{\apl}=\phi_{\dot\apl}$, the first case in (1) follows immediately from \eqref{hneqn}.

    If $h=n$ and $a_{n+1,n+1}=2$, then all $G(q)$-orbits involved are non-split except $\sO_{{}_n\!A_{n+1}}$ which splits into two $\cG(q)$-orbits by Proposition \ref{split}. Since $\text{sgn}(\ur{_nA_{n+1}})=\text{sgn}(\ura)$ by \eqref{sgnhapsgna}, we have $\sO_{_n\!A_{n+1}}=\csO({\paep{_n A_{n+1}}})\sqcup\csO({\mamep{_n A_{n+1}}})$ and $\cw({\paep{_n A_{n+1}}})=\nu^\bullet=\cw({\mamep{_n A_{n+1}}})$. Thus, $e_{_n\!A_{n+1}}=\phi_{{\paep{_n A_{n+1}}}}+\phi_{{\mamep{_n A_{n+1}}}}$ in $S^\scd_\bsq(n,r)$. Substituting gives the second case in (1).

(2) %In $S^\scb_{\bsq,1}(n,r)$, we have $e_B*e_A=\sum_{1\leq p\leq N}\gba e_{\apl}$. Now we consider on $S_\bsq^\scd(n,r)$, 
For $A\in \hbb\sqcup\hbh$, applying the orbit splittings in \eqref{sgnhap} to \eqref{hneqn} yields
\begin{equation}    \label{th6.2(1)}
\phi_{\bzero} * (\phi_{\paep{A}}+\phi_{\mamep{A}})=\begin{cases}
    \sum_{1\leq p\leq N}\gba(\phi_{\papsgn}+\phi_{\mapmsgn}), &\text{ if }h<n,\\
    \sum\limits_{1\leq p<n+1}\!\gba (\phi_{\papsgn}+\phi_{\mapmsgn})+\!\sum\limits_{ p>n+1}\!\gba (\phi_{\maep{\apl}}+\phi_{\pamep{\apl}}),&\text{ if }h=n,
\end{cases}
\end{equation}
where we omitted the $p=n+1$ term in the $h=n$ case since the central entry of $_nA_{n+1}=-2$. We want to separate \eqref{th6.2(1)}
into two formulas for $\phi_{\bzero} * \phi_{\paep{A}}$ and $\phi_{\bzero} *\phi_{\mamep{A}}$.

The case for $A\in\hbb$ is easy since this implies all $\apl\in \hbb$. Thus, 
%Consider $ \phi_{\bzero} * \phi_{\paep{A}}=\sum_{A^\prime\in \hbb}g_{\check{\mathbb A^\prime}}\phi_{\check{\mathbb A^\prime}},$
 %we have  
 $$\cw(\paep{A})=\cw(\papsgn)=\cw(\maep{\apl})=\nu^\ep,\;\;\;\cw(\mamep{A})=\cw(\mamep{\apl})=\cw(\pamep{\apl})=\nu^{-\ep}.$$ 
 Hence, by \eqref{phiAB}, multiplying the idempotent $1_{\nu^\ep}$ (resp., $1_{\nu^{-\ep}}$) to the right hand side of \eqref{th6.2(1)} yields the required formula for $\phi_{\bzero} * \phi_{\paep{A}}$
(resp.,  $\phi_{\bzero} * \phi_{\mamep{A}}$).

%Suppose $A\in\hbh$ and $rw(\check {\mathbb A})=\mu^\bullet$. {\color{red}(still thinking about how to prove)}

The case for $A\in\hbh$ cannot be obtained by an idempotent argument as above since all terms on the right hand side of \eqref{th6.2(1)} have a sole column weight $\nu^\bullet$. We need to determine the terms by the definition of convolution products in \eqref{f_Of_O'}.
In other words,
to compute $\phi_{\bzero} * \phi_{\paep{A}}$ and $\phi_{\bzero} *\phi_{\mamep{A}}$, it requires to compute the convolution products
$f_{\csO(\bzero)} * f_{\csO(\paep{A})}$ and $f_{\csO(\bzero)} *f_{\csO(\mamep{A})}$ in $\fkF^\jmath_{\cG(q)}(\sX\times\sX)$ via \eqref{f_Of_O'}. The following claim is sufficient to complete the proof.

%and $(rw(\check {\mathbb A}),cw(\check {\mathbb A}))=(\mu^\bullet,\beta^\bullet)$,  we have all $\apl\in \hbh$. 
%Consider \begin{equation}\label{bbulabul}
%    \phi_{\bzero} * \phi_{\paep{A}}=\sum_{A^\prime\in \hbh}g_{\check{\mathbb A^\prime}}\phi_{\check{\mathbb A^\prime}},
%\end{equation} 
% we have  $cw(\paep{A})=\beta^\bullet$ and $(rw(\check{\mathbb A^\prime}), cw(\check{\mathbb A^\prime}))=(\la^\bullet,\beta^\bullet)$ for all $\check{\mathbb A^\prime}$ in summand. 
\begin{center}
\begin{itemize}
\item[\textbf{Claim}:\!\!]\; For $h\in[1,n]$ and $p\in[1,N]$, if $F, F^\prime, E\in \mathcal F_{n,r}^\jmath$ satisfy $\fkm({F},E)=B$, $\fkm(E,F^\prime)=A$, and $\fkm({F},F^\prime)=\apl$, then
  \begin{equation}\label{bbulabul3}(E,F')\in
{ \csO(^{\ep_1}\!A^{\ep_2})}\implies
        (F,F^\prime)\in\begin{cases}
   \csO( ^{\ep_1}\!{\apl}^{\ep_2}), &\text{ if $h<n$ or $h=n$ and $p\leq n$},\\
   \csO( ^{-\ep_1}\!{\apl}^{\ep_2}), &\text{ if $h=n$ and $p\geq n+2$}.\end{cases}
    \end{equation}
\end{itemize}
\end{center}
By the claim, we see that $\phi_{\bzero} * \phi_{\paep{A}}$ (resp., $\phi_{\bzero} *\phi_{\mamep{A}}$) is a linear combination of $\phi_{\papsgn}$ (resp., $\phi_{\mamep{\apl}}$)  if $h<n$, or a linear combination of $\phi_{\papsgn}$ and $\phi_{\maep{\apl}}$ (resp., $\phi_{\mamep{\apl}}$ and $\phi_{\pamep{\apl}}$)  if $h=n$. Note that, as natural basis elements in $S^\scd_\bsq(n,r)$ defined by distinct $(\cW_{\la^\bullet},\cW_{\mu^\bullet})$ double cosets (see Definition \ref{defkd}(4)), $\phi_{\papsgn}$ and $\phi_{\mamep{\apl}}$ (or $\{\phi_{\papsgn},\phi_{\maep{\apl}}\}$ and $\{\phi_{\mamep{\apl}}, \phi_{\pamep{\apl}}\}$) are linearly independent. Hence, equating with \eqref{th6.2(1)} gives \eqref{bzeropaep}.

%If the claim follows, then $\phi_{\check{\mathbb A^\prime}}$ in the summand of \eqref{bbulabul} should be \eqref{bbulabul2} and thus, the first equation in (2) follows.

It remains to prove the claim.    Since $a=a_{n+1,n+1}=0$ and $\fkm(E,F^\prime)=A$, 
$(E,F')\in
{ \csO(^{\ep_1}\!A^{\ep_2})}$ implies, by Lemma \ref{O_A2},  that  $\ep_11=(-1)^{d_{E,F'}-r}$ and $\ep_21=(-1)^{d_{F',E}-r}$.% (see Lemma \ref{O_A2}).

 %   $\epsilon_1$ and $\epsilon_2$ is determined by $M_{E,F'}:=(\tilde{F}_n+\tilde{F}_{n+1}\cap F_{n}^{'})\cap M_r$ and $M_{F,F'}:=({F}^{'}_n+{F}^{'}_{n+1}\cap \tilde{F}_{n})\cap M_r$, respectively. Similarly, the corresponding sign on the left and right corner of $\apl$ is determined by $F_n+F_{n+1}\cap F_{n}^{'}$ and $F^{'}_n+F^{'}_{n+1}\cap F_{n}$, respectively. Therefore, we only need to compare the difference between $F_n+F_{n+1}\cap F_{n}^{'}$ and $\tilde{F}_n+\tilde{F}_{n+1}\cap F_{n}^{'},$ $F^{'}_n+F^{'}_{n+1}\cap F_{n}$ and ${F}^{'}_n+{F}^{'}_{n+1}\cap \tilde{F}_{n}$.
   If $h<n$, then  $E\overset1\subset_hF$ ($\iff\fkm({F},E)=B$ by Corollary \ref{Cormulti}) implies $F_j=E_j$, for $j\geq n>p$, and so,
    \begin{equation}\label{h<n}
    F_n+F_{n+1}\cap F_{n}^{'}=E_n+E_{n+1}\cap F_{n}^{'},\ F^{'}_n+F^{'}_{n+1}\cap F_{n}={F}^{'}_n+{F}^{'}_{n+1}\cap E_{n}.\end{equation} 
    Thus, $d_{F,F'}=d_{E,F'}$ and $d_{F',F}=d_{F',E}$. Hence, $(F,F')\in\csO({}^{\epsilon_1}\hspace{-1mm}\apl^{\epsilon_2})$. 
  
 %   Since $(E,F^\prime)\in ^{\ep_1}\!\mathcal{O}_A^{\ep_2}$, we have $\ep_1$ and ${\ep_2}$ is determined by $E_n+E_{n+1}\cap F_{n}^{'}$ and ${F}^{'}_n+{F}^{'}_{n+1}\cap E_{n}$, respectively. Similarly, the left and right sign of $_h{\check{\mathbb A}}_p$ is determined by $F_n+F_{n+1}\cap F_{n}^{'}$ and $F^{'}_n+F^{'}_{n+1}\cap F_{n}$, respectively. Therefore, we only need to compare the difference between $F_n+F_{n+1}\cap F_{n}^{'}$ and $E_n+E_{n+1}\cap F_{n}^{'},$ $F^{'}_n+F^{'}_{n+1}\cap F_{n}$ and ${F}^{'}_n+{F}^{'}_{n+1}\cap E_{n}$.

  %  If $h<n$, then $_h{\check{\mathbb A}}_p$ should be  $^{\ep_1}\hspace{-1mm}{\apl}^{{\ep_2}}$ since \begin{equation}\label{L}
   % F_n+F_{n+1}\cap F_{n}^{'}=E_n+E_{n+1}\cap F_{n}^{'},\ F^{'}_n+F^{'}_{n+1}\cap F_{n}={F}^{'}_n+{F}^{'}_{n+1}\cap E_{n}.\end{equation} 
    
    Now assume $h=n$ and $F_n=E_n+\lr{e}$. Then $e\in E_{n+1}$ since $F_n^\perp=F_{n+1}=E_{n+1}= E_{n}^\perp$. Note also that the hypothesis in the Claim implies that the index $p$ satisfies the conditions:
    \begin{equation}\label{condition on p}
    F_n\cap F_j^{'}=E_n\cap F_j^{'}\text{ if }j<p\text{ and }F_n\cap F_j^{'}\neq E_n\cap F_j^{'}\text{ if }j\geq p.
    \end{equation}
     Thus,
    if $p\leq n$, then $F_n\cap F_n^{'}\neq E_n\cap F_n^{'}$ and $F_{n}\cap F_{n+1}^{'}\neq E_{n}\cap F_{n+1}^{'}$. 
   Hence, $e\in F^{'}_n$ and consequently, \eqref{h<n} continue to hold. This implies  $(F,F')\in\csO(^{\ep_1}\hspace{-1mm}{{}_n\!A_p}\!^{{\ep_2}})$ for all $p\leq n$.

    Finally, assume $p\geq n+2$. Then $F_n\cap F_{n+1}^{'}=E_n\cap F_{n+1}^{'}$ and so, $e\notin F^{'}_{n+1}$. Thus, $F^{'}_{n+1}\cap F_{n}={F}^{'}_{n+1}\cap  E_{n}$. Hence,
    $$e\in (F_n+F_{n+1}\cap F_{n}^{'})-( E_n+ E_{n+1}\cap F_{n}^{'}),\text{ and } F^{'}_n+F^{'}_{n+1}\cap F_{n}={F}^{'}_n+{F}^{'}_{n+1}\cap  E_{n}.$$
    Since the central entries of $A$ and $\apl$ are 0, this display continues to hold if $F'_n$ is replaced by $F'_{n+1}$. Hence, we have 
    $d_{F,F'}=d_{E,F'}+1$ and $d_{F',F}=d_{F',E}$. (This can be seen by taking $(F,E)=(F^\wla,F^\wmu)$ and $(F,F')=(F^\wla, F^{A'})$. Thus, $F_n=E_n+\lr{e}$ with $e=e_m, m\in I'_{n,p}$; see Proposition \ref{new}.) 
    Consequently, $(F,F')\in\csO(^{-\ep_1}\hspace{-1mm}{{}_n\!A_p}\!^{{\ep_2}})$ for all $p> n+1$.\end{proof}

\subsection{The case $B\in\bbb$ (\!$\implies h<n$)} In this case, $\sO_B=\csO(^+\!B^+)\sqcup \csO(^-\!B^-)$. Thus, it suffices to consider those $A$ in \eqref{hneqn} with $A\in\bbb\sqcup\bbh$.

\begin{thm}\label{th6.3}
Maintain the same assumptions on $A,B,h,\la,\mu,\nu$, and $\ep=\text{\rm sgn }(\urcm)$ as in Theorem \ref{thmhhh}, and assume $B=E^\th_{h,h+1}+\wga\in\bbb$.  Then $h<n$ and, for $A\in\bbh\sqcup\bbb$, the following multiplication formulas hold in $S_\bsq^\scd(n,r)$:

\begin{equation}\label{bbbb}
   \phi_{\pap{B}}*\phi_{\paep{A}}=%\begin{cases}
    \sum_{1\leq p\leq N}\gba \phi_{\papsgn},% &\text{ if }h<n,\\
   % \sum \limits_{1\leq p<n+1}\!\gba \phi_{\paep{\apl}}+ \sum \limits_{ p>n+1}\!\gba \phi_{\pamep{\apl}},&\text{ if }h=n,
%\end{cases}
\qquad
\phi_{\mam{B}}*\phi_{\mamep{A}}=%\begin{cases}
    \sum_{1\leq p\leq N}\gba \phi_{\mapmsgn}.%, &\text{ if }h<n,\\
  %  \sum \limits_{1\leq p<n+1}\!\gba \phi_{\mamep{\apl}}+\sum \limits_{ p>n+1}\!\gba \phi_{\maep{\apl}}, &\text{ if }h=n,
%\end{cases}
\end{equation} 
\end{thm}

\begin{proof}The assertion $h<n$ is clear since the central component of $\co(B)$ is at least 2 when $h=n$.
    By the hypothesis, we have $(\rw(\pap{B}), \cw(\pap{B}))=(\la^+,\mu^+)$ and $(\rw(\mam{B}), \cw(\mam{B}))=(\la^-,\mu^-)$.
    If $A\in\bbh\sqcup\bbb$, then $\sO_A=\csO(\paep{A})\sqcup\csO(\mamep{A})$ with
    $$(\rw(\paep{A}),\cw(\paep{A}))=\begin{cases}(\mu^+,\nu^{\ep}),&\text{ if }A\in\bbb;\\
    (\mu^+,\nu^\bullet),&\text{ if }A\in\bbh.\end{cases}\qquad
    (\rw(\csO(\mamep{A}),\cw(\csO(\mamep{A})))=\begin{cases}(\mu^-,\mu^{-\ep}),&\text{ if }A\in\bbb;\\
    (\mu^-,\nu^{\bullet}),&\text{ if }A\in\bbh.\end{cases}$$

 %   $_n A_{n+1}$ does not exist since $(_n A_{n+1})_{n+1,n+1}=-2$. So we will omit $\phi_{_n A_{n+1}}$ in the following arguement. Also, all $\apl$ except $_n A_{n+1}$ satisfy $(\apl)_{n+1,n+1}=0$.

By \eqref{sgnhap}(2), \eqref{hneqn} has the following decomposition in $S_\bsq^\scd(n,r)$, 
\begin{equation}\label{bbbeq1}
    (\phi_{\pap{B}}+\phi_{\mam{B}})*(\phi_{\paep{A}}+\phi_{\mamep{A}})=%\begin{cases}
        \sum_{1\leq p\leq N}\gba (\phi_{\paep{\apl}}+\phi_{\mamep{\apl}}).%, &\text{ if }h<n,\\
      %  \sum_{1\leq p<n+1}\gba (\phi_{\paep{\apl}}+\phi_{\mamep{\apl}})\\
    %  \qquad\qquad\qquad+\sum_{p>n+1}\gba(\phi_{\pamep{\apl}}+\phi_{\maep{\apl}}), &\text{ if }h=n.\end{cases}
\end{equation}
Clearly, LHS$= \phi_{\pap{B}}*\phi_{\paep{A}}+\phi_{\mam{B}}*\phi_{\mamep{A}}$ since $\phi_{\pap{B}}*\phi_{\mamep{A}}=0=\phi_{\mam{B}}*\phi_{\paep{A}}$. Also, every natural basis element $\phi_{\cBA'}$ on the right hand side has row weight
$\rw(\cBA')=\la^+$ or $\la^-$.
Thus,  \eqref{bbbb} is obtained by multiplying idempotents $1_{\la^+}$ or
$1_{\la^-}$ to the left hand side of \eqref{bbbeq1}.
\end{proof}
 % If $A\in \bbh$, then the row and column weight pair of $\paep{A}$ and $\mamep{A}$ is $(\mu^+,\beta^\bullet )$ and $(\mu^-,\beta^{\bullet} )$, respectively. 
%   In $S^\scb_{\bsq,1}(n,r)$, we have $e_B*e_A=\sum_{1\leq p\leq N}\gba e_{\apl}$. 

%where for all $\check{\mathbb A}^\prime$ in the right hand side summand, we have
%$(rw(\check{\mathbb A}^\prime), cw(\check {\mathbb A}^\prime))$ is $(\la^+, \beta^\bullet)$ or $(\la^-, \beta^{\bullet})$.

%We will first prove the first equation $\phi_{\pap{B}}*\phi_{\paep{A}}$ in \eqref{bbbb} where $\ep=\text{sgn }(\urcm)$. Following the argument above, we have all $_h{\check {\mathbb A}}_p$ in the summand of $\phi_{\pap{B}}*\phi_{\paep{A}}$ satisfy the row and column weight pair be $(\la^+, \beta^\bullet)$. Thanks to \eqref{sgnhapsgna} and Definition~\ref{defkd}(2),  we have the right hand side summand in \eqref{bbbeq1} with row and column weight pair be $(\la^+, \beta^\bullet)$ should be 
%$\begin{cases}
%   \phi_{\papsgn}, &\text{ if $h<n$ and $p\in[1,N]$ or  $h=n$ and $\in[1,n]$},\\
%  \phi_{\pamep{\apl}},&\text{ if $h=n$ and $p>n+1$},\\
%\end{cases}$
%and the first equation in \eqref{bbbb} follows. 

%The proof for $\phi_{\mam{B}}*\phi_{\mamep{A}}$ is similar, the only difference is to find $_h{\check {\mathbb A}}_p$ satisfy the row and column weight pair be $(\la^-, \beta^\bullet)$.
%The result $\phi_{\pap{B}}* \phi_{\mamep{A}}=\phi_{\mam{B}}* \phi_{\paep{A}}=0$ holds since $cw(\pap{B})\neq rw(\mamep{A})$ and $cw(\mam{B})\neq rw(\paep{A})$. Thus, we finish the proof.

\subsection{The case $B\in\bbh$(\!\!$\implies h=n$)} We have in this case $\sO_B=\csO(^+\!B^+)\sqcup \csO(^-\!B^-)$ with $(\rw(^+\!B^+),\cw(^+\!B^+))=(\la^+,\mu^\bullet)$ and $(\rw(^-\!B^-),\cw(^-\!B^-))=(\la^-,\mu^\bullet)$. Thus, it suffices to consider those $A$ with $A\in\hhh\sqcup\hbh\sqcup\hbb$.
%Recall the notation in Corollary \ref{Cormulti}.

\begin{thm}\label{thbbh}
Let $A,B\in\Xi_{N,2r}$ and $\ep=\text{\rm sgn }(\urcm)$, and assume $\co(B)=\ro(A)$, $B\in\bbh$, and $B-E^\th_{n,n+1}$ is diagonal.  Then the following multiplication formulas hold in $S_\bsq^\scd(n,r)$:
%Suppose $B\in\bbh$, $(rw(\pap{B}), cw(\pap{B}))=(\la^+,\mu^\bullet)$ and $(rw(\mam{B}), cw(\mam{B}))=(\la^-,\mu^\bullet)$.
\begin{itemize}
    \item[(1)]If $A\in\hhh$, then $\sO_A=\csO(\dot A)$, and 
   $$ \phi_{\pap{B}}* \phi_{\azero}=g_{_{n,A,{n+1}}}
  \phi_{\paep{_n A_{n+1}}},\qquad \phi_{\mam{B}}* \phi_{\azero}=g_{_{ n,  A,{n+1}}}\phi_{\mamep{_n A_{n+1}}}.$$
  \item[(2)] If $A\in\hbh\sqcup\hbb$,  then $\sO_A=\csO(\paep{A})\sqcup\csO(\mamep{A})$ and
  \begin{equation}\begin{aligned}\nonumber
   \phi_{\pap{B}}*\phi_{\paep{A}}&=%\begin{cases}
    %\sum_{1\leq p\leq N}\gba \phi_{\papsgn}, &\text{ if }h<n,\\
  \sum \limits_{1\leq p<n+1}\!g_{_n,A,p} \phi_{\paep{_n\!A_p}},%&\text{ if }h=n,\end{cases}
\\
\phi_{\mam{B}}*\phi_{\mamep{A}}&=%=\begin{cases}
  %  \sum_{1\leq p\leq N}\gba \phi_{\mapmsgn}, &\text{ if }h<n,\\
  \sum \limits_{1\leq p<n+1}\!g_{n,A,p} \phi_{\mamep{\apl}}, %&\text{ if }h=n,\end{cases}
   \\
\phi_{\pap{B}}* \phi_{\mamep{A}}&=%\begin{cases}
    %0, &\text{ if }h<n,\\
   \sum \limits_{ p>n+1}\!g_{n,A,p} \phi_{\pamep{\apl}}, %&\text{ if }h=n,\end{cases}
    \\
\phi_{\mam{B}}* \phi_{\paep{A}}&=%\begin{cases}
   % 0,&\text{ if }h<n,\\
\sum \limits_{ p>n+1}\!g_{n,A,p} \phi_{\maep{\apl}}.%&\text{ if }h=n.\end{cases}
  \end{aligned}
\end{equation} 
\end{itemize}\end{thm}

\begin{proof}Recall in this case $B=E^\th_{n,n+1}+\wga$ with $\wla=\ro(B),\wmu=\co(B)=\ro(A)$, and $\nu=\co(A)$.  % If $h<n$, then  $\text{co}(B)_{n+1}=0$. Thus, $e_B*e_A=0$ for $h<n$ and $.
  %If $B\in\bbh$, we must have $h=n$.

    (1) If $A\in\hhh$, then $\text{ro}(A)_{n+1}\neq0$ and $(\rw(\azero), \cw(\azero))=(\mu^\bullet, \nu^\bullet).$
Since $\la_{n+1}=0$ and $\text{ro}(A)_{n+1}=\text{co}(B)_{n+1}=2$, it follows that $a_{n+1,n+1}=2$ and $a_{n+1,p}=0$, for all $p\neq n+1$. Thus, the $(n+1,p)$ entry of $\apl$ is negative, for all $p\neq n+1$. Hence, \eqref{hneqn} becomes 
    $e_{{B}}* e_{A}=g_{ n,  A,{n+1}}
  e_{{_n A_{n+1}}}$.  
 Since the central entry of $_nA_{n+1}$ is 0, we have $\sO_{_nA_{n+1}}=\csO(\paep{_n A_{n+1}})\sqcup\csO(\mamep{_n A_{n+1}})$
 splits into two orbits and $\text{sgn}(\ur{_nA_{n+1}})=\text{sgn}(\ura)$, we have in $S_\bsq^\scd(n,r)$,  
 $$(\phi_{\pap{B}}+\phi_{\mam{B}})* \phi_{\azero}=g_{_{ B,  A,_n A_{n+1}}}
  (\phi_{\paep{_n A_{n+1}}}+\phi_{\mamep{_n A_{n+1}}}).$$
  Multiplying idempotents $1_{\la^+}$ or
$1_{\la^-}$ to the left hand side gives (1).

  (2) If $A\in\hbh\sqcup\hbb$, then the central entry of $_n A_{n+1}$ is $-2$. Thus, the term $e_{_n A_{n+1}}$in \eqref{hneqn}  is omitted.
 For every other term $\apl$, its central entry is 0. Thus, the corresponding $G$-orbits are all split. Applying  \eqref{sgnhap} to \eqref{hneqn} yields the following formulas in $S_\bsq^\scd(n,r)$, 
 \begin{equation}\label{bbheq1}
 \aligned
    (\phi_{\pap{B}}+\phi_{\mam{B}})*&(\phi_{\paep{A}}+\phi_{\mamep{A}})=\phi_{\pap{B}}*\phi_{\paep{A}}+\phi_{\mam{B}}*\phi_{\paep{A}}+\phi_{\pap{B}}*\phi_{\mamep{A}}+\phi_{\mam{B}}*\phi_{\mamep{A}}\\
&= %   \begin{cases}
      %  \sum_{1\leq p\leq N}\gba (\phi_{\paep{\apl}}+\phi_{\mamep{\apl}}), &\text{ if }h<n,\\
        \sum_{1\leq p<n+1}g_{n,A,p} (\phi_{\paep{\apl}}+\phi_{\mamep{\apl}}) +\sum_{p>n+1}g_{n,A,p}(\phi_{\pamep{\apl}}+\phi_{\maep{\apl}}),% &\text{ if }h=n,
    %\end{cases}
    \endaligned
\end{equation}

 If $A\in\hbb$, then both sides have the row weight $\la^+$ or $\la^-$ and the column weight $\nu^+$ or $\nu^-$.
 Multiplying idempotents $1_{\la^\pm}$ and $1_{\nu^\pm}$ to left and right hand sides of \eqref{bbheq1}, respectively, gives (2) in this case.

% $$\aligned
% (\rw(\paep{\apl}),\cw(\paep{\apl}))&=(\la^+,\nu^\ep)\quad\text{ and }\quad(\rw(\mamep{\apl}),\cw(\mamep{\apl}))=(\la^-,\nu^{-\ep})\\
% (\rw(\pamep{\apl}),\cw(\pamep{\apl}))&=(\la^+,\nu^{-\ep})\quad\text{ and }\quad(\rw(\maep{\apl}),\cw(\maep{\apl}))=(\la^-,\nu^{\ep}).\\
% \endaligned$$ 

 Finally, let $A\in\hbh$.  Then both sides still have the row weight $\la^+$ or $\la^-$, but a sole column weight $\nu^\bullet$. We cannot use only weights to separate them. However, we may combined it with the Claim in the proof of Theorem \ref{thmhhh}(2).
%  $$\aligned
% (\rw(\paep{\apl}),\cw(\paep{\apl}))&=(\la^+,\nu^\bullet)\quad\text{ and }\quad(\rw(\mamep{\apl}),\cw(\mamep{\apl}))=(\la^-,\nu^\bullet)\\
% (\rw(\pamep{\apl}),\cw(\pamep{\apl}))&=(\la^+,\nu^\bullet)\quad\text{ and }\quad(\rw(\maep{\apl}),\cw(\maep{\apl}))=(\la^-,\nu^\bullet).\\
% \endaligned$$ 
 % and the row and column weight of $\paep{A}$ and $\mamep{A}$ are both $(\mu^\bullet,\beta^\bullet)$.
 
Multiplying $1_{\la^+}$ to the left hand side of \eqref{bbheq1} yields that $\phi_{\pap{B}}*\phi_{\paep{A}}+\phi_{\pap{B}}*\phi_{\mamep{A}}$ is a linear combination of $\phi_{\cBA'}$ with 
%For all $\phi_{_h{\check {\mathbb A}}_p}$ in the summand of $\phi_{\pap{B}}*\phi_{\paep{A}}$, the row and column weight pair of $_h{\check {\mathbb A}}_p$ should be $(\la^+, \beta^\bullet)$. By \eqref{bbheq1} and Definition~\ref{defkd}(2), all $\phi_{_h{\check {\mathbb A}}_p}$ are in
\begin{equation*}\begin{aligned}
    \label{bbheq2}
    %&\cBA'\in\{{\paep{\apl}} |p\in[1,N] \},  &\text{ if }& h<n, \text{ or }\\
    &\cBA'\in\{{\paep{\apl}} |p\in[1,n+1] \}\cup \{{\pamep{\apl}} |p>n+1 \}.%, & \text{ if } & h=n. 
\end{aligned}\end{equation*} 
By the Claim in the proof of Theorem~\ref{thmhhh}, $\phi_{\pap{B}}*\phi_{\paep{A}}$ is a linear combination of $\phi_{\cBA'}$ with
\begin{equation*}\begin{aligned}
    \label{bbheq3}
    %&{\cBA'}\in\{{\paep{\apl}} |p\in[1,N] \},  &\text{ if }& h<n,\text{ or }\\
    &{\cBA'}\in\{{\paep{\apl}} |p\in[1,n+1] \}\cup \{{\maep{\apl}} |p>n+1 \}.%, & \text{ if } & h=n. 
\end{aligned}\end{equation*} 
Hence, we have $$\phi_{\pap{B}}*\phi_{\paep{A}}=%\begin{cases}    \sum_{1\leq p\leq N}\gba \phi_{\papsgn}, &\text{ if }h<n,\\
   \sum \limits_{1\leq p<n+1}\!\gba \phi_{\paep{\apl}},%&\text{ if }h=n,\end{cases}
   $$
and, consequently, we have $$\phi_{\pap{B}}* \phi_{\mamep{A}}=%\begin{cases}  0, &\text{ if }h<n,\\
    \sum \limits_{ p>n+1}\!\gba \phi_{\pamep{\apl}}.% &\text{ if }h=n.\end{cases}
    $$

The case for $\phi_{\mam{B}}*\phi_{\mamep{A}}$ and $\phi_{\mam{B}}* \phi_{\paep{A}}$ can be proved similarly.
\end{proof}
% for $\check B\in \{\pap{B}, \mam{B}, \bzero\}$ and $\cBA\in\check\Xi(n,r)$ satisfying $\rw(\cBB)=\lw(\cBA)$. We first determine possible $\cBA$ and $\cBM$ to occur in these formulas.

\section{Multiplication formulas in the $\bsq$-Schur algebra $S_\bsq^\scd(n,r)$, II}
%(\uppercase\expandafter{\romannumeral2})}
In this section, we derive the multiplication in $S_\bsq^\scd(n,r)$ formulas arising from the one in Theorem~\ref{sjcon}(2): 
\begin{equation}\label{ecea}
    e_{E_{h+1,h}^\theta+\wga} * e_{A}=\delta_{\wga,\ro(A)-e^\theta_{h}}\sum \limits_{1 \leqslant p \leqslant N} \gca e_{\apu},
\end{equation}
where $h\in[1,n]$, $\ga\in\Lambda(n+1,r), A=(a_{i,j})\in \Xi_{N,2r}$, $\apu=A-E_{h,p}^\theta+E^{\theta}_{h+1,p}$, for all $p\in[1,N]$, and 
\begin{equation}\label{coefg'}\gca=\begin{cases}
   \bsq^{\sum_{j<p} a_{h+1,j}} \ggi{a_{h+1,p}+1} ,&\text{ if $h<n$ or }h=n, p\neq n+1\text{ and }a_{h,p}>0;\\
\bsq^{\sum\limits_{j<n+1} a_{n+1,j}}\big( \ggi{a_{n+1,n+1}+1}+\bsq^{\frac{a_{n+1,n+1}}2}\big) ,&\text{ if }h=n, p=n+1, \text{ and }a_{n,n+1}>0;\\
0,&\text{ if } a_{h,p}=0.
\end{cases}
\end{equation}

In \eqref{ecea}, let $ C=E_{h+1,h}^{\theta}+\wga$ and assume $\la=\text{ro}(B)$ and  $\mu=\text{ro}(A)=\text{co}(B)$. Then the upper right corner matrix $C^{\text{\normalsize $\llcorner$}}=0$ (and so, $\text{sgn}(C^{\text{\normalsize $\llcorner$}})=+$). Thus, by Proposition~\ref{split}, we have 
$${\mathcal{O}}_C=\begin{cases}
    \check{\mathcal{O}}(\pap{C})\sqcup\check{\mathcal{O}}(\mam{C}), &\text{ if }c_{n+1,n+1}=0,\\
    \qquad  \check{\mathcal{O}}(\dot{C}), &\text{ otherwise}
\end{cases}\quad\text{ and }\quad \begin{cases}
    C\in \bbb\sqcup \hhh, &\text{ if }h<n,\\
    C\in \hbb\sqcup \hhh,  &\text{ if }h=n.
\end{cases}$$

Since $\apu=A-E_{h,p}^\theta+E^{\theta}_{h+1,p}$, it follows that 
\begin{equation}
   \text{sgn}(\apu^{\text{\normalsize $\llcorner$}})=\begin{cases}
       \text{sgn}(\urcm), \ &\text{ if $h<n$ or $h=n$ and $1\leq p\leq n+1$},\\
       -\text{sgn}(\urcm), \ &\text{ if $h=n$ and $p>n+1$}.
   \end{cases}
\end{equation} 
If $a_{n+1,n+1}=0$ and $\ep=\text{sgn }(\urcm)$, then we have $G(q)$-orbits $\mathcal{O}_A,$ $\mathcal{O}_{\apu}$ splitting into $\cG(q)$-orbits:
\begin{equation}\label{oapu}\begin{aligned}
(1)\quad  \mathcal O_{A}&=\check{\mathcal O}(\paep{A})\sqcup \check{\mathcal O}(\mamep{A}),\\
(2)\  \mathcal O_{\apu}&=\begin{cases}
    \check{\mathcal O}(\paep{\apu})\sqcup \check{\mathcal O}(\mamep{\apu}), &\text{ if $h<n$ or $h=n$ and }1\leq p\leq n+1,\\
    \check{\mathcal O}(\pamep{\apu})\sqcup \check{\mathcal O}(\maep{\apu}), &\text{ if $h=n$ and }p>n+1.
\end{cases}
\end{aligned}\end{equation}

%{\color{red}(Does not use in the following Theorems)For simplicity, we fix the following notational abbreviation for the central entries of $A$ and $C$:
%$$a:=a_{n+1,n+1}, \qquad c:=c_{n+1,n+1}.$$}

\subsection{The $C\in\hhh$ case.} Since ${\mathcal{O}}_C=\check{\mathcal{O}}(\dot{C})$ in this case and $\cw(\czero)=\mu^\bullet\in \Lambda^\bullet(n+1,r)$, it suffices to consider those $A$ in \eqref{ecea} with $A\in \hhh\sqcup \hbb \sqcup \hbh$. The following result is the counterpart of Theorem \ref{thmhhh}. Note that the two cases  $A\in \hbb$ and $A\in \hbh$ are quite different from the corresponding cases there.

\begin{thm}\label{thmchhh}Let $A,C\in \Xi_{N,2r}$ with $ C=E_{h+1,h}^{\theta}+\wga$, and assume $\wla=\ro(C)$, $ \wmu=\co(C)=\ro(A)$, and $\wnu=\co(A)$, where $h\in[1,n]$, and $\ga,\la, \mu, \nu\in \Lambda(n+1,r).$ Let  $\ep=\text{sgn }(\urcm)$. If $C\in\hhh$, then ${\mathcal{O}}_C=\check{\mathcal{O}}(\dot{C})$ and the following multiplication formulas hold in $S_\bsq^\scd(n,r)$:
    \begin{itemize}
        \item [(1)]If $A\in\hhh$, then ${\mathcal{O}}_A=\check{\mathcal{O}}(\dot{A})$ and
  $  \phi_{\czero} * \phi_{\azero}=
  \sum_{1 \leqslant p \leqslant N} \gca \phi_{ _h\!{\dot A}_{\overline {p}}}.$

%\item[(2)] If $A\in \hbb$,  then
% \begin{equation}\begin{aligned}\nonumber
% \phi_{\czero} * \phi_{\paep{A}}
%=&\begin{cases}\sum\limits_{1\leq p\leq N}\gca \phi_{\paep{\apu}}, &\text{ if }h<n,\\
%\sum\limits_{1\leq p<n+1}\!\gca \phi_{\paep{\apu}}+\!\sum\limits_{ p>n+1}\!\gca \phi_{\maep{\apu}}, &\text{ if }h=n,\end{cases}
%\\
%\phi_{\czero} * \phi_{\mamep{A}}
%=&\begin{cases}
%\sum\limits_{1\leq p\leq N}\gca \phi_{\mapmsgn}, &\text{ if }h< n,\\
%\sum\limits_{1\leq p<n+1}\!\gca \phi_{\mamep{\apu}}+\hspace{-2mm}\sum\limits_{ p>n+1}\hspace{-2mm}\gca \phi_{\pamep{\apu}}, &\text{ if }h=n.
%\end{cases} \end{aligned}\end{equation}

\item [(2)] If $A\in\hbb\sqcup \hbh$,  then
 \begin{equation}\begin{aligned}\nonumber
 \phi_{\czero} * \phi_{\paep{A}}
=&\begin{cases}\sum\limits_{1\leq p\leq N}\gca \phi_{\paep{\apu}}, &\text{ if }h<n,\\
\sum\limits_{1\leq p<n+1}\!\gca \phi_{\paep{\apu}}+\frac{1}{2}{g'_{n,A, {\overline{n+1}}}}\phi_{_n {\dot A}_{\overline{n+1}}}+\!\sum\limits_{ p>n+1}\!\gca \phi_{\maep{\apu}}, &\text{ if }h=n,\end{cases}
\\
\phi_{\czero} * \phi_{\mamep{A}}
=&\begin{cases}
\sum\limits_{1\leq p\leq N}\gca \phi_{\mapmsgn}, &\text{ if }h< n,\\
\sum\limits_{1\leq p<n+1}\!\gca \phi_{\mamep{\apu}}+\frac{1}{2}{g'_{n,A,{\overline{n+1}}}}\phi_{_n {\dot A}_{\overline{n+1}}}+\hspace{-2mm}\sum\limits_{ p>n+1}\hspace{-2mm}\gca \phi_{\pamep{\apu}}, &\text{ if }h=n.
\end{cases} \end{aligned}\end{equation}
    \end{itemize}
Here all coefficients $\gca$ are given in \eqref{coefg'}.
\end{thm}
\begin{proof}
(1) Unlike the case in Theorem \ref{thmhhh}(1), we have all $\apu\in\hhh$ in this case. Thus, its proof is almost identical to the first paragraph of the corresponding proof there.
%For  $A\in\hhh$, by the hypotheses, we have $(\text{rw}( \azero), \text{cw}(\azero))=(\mu^\bullet, \nu^\bullet)$.
%For all $h\leq n$ and $1\leq p\leq N$, the central entry of $\apu$ is nonzero, it follows that $\mathcal O_{\apu}=\check{\mathcal O}(\dot{\apu})$ is non-split. Hence, we have $e_C=\phi_{\czero}$, $e_A=\phi_{\azero}$ and $e_{\apu}=\phi_{\dot{\apu}}$, then (1) follows from Theorem~\ref{sjcon}(2).

(2)
If $A\in \hbh\sqcup\hbb$, then $_nA_{\overline{n+1}}=A-E^\th_{n,n+1}+E^\th_{n+1,n+1}$ has 2 at the central entry and a negative $(n,n+1)$ entry if $A\in\hbb$. All other $\apu$ has a 0 central entry. Thus,
applying \eqref{oapu} to \eqref{ecea} yields
\begin{equation}\label{chhhah1}
   \phi_{\czero} * (\phi_{\paep{A}}+\phi_{\mamep{A}})=\begin{cases}
    \sum\limits_{1\leq p\leq N}\gca(\phi_{\paep{\apu}}+\phi_{\mamep{\apu}}), &\text{ if }h<n,\\
    \sum\limits_{1\leq p<n+1}\!\gca (\phi_{\paep{\apu}}+\phi_{\mamep{\apu}})
    +g'_{n,A,\overline{n+1}}\phi_{_n {\dot A}_{\overline{n+1}}}&\\
    \qquad\qquad\qquad
    +\sum\limits_{ p>n+1}\!\gca (\phi_{\maep{\apu}}+\phi_{\pamep{\apu}}),&\text{ if }h=n.
\end{cases}
\end{equation}
%the constant ${\sf c}=\begin{cases}0,&\text{ if }A\in\hbb;\\1,&\text{ if }A\in\hbh.\end{cases}$

 % since $\nu_{n+1}=0$, it follows that $_nA^{\overline{n+1}}$ has a negative entry at $(n,n+1)$. 
% then all $\apu$ on the right hand side of \eqref{chhhah1} should have the column weight $\nu^+$ or $\nu^-$. Therefore, $\phi_{_n {\dot A}_{\overline{n+1}}}$ does not exist in this case since cw$(_n {\dot A}_{\overline{n+1}})=\nu^\bullet$, so we will omit the $p=n+1$ term in the $h=n$ case.
If $A\in\hbb$, then $g'_{n,A,\overline{n+1}}=0$. % and all other $\apu\in\hbb$. 
Thus, this case is proved by using an idempotent argument as in the proof of Theorem \ref{thmhhh}(2). 
%$$\text{cw}(\paep{A})=\text{cw}(\paep{\apu})=\text{cw}(\maep{\apu})=\nu^\ep,\ 
%\text{cw}(\mamep{A})=\text{cw}(\mamep{\apu})=\text{cw}(\pamep{\apu})=\nu^{-\ep}.$$
%Multiplying the idempotent $1_{\nu^\ep}(\text{resp., }1_{\nu^{-\ep}})$ on the right hand side of \eqref{chhhah1}, and $\phi_{\czero} * \phi_{\paep{A}}(\text{resp., }\phi_{\czero} *\phi_{\mamep{A}})$ follows as (2).

 The case for $A\in\hbh$ is a bit more complicated just like its counterpart in Theorem \ref{thmhhh}. First, we need to modify the Claim there as follows:
\begin{center}
\begin{itemize}
\item[\textbf{Claim}\!\!]{\bf 1.} \; For $h\in[1,n]$ and $p\in[1,N]$, if $F, F^\prime, E\in \mathcal F_{n,r}^\jmath$ satisfy $\fkm({F},E)=C$, $\fkm(E,F^\prime)=A$ and $\fkm({F},F^\prime)=\apu$, then   \begin{equation}\nonumber
        (E,F^\prime)\in \mathcal {\check O}(^{\ep_1}\!A^{\ep_2}) \Longrightarrow
        (F,F^\prime)\in\begin{cases}
    \mathcal {\check O}(^{\ep_1}\!{\apu}^{\ep_2}), &\text{ if $h<n$ or $h=n$ and $p\leq n$};\\
    \mathcal {\check O}({_n {\dot A}_{\overline{n+1}}}), &\text{ if $h=n$, $p=n+1$ (and $a_{n,n+1}\geq1$)};\\
    \mathcal {\check O}(^{-\ep_1}\!{\apu}^{\ep_2}), &\text{ if $h=n$ and $p\geq n+2$}.\end{cases}
    \end{equation}
    \end{itemize}
    \end{center}
The proof of the claim is similar to that of the Claim in subsection 6.1 with the roles of $F$ and $E$ swapped.

 Next, we need to split the coefficient $g'_{n,A,\overline{n+1}}$ in \eqref{chhhah1}. 
\medskip
\begin{itemize}
\item[\textbf{Claim}\!\!]{\bf 2.} \; For $(F,F')\in\sO_{_n\!A_{\overline{n+1}}}$, the set $\sF=\{E\in\sF^\jmath_{n,r}\mid (F,E)\in\sO_C,(E,F')\in\sO_{_n\!A_{\overline{n+1}}}\}$, which defines $g'_{n,A,n+1}$,  is a disjoint union $\sF=\sE\sqcup\hat\sE$ of two subsets with equal cardinality, where
$$\aligned\sE&=
\{E\in\sF^\jmath_{n,r}\mid (F,E)\in\csO(\dot C),(E,F')\in\csO({}^+\!A^\ep)\},\\
\sE'&=
\{E\in\sF^\jmath_{n,r}\mid (F,E)\in\csO(\dot C),(E,F')\in\csO({}^-\!A^{-\ep})\}.\endaligned$$
 \end{itemize}

\medskip    
Indeed, the disjoint union is clear. To see the assertion of equal cardinality, we construct a bijection 
$$f:\sE\longrightarrow\sE',\;\;\; E\longmapsto E'.$$
We now follow the notation used in Proposition \ref{new}. Thus, $F_i=\lr{v_1,\ldots,v_{\widetilde\al_i}}$ $(\al=\wla)$ and $F'=F^{A'}$, where
$A'=_n\!A_{\overline{n+1}}$ has central entry 2.
For $E\in\sF$, by Corollary \ref{claim}, we have $E_n=F_n+\lr{v}$, where $v$ has the form $v=c_1v_{\tla_{n}+1}+\cdots+c_mv_{\tla_n+m}+c_{r}v_{r}+c_{r+1}v_{r+1}$ with $m=a_{n+1,1}+\cdots+a_{n+1,n}$ and either $c_r=0$ or $c_{r+1}=0$, but not both 0. 

Suppose now $E\in\sE$ with $c_r\neq0$ (we may assume $c_r=1$)  and define
$E_{n+1}$ by removing $v_{r+1}$ from $F_{n+1}$, and $E_i=F_i$, for all $i\neq n, n+1$. 
 Then $E_n=F_n+\lr{v}\subset E_{n+1}$. Note that $E$ is an $n$-step isotropic flag since 
$$E_n^\perp=F_n^\perp\cap\lr{v}^\perp= F_{n+1}\cap\lr{v_k,v\mid k\in[1,N]\setminus\{r,r+1\}}=\lr{v_k,v\mid k\in[1,\widetilde{\al}_{n+1}]\setminus\{r+1\}}=E_{n+1}.$$
Let $v'=c_1v_{\tla_{n}+1}+\cdots+c_mv_{\tla_n+m}+v_{r+1}$ and define $E'\in\sF$ by setting 
$$E'_n=F_n+\lr{v'}, \;E'_{n+1}=\lr{v_k, v'\mid k\in[1,\widetilde\al_{n+1}]\setminus\{r\}},\;\text{ and }\;E'_i=F_i, \;\forall i\neq n.$$ Then one checks easily that $E'$ is $n$-step isotropic and $E'\in\sE'$. It is clear that the map $E\mapsto E'$ is bijective, proving Claim 2.

By Claim 1, $ \phi_{\czero} * \phi_{\paep{A}}$ is a linear combination of $\phi_{\paep{\apu}}$ if $h<n$, or a linear combination of $\phi_{\paep{\apu}}$, $\phi_{\maep{\apu}}$ and $\phi_{_n {\dot A}_{\overline{n+1}}}$ if $h=n$. By Claim 2, the coefficient $g'_{n,A,n+1}$ is halved. The argument for $\phi_{\czero} * \phi_{\mamep{A}}$ is similar. The proof is complete.
\end{proof}

\subsection{The case $C\in\bbb$ (\!$\implies h<n$).} In this case, ${\mathcal{O}}_C=\check{\mathcal{O}}(\pap{C})\sqcup\check{\mathcal{O}}(\mam{C}).$ Enough to consider those $A$ with $A\in\bbb\sqcup\bbh.$ This case is entirely similar to Theorem \ref{th6.3}. We omit the proof.

\begin{thm} Maintain the same assumptions on $A,C,h,\ga,\la,\mu,\nu,$ and $\ep=\text{sgn }(\urcm)$ as in Theorem~\ref{thmchhh}, and assume $C=E_{h+1,h}^{\theta}+\check \gamma\in\bbb$. Then $h<n$, and, for $A\in\bbh\sqcup\bbb$, the following multiplication formulas hold in   $S_\bsq^\scd(n,r)$:
\begin{equation}\begin{aligned}\nonumber
    \phi_{\pap{C}} * \phi_{\paep{A}}&=\sum\limits_{1\leq p\leq N}\gca \phi_{^{+}\!{(\apu)}^{\ep}}, \qquad
    \phi_{\mam{C}}*\phi_{\mamep{A}}=
    \sum\limits_{1\leq p\leq N}\gca \phi_{^{-}\!{(\apu)}^{-\ep}}.\\
    % \phi_{\pap{C}} * \phi_{\maep{A}}&= \phi_{\mam{C}} * \phi_{\paep{A}}=0.
    \end{aligned}\end{equation}
\end{thm}

%\begin{proof}
 %   By the hypothesis, we have $(\text{rw}(\pap{C}), \text{cw}(\pap{C}))=(\la^+,\mu^+)$ and $(\text{rw}(\mam{C}), \text{cw}(\mam{C}))=(\la^-,\mu^-)$.
%    If $A\in\bbh\sqcup\bbb$, then $\mathcal O_{A}=\check{\mathcal O}(\paep{A})\sqcup \check{\mathcal O}(\mamep{A})$ with
%    $$(\text{rw}(\paep{A}), \text{cw}(\paep{A}))=\begin{cases}       (\la^+, \nu^\ep), &\text{ if }A\in\bbb,\\
  %      (\la^+, \nu^\bullet), &\text{ if }A\in\bbh;
 %   \end{cases}\qquad
%    (\text{rw}(\mamep{A}), \text{cw}(\mamep{A}))=\begin{cases}
 %       (\la^-, \nu^{-\ep}), &\text{ if }A\in\bbb,\\
%        (\la^-, \nu^\bullet), &\text{ if }A\in\bbh.
  %  \end{cases}$$

%By \eqref{oapu}(2), \eqref{ecea} has the following decomposition in  $S_\bsq^\scd(n,r)$,
 %   \begin{equation}\label{cbbb1}
  %  (\phi_{\pap{C}}+\phi_{\mam{C}})*(\phi_{\paep{A}}+\phi_{\mamep{A}})=
 %       \sum\limits_{1\leq p\leq N}\gca (\phi_{\paep{\apu}}+\phi_{\mamep{\apu}}).
%\end{equation}

%Clearly, LHS=$\phi_{\pap{C}}* \phi_{\paep{A}}+\phi_{\mam{C}} * \phi_{\mamep{A}}$ since
% $\phi_{\pap{C}} * \phi_{\maep{A}}= \phi_{\mam{C}} * \phi_{\paep{A}}=0.$
% Moreover, every natural basis element $\phi_{\check{\mathbb A}^\prime}$ on the right hand side satisfies rw$(\check{\mathbb A}^\prime)=\la^+$ or $\la^-$.
%Multiplying idempotents $1_{\la^+}(\text{resp., }1_{\la^-})$ on both sides of \eqref{cbbb1} on the left, we could easily find the formula for $\phi_{\pap{C}} * \phi_{\paep{A}}(\text{resp., }\phi_{\mam{C}}*\phi_{\mamep{A}})$.
%\end{proof}

\subsection{The case $C\in\hbb$ (\!$\implies h=n$).} We have in this case ${\mathcal{O}}_C=\check{\mathcal{O}}(\pap{C})\sqcup\check{\mathcal{O}}(\mam{C})$ with  $(\text{rw}(\pap{C}), \text{cw}(\pap{C}))=(\la^\bullet,\mu^+)$ and $(\text{rw}(\mam{C}), \text{cw}(\mam{C}))=(\la^\bullet,\mu^-)$. Thus, enough to consider those $A$ with $A\in\bbb\sqcup\bbh.$

\begin{thm}\label{thchbb} Let $A,C\in\Xi_{N,2r}$ and $\ep=\text{\rm sgn }(\urcm)$, and assume $\co(C)=\ro(A)$, $C\in\hbb$, and $C-E^\th_{n+1,n}$ is diagonal.  Then, for $A\in\bbb\sqcup\bbh$, $\mathcal O_{A}=\check{\mathcal O}(\paep{A})\sqcup \check{\mathcal O}(\mamep{A})$ and the following multiplication formulas hold in $S_\bsq^\scd(n,r)$:
%Maintain the same assumptions on $A,C,h,\ga,\la,\mu, \nu,$ and $\ep=\text{sgn }(\urcm)$ as in Theorem~\ref{thmchhh}, and assume $C=E_{h+1,h}^{\theta}+\check \gamma\in\hbb$. Then, for $A\in\bbb\sqcup\bbh$, $\mathcal O_{A}=\check{\mathcal O}(\paep{A})\sqcup \check{\mathcal O}(\mamep{A})$ and the following multiplication formulas hold in   $S_\bsq^\scd(n,r)$:
%\begin{itemize}
%\item[(1)] If $A\in\bbb,$ then \begin{equation}\begin{aligned}\nonumber
%    \phi_{\pap{C}} * \phi_{\paep{A}}&=
%    \sum\limits_{1\leq p<n+1}\gca \phi_{^{+}\!{(\apu)}^{\ep}}+\sum\limits_{p>n+1}\gca \phi_{^{-}\!{(\apu)}^{\ep}},\\
%    \phi_{\mam{C}}*\phi_{\mamep{A}}&=
%    \sum\limits_{1\leq p<n+1}\gca \phi_{^{-}\!{(\apu)}^{-\ep}}+\sum\limits_{p>n+1}\gca \phi_{^{+}\!{(\apu)}^{-\ep}}.\\
    % \phi_{\pap{C}} * \phi_{\maep{A}}&= \phi_{\mam{C}} * \phi_{\paep{A}}=0.
%    \end{aligned}\end{equation}
 %   \item[(2)] 
  %  If $A\in\bbb\sqcup\bbh$,  then
  \begin{equation}\begin{aligned}\nonumber
    \phi_{\pap{C}} * \phi_{\paep{A}}&=
    \sum\limits_{1\leq p<n+1}\gca \phi_{^{+}\!{(\apu)}^{\ep}}+\frac{1}{2}{g'_{n,A,\overline{n+1}}}\phi_{_n {\dot A}_{\overline{n+1}}}+\hspace{-3mm}\sum\limits_{p>n+1}\gca \phi_{^{-}\!{(\apu)}^{\ep}},\quad \\
    \phi_{\mam{C}}*\phi_{\mamep{A}}&=
    \sum\limits_{1\leq p<n+1}\gca \phi_{^{-}\!{(\apu)}^{-\ep}}+\frac{1}{2}{g_{n,A, {\overline{n+1}}}}\phi_{_n {\dot A}_{\overline{n+1}}}+\hspace{-3mm}\sum\limits_{p>n+1}\gca \phi_{^{+}\!{(\apu)}^{-\ep}}.
    \end{aligned}\end{equation}
%\end{itemize}
\end{thm}

\begin{proof}
 If $A\in\bbb$, then $(\text{rw}(\paep{A}), \text{cw}(\paep{A}))=(\mu^+, \nu^\ep)$, $(\text{rw}(\mamep{A}), \text{cw}(\mamep{A}))=(\mu^-, \nu^{-\ep})$, and the central (or $(n+1)$th) component of $\co(A)$ is 0. Thus, the $(n,n+1)$ entry  of $_n { A}_{\overline{n+1}}$ is negative. Hence, $\phi_{_n {\dot A}_{\overline{n+1}}}$ does not occur in this case and so we will omit the $p=n+1$ term.

By \eqref{oapu}(2), \eqref{ecea} has the following decomposition in  $S_\bsq^\scd(n,r)$,
    \begin{equation}\label{chbb1}
    (\phi_{\pap{C}}+\phi_{\mam{C}})*(\phi_{\paep{A}}+\phi_{\mamep{A}})=
        \sum\limits_{1\leq p<n+1}\gca (\phi_{\paep{\apu}}+\phi_{\mamep{\apu}})
      +\hspace{-3mm}\sum\limits_{p>n+1}\gca(\phi_{\pamep{\apu}}+\phi_{\maep{\apu}}).
\end{equation}
Clearly, LHS=$\phi_{\pap{C}}* \phi_{\paep{A}}+\phi_{\mam{C}} * \phi_{\mamep{A}}$ since
 $\phi_{\pap{C}} * \phi_{\maep{A}}= \phi_{\mam{C}} * \phi_{\paep{A}}=0.$
 Moreover, every natural basis element $\phi_{\check{\mathbb A}^\prime}$ on the right hand side satisfies cw$(\check{\mathbb A}^\prime)=\nu^\ep$ or $\nu^{-\ep}$.
Multiplying idempotents $1_{\nu^\ep}(\text{resp., }1_{\nu^{-\ep}})$ to the right hand sides of \eqref{chbb1} yields the required formulas for $\phi_{\pap{C}} * \phi_{\paep{A}}(\text{resp., }\phi_{\mam{C}}*\phi_{\mamep{A}})$.

 If $A\in\bbh$, then $(\text{rw}(\paep{A}), \text{cw}(\paep{A}))=(\mu^+, \nu^\bullet)$ and $(\text{rw}(\mamep{A}), \text{cw}(\mamep{A}))=(\mu^-, \nu^{\bullet}).$
By \eqref{oapu}(2), \eqref{ecea} has the following decomposition in  $S_\bsq^\scd(n,r)$,
    \begin{equation}\begin{aligned}\label{chbb2}
    (\phi_{\pap{C}}+\phi_{\mam{C}})*(\phi_{\paep{A}}+\phi_{\mamep{A}})=&
        \sum\limits_{1\leq p<n+1}\gca (\phi_{\paep{\apu}}+\phi_{\mamep{\apu}})\\
      &+{g'_{n, A,{\overline{n+1}}}}\phi_{_n {\dot A}_{\overline{n+1}}}+\hspace{-3mm}\sum\limits_{p>n+1}\gca(\phi_{\pamep{\apu}}+\phi_{\maep{\apu}}),
\end{aligned}\end{equation}
where the central entry of $\apu$ is $0$ except $_n { A}_{\overline{n+1}}$ which is $2.$ 

Again, LHS=$\phi_{\pap{C}}* \phi_{\paep{A}}+\phi_{\mam{C}} * \phi_{\mamep{A}}$ since
 $\phi_{\pap{C}} * \phi_{\maep{A}}= \phi_{\mam{C}} * \phi_{\paep{A}}=0.$ We cannot only use weights to separate the rest two items. Instead, by Claim 1 and Claim 2 in the proof of Theorem~\ref{thmchhh}(2)\&(3),  $\phi_{\pap{C}} * \phi_{\paep{A}}$ is a linear combination of $\phi_{\check{\mathbb A}^\prime}$ with 
\begin{equation}\begin{aligned}\nonumber
    {\check{\mathbb A}^\prime}\in \{{\paep{\apl}} \mid p\in[1,n] \}\cup \{{_n {\dot A}_{\overline{n+1}}}\}\cup\{{\maep{\apl}} \mid p>n+1 \}, 
\end{aligned}\end{equation} which gives the first equation in (2).
The case for $\phi_{\mam{C}}*\phi_{\mamep{A}}$ can be similarly obtained.
\end{proof}

\begin{rem}\label{error2}
    In \cite[Prop.~4.3.2]{FL}, Z. Fan and the second-named author gave three sets of multiplication formulas in $S^\scd_\bsq(n,r)$ which do not cover all the cases discussed in \S\S6--7. For example, the case when $A\in\hbh$ is not mentioned at all. Even for those presented there, the $h=n$ subcase is also missing.
\end{rem}

\end{document}